\documentclass[12pt,reqno,twoside]{amsart}%
\usepackage{amsmath,amsthm,amscd,amsthm,upref,indentfirst}
\usepackage{amsfonts,mathrsfs}
\usepackage{amssymb,amsbsy,bm}
\usepackage{graphicx}
\usepackage{amsaddr}
\usepackage{soul}

\usepackage[us,long,24hr]{datetime}

\usepackage{mathabx}
\usepackage[usenames,dvipsnames]{color}

\theoremstyle{plain}
\newtheorem{Theorem}{Theorem}

\newtheorem{Lemma}[Theorem]{Lemma}
\newtheorem{Proposition}[Theorem]{Proposition}

\theoremstyle{definition}
\newtheorem{Definition}[Theorem]{Definition}

\newtheorem{Corollary}[Theorem]{Corollary}

\theoremstyle{remark}
\newtheorem{Remark}[Theorem]{Remark}

\newtheorem{Example}[Theorem]{Example}
\numberwithin{equation}{section}
\numberwithin{Theorem}{section}

\normalfont\upshape
\usepackage{exscale}
\usepackage[scaled=0.86]{helvet}

\usepackage[scr=euler,scrscaled=1.1,bb=txof,bbscaled=1.16,cal=boondox,calscaled=1.1]{mathalfa}
\renewcommand{\mathit}{\bm}
\renewcommand{\mathfrak}{\mathscr}
\renewcommand{\mathtt}[1]{\scalebox{1.2}{\bf \texttt{\upshape#1}}}
\renewcommand{\emph}[1]{\textcolor{blue}{\textbf{#1}}}

\usepackage{float}

\renewcommand{\hl}[1]{#1}

\usepackage[justification=centering]{caption}

\usepackage[normalem]{ulem}
\usepackage[square,comma]{natbib}

\usepackage{mparhack}

\usepackage{keyval}
\usepackage{totcount}
\newtotcounter{citnum} 
\def\oldbibitem{} \let\oldbibitem=\bibitem
\def\bibitem{\stepcounter{citnum}\oldbibitem}

\usepackage[verbose]{backref}
\backrefsetup{verbose=false}
\renewcommand*{\backref}[1]{}
\renewcommand*{\backrefalt}[4]{[{\tiny%
    \ifcase #1 \textsl{Not cited}%
          \or \textsl{Cited on page}~\textcolor{BrickRed}{#2}%
          \else \textsl{Cited on pages}~\textcolor{BrickRed}{#2}%
    \fi%
    }]}

\usepackage{epigraph}

\usepackage{fancyhdr}
\author{\small\scshape S\lowercase{teven} D\lowercase{uplij}}

\address{
Center for Information Technology (CIT),
Universit\"at M\"unster,
R\"ontgenstrasse 7-13\\
D-48149 M\"unster,
Germany}
\email{\small \sf douplii@uni-muenster.de;
sduplij@gmail.com;
https://ivv5hpp.uni-muenster.de/u/douplii}

\title{\large\bfseries\scshape
H\lowercase{yperpolyadic structures}}

\date{\textit{of start} September 23, 2023. \textit{Date}:
\textit{of completion}
December 3, 2023.
\newline
\mbox{}\hskip 1.16em
\textit{Total}:
60
references
}

\renewcommand{\refname}{\textsc{References}}

\let\origsection\section
\renewcommand{\section}[1]{\sectionmark{#1}\origsection{#1}}
\let\origsubsection\subsection
\renewcommand{\subsection}[1]{\subsectionmark{#1}\origsubsection{#1}}

\makeatletter
\renewenvironment{thebibliography}[1]{%
  \@xp\origsection\@xp*\@xp{\refname}%
  \normalfont\footnotesize\labelsep .9em\relax
  \renewcommand\theenumiv{\arabic{enumiv}}\let\p@enumiv\@empty
  \vspace*{-5pt}
  \list{\@biblabel{\theenumiv}}{\settowidth\labelwidth{\@biblabel{#1}}%
    \leftmargin\labelwidth \advance\leftmargin\labelsep
    \usecounter{enumiv}}%
  \sloppy \clubpenalty\@M \widowpenalty\clubpenalty
  \sfcode`\.=\@m
}{%
  \def\@noitemerr{\@latex@warning{Empty `thebibliography' environment}}%
  \endlist
}
\makeatother

\subjclass[2010]{11R52, 15A03, 15A72, 17A35, 17A40, 17A42, 20N10, 20N15}
\keywords{hypercomplex algebra, ternary algebra, $n$-ary algebra, querelement, quaternion, octonion, Cayley-Dickson construction, division algebra, structure constant, vector space, vectorization, vector multiplication}

\bibdata{lit-all.bib,hyperpol.bib}
\bibstyle{chicago1cc}

\begin{document}

\mbox{}
\vskip 2.5cm
\begin{abstract}

\noindent 
We introduce a new class of division algebras, the hyperpolyadic
algebras, which correspond to the binary division algebras $\mathbb{R}$,
$\mathbb{C}$, $\mathbb{H}$, $\mathbb{O}$ without considering new elements.
First, we use the  matrix polyadization procedure proposed earlier which
increases the dimension of the algebra. The algebras obtained in this way obey binary
addition and a nonderived $n$-ary multiplication and their subalgebras are
division $n$-ary algebras. For each invertible element, we define a new norm
which is polyadically multiplicative, and the corresponding map is a $n$-ary
homomorphism. We define a polyadic analog of the Cayley--Dickson construction
which corresponds to the consequent embedding of monomial matrices from the
polyadization procedure. We then obtain another series of $n$-ary algebras
corresponding to the binary division algebras which have a higher dimension, which
is proportional to the intermediate arities, and which are not isomorphic to those
obtained by the previous constructions. Second, a new polyadic product of
vectors in any vector space is defined, which is consistent with the
polyadization procedure using vectorization. Endowed with this introduced
product, the vector space becomes a polyadic algebra which is a division
algebra under some invertibility conditions, and its structure constants are
computed. Third, we propose a new iterative process (we call it the
\textquotedblleft imaginary tower\textquotedblright), which leads to nonunital
nonderived ternary division algebras of half the dimension, which we call
\textquotedblleft half-quaternions\textquotedblright\ and \textquotedblleft
half-octonions\textquotedblright. The latter are not the subalgebras of the binary
division algebras, but subsets only, since they have different arity.
Nevertheless, they are actually ternary division algebras, because they allow
division, and their nonzero elements are invertible. From the multiplicativity
of the introduced \textquotedblleft half-quaternion\textquotedblright\ norm, we
obtain the ternary analog of the sum of two squares identity. We show that
the ternary division algebra of imaginary
\textquotedblleft half-octonions\textquotedblright\ is unitless and 
totally associative.

\end{abstract}
\maketitle

\thispagestyle{empty}


\newpage
\thispagestyle{empty}
\mbox{}
\vspace{1cm}
\tableofcontents
\vspace{-2.6cm}

\pagestyle{fancy}

\addtolength{\footskip}{15pt}

\renewcommand{\sectionmark}[1]{%
\markboth{
{ \scshape #1}}{}}

\renewcommand{\subsectionmark}[1]{%
\markright{
\mbox{\;}\\[5pt]
\textmd{#1}}{}}

\fancyhead{}
\fancyhead[EL,OR]{\leftmark}
\fancyhead[ER,OL]{\rightmark}
\fancyfoot[C]{\scshape -- \textcolor{BrickRed}{\thepage} --}

\renewcommand\headrulewidth{0.5pt}
\fancypagestyle {plain1}{ %
\fancyhf{}
\renewcommand {\headrulewidth }{0pt}
\renewcommand {\footrulewidth }{0pt}
}

\fancypagestyle{plain}{ %
\fancyhf{}
\fancyhead[C]{\scshape S\lowercase{teven} D\lowercase{uplij} \hskip 0.7cm \MakeUppercase{Polyadic Hopf algebras and quantum groups}}
\fancyfoot[C]{\scshape - \thepage  -}
\renewcommand {\headrulewidth }{0pt}
\renewcommand {\footrulewidth }{0pt}
}

\fancypagestyle{fancyref}{ %
\fancyhf{} 
\fancyhead[C]{\scshape R\lowercase{eferences} }
\fancyfoot[C]{\scshape -- \textcolor{BrickRed}{\thepage} --}
\renewcommand {\headrulewidth }{0.5pt}
\renewcommand {\footrulewidth }{0pt}
}

\fancypagestyle{emptyf}{
\fancyhead{}
\fancyfoot[C]{\scshape -- \textcolor{BrickRed}{\thepage} --}
\renewcommand{\headrulewidth}{0pt}
}
\mbox{}
\vskip 2.5cm
\thispagestyle{emptyf}

\section{\textsc{Introduction}}

The field extension is a fundamental concept of algebra
\cite{mccoy,rotman,lovett} and number theory \cite{bor/sch,neurkich,samuel}.
Informally, the main idea is to enlarge a given structure in a special way
(using elements from outside the underlying set) and try to obtain a resulting
algebraic structure with \textquotedblleft good\textquotedblright\ properties.
One of the first well-known examples is the field of complex numbers
$\mathbb{C}$, which is a simple field extension of real numbers $\mathbb{R}$.
The direct generalization of this construction leads to the hypercomplex
numbers (see, e.g., \cite{wed1908,kan/sol}) defined as finite $D$-dimensional
algebras $\mathbb{A}$ over the reals with a special basis (with squares
restricted to $0,\pm1$). Among numerous versions of hypercomplex number
systems \cite{haw1902,tab1904} (for a modern review, see, e.g.,
\cite{yaglom,bur/bur}), only the complex numbers $\mathbb{A}=\mathbb{C}$
($D=2$), quaternions $\mathbb{A}=\mathbb{H}$ ($D=4$), and octonions
$\mathbb{A}=\mathbb{O}$ ($D=8$) are classical division algebras (with no zero
divisors or nilpotents) \cite{schafer,sal92,gubareni}, and the latter two  can
be obtained via the Cayley--Dickson doubling procedure
\cite{dic19,alb42,scha54}.

In this paper, we construct nonderived hyperpolyadic structures corresponding
to the above division algebras without introducing new elements (as in, e.g.,
\cite{dub/vol,lip/rau/vol}). Recall the numerous applications of higher arity
structures in physics \cite{ker2000,ker2012,cast1}, in particular in particle
dynamical models \cite{nam0,tak3,ata/mak/sil} and supersymmetry
\cite{abr/ker/roy,bar/gun1,bag/lam1}. There are plenty of $n$-ary
generalizations of associative and Lie algebras, for examples and reviews, see
\cite{car4,fil0,mic/vin,azc/izq}.

Here, we will first use the matrix polyadization procedure proposed by the
author in \cite{duplij2022}. We show that the polyadic analog of the
Cayley--Dickson construction can only lead to non-division algebras of higher
dimensions than the initial division algebras. For the $n$-ary algebras thus
obtained, we introduce a new norm which is polyadically multiplicative and is
well defined for invertible elements.

Second, we propose a new polyadic product of vectors in any vector space,
which is consistent with the polyadization procedure using vectorization.
Endowed with this product, the vector space becomes a polyadic algebra, which is
a division algebra under some invertibility conditions. Its structure
constants are computed, and a numerical example is given.

Third, on the subsets of the binary division algebras, we propose another new
construction (we call it the \textquotedblleft imaginary
tower\textquotedblright) and an iterative process which naturally gives the
corresponding nonderived ternary division algebras of half a dimension. The
latter are not subalgebras, because they have a different multiplication and
different arity than the initial algebras. We call the nonunital ternary
algebras obtained in this way \textquotedblleft
half-quaternions\textquotedblright\ and \textquotedblleft
half-octonions\textquotedblright. They are actually division algebras, because
nonzero elements are invertible and thus allow division. From the
multiplicativity of the \textquotedblleft half-quaternion\textquotedblright%
\ norm, we obtain the ternary analog of the sum of two squares identity.
Finally, we show that the unitless ternary division algebra of imaginary
\textquotedblleft half-octonions\textquotedblright\ is ternary totally associative.

\section{\textsc{Preliminaries}}

Here, we briefly remind the reader of notation from \cite{duplij2022}. A
(one-set) polyadic algebraic structure $\mathcal{A}$ is a set $\mathsf{A}$
closed with respect to polyadic operations (or $n$-ary multiplication)
$\mu^{\left[  n\right]  }:\mathsf{A}^{n}\rightarrow\mathsf{A}$ ($n$-ary
magma). We denote polyads \cite{pos} as $\vec{a}=\vec{a}^{\left(  k\right)
}=\left(  a_{1},\ldots,a_{k}\right)  $, $a_{i}\in\mathsf{A}$, and
$a^{k}=\overset{k}{\left(  \overbrace{a,\ldots,a}\right)  }$, $a\in\mathsf{A}$
(usually, the value of $k$ follows from the context). A (positive) polyadic
power \hl{is} 
\begin{equation}
a^{\left\langle \ell_{\mu}\right\rangle }=\left(  \mu^{\left[  n\right]
}\right)  ^{\circ\ell_{\mu}}\left[  a^{\ell_{\mu}\left(  n-1\right)
+1}\right]  ,\ \ \ \ a\in\mathsf{A},\ \ell_{\mu}\in\mathbb{N}. \label{al}%
\end{equation}
\hl{A} 
 polyadic $\ell_{\mu}$-idempotent (or idempotent for $\ell_{\mu}=1$) is
defined by $a^{\left\langle \ell_{\mu}\right\rangle }=a$. A polyadic zero is
defined by $\mu^{\left[  n\right]  }\left[  \vec{a},z\right]  =z$,
$z\in\mathsf{A}$, $\vec{a}\in\mathsf{A}^{n-1}$, where $z$ can be on any place.
An element of a polyadic algebraic structure $a$ is called $\ell_{\mu}%
$-nilpotent (or nilpotent for $\ell_{\mu}=1$), if there exist $\ell_{\mu}$
such that $a^{\left\langle \ell_{\mu}\right\rangle }=z$. A polyadic (or
$n$-ary) identity (or neutral element) is defined by%
\begin{equation}
\mu^{\left[  n\right]  }\left[  a,e^{n-1}\right]  =a,\ \ \ \ \forall
a\in\mathsf{A}, \label{mn}%
\end{equation}
where $a$ can be on any place on the l.h.s. of (\ref{mn}). In addition, there
exist neutral polyads (usually not unique) satisfying%
\begin{equation}
\mu^{\left[  n\right]  }\left[  a,\vec{n}\right]  =a,\ \ \ \ \forall
a\in\mathsf{A}. \label{man}%
\end{equation}

A one-set polyadic algebraic structure $\left\langle \mathsf{A}\mid
\mu^{\left[  n\right]  }\right\rangle $ is totally associative, if
\begin{equation}
\left(  \mu^{\left[  n\right]  }\right)  ^{\circ2}\left[  \vec{a},\vec{b}%
,\vec{c}\right]  =\mu^{\left[  n\right]  }\left[  \vec{a},\mu^{\left[
n\right]  }\left[  \vec{b}\right]  ,\vec{c}\right]  =invariant, \label{ma}%
\end{equation}
with respect to the placement of the internal multiplication on any of the $n$
places, and $\vec{a},\vec{b},\vec{c}$ are polyads of the necessary sizes.

A polyadic semigroup $\mathcal{S}^{\left[  n\right]  }$ is a one-set and
one-operation structure in which $\mu^{\left[  n\right]  }$ is totally
associative. A polyadic structure is (totally) commutative, if $\mu^{\left[
n\right]  }=\mu^{\left[  n\right]  }\circ\sigma$, for all $\sigma\in S_{n}$. A
polyadic structure is solvable, if for all polyads $b$, $c$ and an element
$x$, one can (uniquely) resolve the Equation (with respect to $x$) for
$\mu^{\left[  n\right]  }\left[  \vec{b},x,\vec{c}\right]  =a$, where $x$ can
be on any place, and $\vec{b},\vec{c}$ are polyads of the needed lengths. A
solvable polyadic structure is called $n$-ary quasigroup \cite{belousov}. An
associative polyadic quasigroup is called a $n$-ary (or polyadic) group
$\mathcal{G}^{\left[  n\right]  }$ (for a review, see, e.g., \cite{galmak1}).
In an $n$-ary group, the only solution of
\begin{equation}
\mu^{\left[  n\right]  }\left[  a^{n-1},\tilde{a}\right]
=a,\ \ \ \ \ a,\tilde{a}\in\mathsf{A}, \label{mgg}%
\end{equation}
is called the querelement (the polyadic analog of an inverse) of $a$ and
denoted by $\tilde{a}$ \cite{dor3}, where $\tilde{a}$ can be on any place. The
relation (\ref{mgg}) can be considered as a definition of the unary
queroperation $\bar{\mu}^{\left(  1\right)  }\left[  a\right]  =\tilde{a}$
\cite{gle/gla}.

For further details and references, see \cite{duplij2022}.

\section{\textsc{Matrix Polyadization\label{sec-matr}}}

{Let us briefly (just to establish notation and terminology) recall that the
$2-,4-,8$-dimensional } algebras are the only hypercomplex extensions of the reals
$\mathbb{A}=\mathbb{R}$ (Hurwitz's theorem for composition algebras)
$\mathbb{A}=\mathbb{D}=\mathbb{C},\mathbb{H},\mathbb{O}$ which are normed
division algebras. The first two are associative (and can be represented by
matrices), and only $\mathbb{C}$ is commutative (being a field). We use the
unified {notation}
 $\mathit{z}\in\mathbb{D}$, and if we need to distinguish and
concretize, the standard parametrization will be exploited $\mathbb{C\ni
}\mathit{z}=\mathit{z}_{\left(  2\right)  }=a+b\mathit{i}$ and $\mathbb{H\ni
}\mathit{z}=\mathit{z}_{\left(  4\right)  }=a+b\mathit{i}+c\mathit{j}%
+d\mathit{k}$, etc., $a,d,c,d\in\mathbb{R}$. The standard Euclidean norm
($2$-norm) $\left\Vert \mathit{z}\right\Vert =\sqrt{\mathit{z}^{\ast
}\mathit{z}}$ ({where}
 the conjugate is $\mathit{z}_{\left(  2\right)  }^{\ast
}=a-b\mathit{i}$, etc.), for $\mathit{z}\in\mathbb{D}$ (which for $\mathbb{C}$
coincides with the modulus $\left\Vert \mathit{z}_{\left(  2\right)
}\right\Vert =\left\vert \mathit{z}_{\left(  2\right)  }\right\vert
=\sqrt{a^{2}+b^{2}}$, and $\left\Vert \mathit{z}_{\left(  4\right)
}\right\Vert =\sqrt{a^{2}+b^{2}+c^{2}+d^{2}}$, etc.) has the properties%
\begin{align}
\left\Vert \mathit{1}\right\Vert  &  =1,\label{10}\\
\left\Vert \lambda\mathit{z}\right\Vert  &  =\left\vert \lambda\right\vert
\left\Vert \mathit{z}\right\Vert ,\label{11}\\
\left\Vert \mathit{z}^{\prime}+\mathit{z}^{\prime\prime}\right\Vert  &
\leq\left\Vert \mathit{z}^{\prime}\right\Vert +\left\Vert \mathit{z}%
^{\prime\prime}\right\Vert ,\ \ \ \ \ 1,\lambda\in\mathbb{R},\ \ \mathit{1}%
,\mathit{z},\mathit{z}^{\prime},\mathit{z}^{\prime\prime}\in\mathbb{D},
\label{12}%
\end{align}
and is multiplicative
\begin{equation}
\left\Vert \mathit{z}_{1}\mathit{z}_{2}\right\Vert =\left\Vert \mathit{z}%
_{1}\right\Vert \left\Vert \mathit{z}_{2}\right\Vert \in\mathbb{R}%
_{\geqslant0},\ \ \ \ \mathit{z}\in\mathbb{D}, \label{zz}%
\end{equation}
such that the corresponding mapping $\mathbb{D}\rightarrow\mathbb{R}%
_{\geqslant0}$ is a homomorphism. Each nonzero element of the above normed
unital algebra has the multiplicative inverse $\mathit{z}^{-1}\mathit{z}%
=\mathit{1}$, because the norm vanishes only for $\mathit{z}=0$ and there are
no zero divisors, and therefore, from $\left\Vert \mathit{z}\right\Vert
^{2}=\mathit{z}^{\ast}\mathit{z}\in\mathbb{R}_{\geqslant0}$, it follows that%
\begin{equation}
\mathit{z}^{-1}=\frac{\mathit{z}^{\ast}}{\left\Vert \mathit{z}\right\Vert
^{2}},\ \ \ \mathit{z}\in\mathbb{D}\setminus\left\{  0\right\}  . \label{z1}%
\end{equation}

To construct the polyadic analogs of the binary hypercomplex algebras $\mathbb{A}$
and, in particular, of the binary division algebras $\mathbb{D}$ (over
$\mathbb{R}$), we use the polyadization procedure proposed in
\cite{duplij2022} (called there block-matrix polyadization). It is based on
the general structure theorem for polyadic rings (a generalization of the
Wedderburn theorem): any simple $\left(  2,n\right)  $-ring is isomorphic to
the ring of special cyclic shift block-matrices (of the shape (\ref{zn})) over
a division ring \cite{nik84a}.

Let us introduce the $\left(  n-1\right)  \times\left(  n-1\right)  $ cyclic
shift weighted matrix with the elements from the algebra $\mathbb{A}$%
\begin{equation}
\mathit{Z}=\mathit{Z}^{\left[  n\right]  }=\left(
\begin{array}
[c]{ccccc}%
0 & \mathit{z}_{1} & \ldots & 0 & 0\\
0 & 0 & \mathit{z}_{2} & \ldots & 0\\
0 & 0 & \ddots & \ddots & \vdots\\
\vdots & \vdots & \ddots & 0 & \mathit{z}_{n-2}\\
\mathit{z}_{n-1} & 0 & \ldots & 0 & 0
\end{array}
\right)  ,\ \ \ \mathit{z}_{i}\in\mathbb{A}. \label{zn}%
\end{equation}

The matrices of such a shape, i.e., (\ref{zn}) play a considerable role in
coding \cite{kim/lee2004} and were intensively studied in
\cite{chi/nak,gau/tsai/wang}. Here, we will apply the cyclic shift weighted
matrices to the polyadization procedure introduced in \cite{duplij2022}.

The set of matrices of the form (\ref{zn}) is closed with respect to the
ordinary product $\left(  \cdot\right)  $ of exactly $n$ matrices, but not of
fewer. Therefore, we can define the $n$-ary multiplication \cite{duplij2022}%
{\fontsize{7.8}{7.8}\selectfont\begin{equation}
\mathit{\mu}_{\mathit{Z}}^{\left[  n\right]  }\left[  \overset{n}%
{\overbrace{\mathit{Z}^{\prime},\mathit{Z}^{\prime\prime}\ldots\mathit{Z}%
^{\prime\prime\prime}}}\right]  =\mathit{Z}^{\prime}\cdot\mathit{Z}%
^{\prime\prime}\cdot\ldots\cdot\mathit{Z}^{\prime\prime\prime}=\mathit{Z}%
=\mathit{Z}^{\left[  n\right]  }, \label{mz}%
\end{equation}}
which is nonderived in the sense that the binary (and all $\leq n-1$) products
of $\mathit{Z}^{\left[  n\right]  }$'s are outside the set (\ref{zn}).

\begin{Remark}
\label{rem-add}\hl{The} 
 binary addition in the algebra $\mathbb{A}$ transfers in
the standard way to matrix addition (as component-wise addition), and so we
will mostly pay attention to the multiplicative part, implying that the
addition of $\mathit{Z}$-matrices (\ref{zn}) is always binary.
\end{Remark}

\begin{Definition}
We call the following new algebraic structure%
{\fontsize{7.8}{7.8}\selectfont\begin{equation}
\mathbb{A}^{\left[  n\right]  }=\left\langle \left\{  \mathit{Z}^{\left[
n\right]  }\right\}  \mid(+),\mathit{\mu}_{\mathit{Z}}^{\left[  n\right]
}\right\rangle \label{an}%
\end{equation}}
a $n$-ary hypercomplex algebra $\mathbb{A}^{\left[  n\right]  }$ which
corresponds to the binary hypercomplex algebra $\mathbb{A}$ by the matrix
polyadization procedure.
\end{Definition}

\begin{Proposition}
The dimension $D^{\left[  n\right]  }$ of the $n$-ary hypercomplex algebra
$\mathbb{A}^{\left[  n\right]  }$ is%
\begin{equation}
D^{\left[  n\right]  }=\dim\mathbb{A}^{\left[  n\right]  }=D\left(
n-1\right)  . \label{dn}%
\end{equation}

\end{Proposition}

\begin{proof}
It obviously follows from (\ref{zn}).
\end{proof}

\begin{Proposition}
If the binary hypercomplex algebra $\mathbb{A}$ is associative (for the
dimensions $D=1$, $D=2$ and $D=4$), the $n$-ary multiplication (\ref{mz}) in
components has the cyclic product form \cite{duplij2022}%
\begin{align}
\overset{n}{\overbrace{\mathit{z}_{1}^{\prime}\mathit{z}_{2}^{\prime\prime
}\ldots\mathit{z}_{n-1}^{\prime\prime\prime}\mathit{z}_{1}^{\prime\prime
\prime\prime}}}  &  =\mathit{z}_{1},\nonumber\\
\overset{n}{\overbrace{\mathit{z}_{2}^{\prime}\mathit{z}_{3}^{\prime\prime
}\ldots\mathit{z}_{1}^{\prime\prime\prime}\mathit{z}_{2}^{\prime\prime
\prime\prime}}}  &  =\mathit{z}_{2},\nonumber\\
&  \vdots\nonumber\\
\overset{n}{\overbrace{\mathit{z}_{n-1}^{\prime}\mathit{z}_{1}^{\prime\prime
}\ldots\mathit{z}_{n-2}^{\prime\prime\prime}\mathit{z}_{n-1}^{\prime
\prime\prime\prime}}}  &  =\mathit{z}_{n-1},\ \ \ \ \ \ \mathit{z}%
_{i},\mathit{z}_{i}^{\prime},\ldots,\mathit{z}_{i}^{\prime\prime\prime
},\mathit{z}_{i}^{\prime\prime\prime\prime}\in\mathbb{A}. \label{cy}%
\end{align}

\end{Proposition}

\begin{proof}
This follows from (\ref{zn}) and (\ref{mz}).
\end{proof}

\begin{Remark}
The cycled product (\ref{cy}) can be treated as a $n$-ary extension of the
Jordan pair \cite{loos,loo74}, which is different from \cite{fau95}.
\end{Remark}

\begin{Proposition}
If $\mathbb{A}$ is unital, then $\mathbb{A}^{\left[  n\right]  }$ contains a
$n$-ary unit (polyadic identity (\ref{mn})), being the permutation (cyclic
shift) matrix of the form%
\begin{equation}
\mathit{E}^{\left[  n\right]  }=\left(
\begin{array}
[c]{ccccc}%
0 & \mathit{1} & \ldots & 0 & 0\\
0 & 0 & \mathit{1} & \ldots & 0\\
0 & 0 & \ddots & \ddots & \vdots\\
\vdots & \vdots & \ddots & 0 & \mathit{1}\\
\mathit{1} & 0 & \ldots & 0 & 0
\end{array}
\right)  \in\mathbb{A}^{\left[  n\right]  },\ \ \ \mathit{1}\in\mathbb{A}.
\label{en}%
\end{equation}

\end{Proposition}

\begin{proof}
It follows from (\ref{zn}), (\ref{mz}), and (\ref{cy}).
\end{proof}

Consider now the polyadization of the hypercomplex two-dimensional algebra of
dual numbers.

\begin{Example}
[4-ary dual numbers]\label{exam-dual}The commutative and associative
two-dimensional algebra $\mathbb{A}_{du}=\left\langle \left\{  \mathit{z}%
\right\}  \mid\left(  +\right)  ,\left(  \cdot\right)  \right\rangle $ is
defined by the element $\mathit{z}=a+b\mathit{\varepsilon}$ with
$\mathit{\varepsilon}^{2}=0$, $a,b\in\mathbb{R}$. The binary multiplication
$\mu_{\mathsf{v}}^{\left[  2\right]  }$ of pairs ($2$-tuples) $\mathsf{v}%
_{2}=\left(
\begin{array}
[c]{c}%
a\\
b
\end{array}
\right)  \in\mathbb{A}_{du}^{tu}=\mathbb{A}_{du}^{\left[  2\right]
,tu}=\left\langle \left\{  \mathsf{v}_{2}\right\}  \mid\left(  +\right)
,\mu^{\left[  2\right]  }\right\rangle $ (addition is component-wise, see
\textit{Remark} \ref{rem-add}) is%
\begin{equation}
\mu_{\mathsf{v}}^{\left[  2\right]  }\left[  \left(
\begin{array}
[c]{c}%
a^{\prime}\\
b^{\prime}%
\end{array}
\right)  ,\left(
\begin{array}
[c]{c}%
a^{\prime\prime}\\
b^{\prime\prime}%
\end{array}
\right)  \right]  =\left(
\begin{array}
[c]{c}%
a^{\prime}a^{\prime\prime}\\
a^{\prime}b^{\prime\prime}+b^{\prime}a^{\prime\prime}%
\end{array}
\right)  \in\mathbb{A}_{du}^{tu},\ \ \ \ a^{\prime},a^{\prime\prime}%
,b^{\prime},b^{\prime\prime}\in\mathbb{R}.\label{ab}%
\end{equation}
\hl{It} follows from (\ref{ab}) that $\mathbb{A}_{du}^{tu}$ (and so also
$\mathbb{A}_{du}$) constitutes non-division algebra, because it contains idempotents
and zero divisors (e.g., $\left(
\begin{array}
[c]{c}%
0\\
b
\end{array}
\right)  ^{2}=\left(
\begin{array}
[c]{c}%
0\\
0
\end{array}
\right)  $).

{Using the matrix polyadization procedure, we construct a $4$-ary algebra
$\mathbb{A}_{du}^{\left[  4\right]  }=\left\langle \left\{  \mathit{Z}%
^{\left[  4\right]  }\right\}  \mid\left(  +\right)  ,\mathit{\mu}^{\left[
4\right]  }\right\rangle $ of dimension $D^{\left[  4\right]  }=6$ (see
(\ref{dn})) by introducing the following $3\times3$ cyclic shift weighted
matrix (\ref{zn})}%

\begin{equation}
\mathit{Z}^{\left[  4\right]  }=\left(
\begin{array}
[c]{ccc}%
0 & \mathit{z}_{1} & 0\\
0 & 0 & \mathit{z}_{2}\\
\mathit{z}_{3} & 0 & 0
\end{array}
\right)  =\left(
\begin{array}
[c]{ccc}%
0 & a_{1}+b_{1}\mathit{\varepsilon} & 0\\
0 & 0 & a_{2}+b_{2}\mathit{\varepsilon}\\
a_{3}+b_{3}\mathit{\varepsilon} & 0 & 0
\end{array}
\right)  \in\mathbb{A}_{du}^{\left[  4\right]  },\ \ \ \mathit{z}%
_{1},\mathit{z}_{2},\mathit{z}_{3}\in\mathbb{A}_{du}.\label{z4}%
\end{equation}

\hl{The} cyclic product of the components (\ref{cy}) becomes
\begin{align}
\mathit{z}_{1}^{\prime}\mathit{z}_{2}^{\prime\prime}\mathit{z}_{3}%
^{\prime\prime\prime}\mathit{z}_{1}^{\prime\prime\prime\prime} &
=\mathit{z}_{1},\nonumber\\
\mathit{z}_{2}^{\prime}\mathit{z}_{3}^{\prime\prime}\mathit{z}_{1}%
^{\prime\prime\prime}\mathit{z}_{2}^{\prime\prime\prime\prime} &
=\mathit{z}_{2},\nonumber\\
\mathit{z}_{3}^{\prime}\mathit{z}_{1}^{\prime\prime}\mathit{z}_{2}%
^{\prime\prime\prime}\mathit{z}_{3}^{\prime\prime\prime\prime} &
=\mathit{z}_{3},\ \ \ \ \ \ \mathit{z}_{i},\mathit{z}_{i}^{\prime}%
,\mathit{z}_{i}^{\prime\prime},\mathit{z}_{i}^{\prime\prime\prime}%
,\mathit{z}_{i}^{\prime\prime\prime\prime}\in\mathbb{A}_{dual}.\label{z}%
\end{align}
\hl{In} terms of $6$-tuples $\mathsf{v}_{6}$ over $\mathbb{R}$ (cf. (\ref{ab})), the
$4$-ary multiplication $\mu^{\left[  4\right]  }$ in $\mathbb{A}_{du}^{\left[
4\right]  ,tu}=\left\langle \left\{  \mathsf{v}_{6}\right\}  \mid\left(
+\right)  ,\mu_{\mathsf{v}}^{\left[  4\right]  }\right\rangle $ has the form%
\begin{align}
&  \mu_{\mathsf{v}}^{\left[  4\right]  }\left[  \left(
\begin{array}
[c]{c}%
a_{1}^{\prime}\\
b_{1}^{\prime}\\
a_{2}^{\prime}\\
b_{2}^{\prime}\\
a_{3}^{\prime}\\
b_{3}^{\prime}%
\end{array}
\right)  ,\left(
\begin{array}
[c]{c}%
a_{1}^{\prime\prime}\\
b_{1}^{\prime\prime}\\
a_{2}^{\prime\prime}\\
b_{2}^{\prime\prime}\\
a_{3}^{\prime\prime}\\
b_{3}^{\prime\prime}%
\end{array}
\right)  ,\left(
\begin{array}
[c]{c}%
a_{1}^{\prime\prime\prime}\\
b_{1}^{\prime\prime\prime}\\
a_{2}^{\prime\prime\prime}\\
b_{2}^{\prime\prime\prime}\\
a_{3}^{\prime\prime\prime}\\
b_{3}^{\prime\prime\prime}%
\end{array}
\right)  ,\left(
\begin{array}
[c]{c}%
a_{1}^{\prime\prime\prime\prime}\\
b_{1}^{\prime\prime\prime\prime}\\
a_{2}^{\prime\prime\prime\prime}\\
b_{2}^{\prime\prime\prime\prime}\\
a_{3}^{\prime\prime\prime\prime}\\
b_{3}^{\prime\prime\prime\prime}%
\end{array}
\right)  \right]  \nonumber\\
&  =\left(
\begin{array}
[c]{c}%
a_{1}^{\prime}a_{2}^{\prime\prime}a_{3}^{\prime\prime\prime}a_{1}%
^{\prime\prime\prime\prime}\\
a_{1}^{\prime}a_{2}^{\prime\prime}a_{3}^{\prime\prime\prime}b_{1}%
^{\prime\prime\prime\prime}+a_{1}^{\prime}a_{2}^{\prime\prime}b_{3}%
^{\prime\prime\prime}a_{1}^{\prime\prime\prime\prime}+a_{1}^{\prime}%
b_{2}^{\prime\prime}a_{3}^{\prime\prime\prime}a_{1}^{\prime\prime\prime\prime
}+\allowbreak b_{1}^{\prime}a_{2}^{\prime\prime}a_{3}^{\prime\prime\prime
}a_{1}^{\prime\prime\prime\prime}\\
a_{2}^{\prime}a_{3}^{\prime\prime}a_{1}^{\prime\prime\prime}a_{2}%
^{\prime\prime\prime\prime}\\
a_{2}^{\prime}a_{1}^{\prime\prime\prime}a_{3}^{\prime\prime}b_{2}%
^{\prime\prime\prime\prime}+a_{2}^{\prime}a_{3}^{\prime\prime}b_{1}%
^{\prime\prime\prime}a_{2}^{\prime\prime\prime\prime}+a_{2}^{\prime}%
b_{3}^{\prime\prime}a_{1}^{\prime\prime\prime}a_{2}^{\prime\prime\prime\prime
}+b_{2}^{\prime}a_{3}^{\prime\prime}a_{1}^{\prime\prime\prime}a_{2}%
^{\prime\prime\prime\prime}\\
a_{3}^{\prime}a_{1}^{\prime\prime}a_{2}^{\prime\prime\prime}a_{3}%
^{\prime\prime\prime\prime}\\
a_{3}^{\prime}a_{1}^{\prime\prime}a_{2}^{\prime\prime\prime}b_{3}%
^{\prime\prime\prime\prime}+a_{3}^{\prime}a_{1}^{\prime\prime}b_{2}%
^{\prime\prime\prime}a_{3}^{\prime\prime\prime\prime}+a_{3}^{\prime}%
b_{1}^{\prime\prime}a_{2}^{\prime\prime\prime}a_{3}^{\prime\prime\prime\prime
}+b_{3}^{\prime}a_{1}^{\prime\prime}a_{2}^{\prime\prime\prime}a_{3}%
^{\prime\prime\prime\prime}%
\end{array}
\right)  ,\label{m4}%
\end{align}
which is nonderived and noncommutative due to braidings in the cyclic product
(\ref{z}). The polyadic unit ($4$-ary unit) in the $4$-ary algebra of
$6$-tuples $\mathsf{e}_{6}^{\left[  4\right]  }$ is defined by the equation
(see (\ref{mn}))%
\begin{equation}
\mu_{\mathsf{v}}^{\left[  4\right]  }\left[  \mathsf{v}_{6},\mathsf{e}%
_{6}^{\left[  4\right]  },\mathsf{e}_{6}^{\left[  4\right]  },\mathsf{e}%
_{6}^{\left[  4\right]  }\right]  =\mathsf{v}_{6},\ \ \ \ \ \forall
\mathsf{v}_{6}\in\mathbb{A}_{du}^{\left[  4\right]  ,tu},\label{me4}%
\end{equation}
where $\mathsf{v}_{6}$ can be on any place. Using $4$-ary product (\ref{m4})
and the definition (\ref{me4}), we obtain the manifest form of the polyadic
unit of the $4$-ary algebra of $6$-tuples $\mathbb{A}_{du}^{\left[  4\right]
,tu}$%
\begin{equation}
\mathsf{e}_{6}^{\left[  4\right]  }=\left(
\begin{array}
[c]{c}%
\mathit{1}\\
0\\
\mathit{1}\\
0\\
\mathit{1}\\
0
\end{array}
\right)  \in\mathbb{A}_{du}^{\left[  4\right]  ,tu}.
\end{equation}
\hl{It} follows from (\ref{m4}) that $\mathbb{A}_{du}^{\left[  4\right]  ,tu}$ (and
so also $\mathbb{A}_{dual}^{\left[  4\right]  }$) is a non-division $4$-ary
algebra, and not a field (similarly to the ordinary binary $\mathbb{A}_{du}$),
because it contains $4$-ary idempotents and zero divisors, for instance,%
\begin{equation}
\mu_{\mathsf{v}}^{\left[  4\right]  }\left[  \left(
\begin{array}
[c]{c}%
0\\
b_{1}\\
0\\
b_{2}\\
a_{3}\\
b_{3}%
\end{array}
\right)  ^{4}\right]  =\mu_{\mathsf{v}}^{\left[  4\right]  }\left[  \left(
\begin{array}
[c]{c}%
a_{1}\\
b_{1}\\
0\\
b_{2}\\
0\\
b_{3}%
\end{array}
\right)  ^{4}\right]  =\mu_{\mathsf{v}}^{\left[  4\right]  }\left[  \left(
\begin{array}
[c]{c}%
0\\
b_{1}\\
a_{2}\\
b_{2}\\
0\\
b_{3}%
\end{array}
\right)  ^{4}\right]  =\left(
\begin{array}
[c]{c}%
0\\
0\\
0\\
0\\
0\\
0
\end{array}
\right)  =\mathsf{0}_{\mathsf{v}}.
\end{equation}

Thus, the application of the matrix polyadization procedure to the
commutative, associative, unital, two-dimensional, non-division algebra of
dual numbers $\mathbb{A}_{du}$ gives a noncommutative, $4$-ary nonderived
totally associative, unital, six-dimensional, non-division algebra
$\mathbb{A}_{du}^{\left[  4\right]  }$ over $\mathbb{R}$, which we call the
$4$-ary dual numbers.
\end{Example}

\section{\textsc{Polyadization of Division Algebras\label{sec-polydiv}}}

Let us consider the matrix polyadization procedure for the division algebras
$\mathbb{D}=\mathbb{R},\mathbb{C},\mathbb{H},\mathbb{O}$ in more detail,
paying attention to invertibility and norms.

Recall that the binary division algebra $\mathbb{D}$ (without zero element)
forms a group (with respect to binary multiplication) having the inverse
(\ref{z1}). The polyadic counterpart of the binary inverse is the querelement
(\ref{mgg}).

\begin{Theorem}
The nonderived $n$-ary algebra (\ref{an}) constructed from the binary division
algebra $\mathbb{D}$ is the $n$-ary algebra%
\begin{equation}
\mathbb{D}^{\left[  n\right]  }=\left\langle \left\{  \mathit{Z}^{\left[
n\right]  }\right\}  \mid(+),\mathit{\mu}_{\mathit{Z}}^{\left[  n\right]
},\widetilde{\mathit{Z}}^{\left[  n\right]  }\right\rangle , \label{dz}%
\end{equation}
where $\widetilde{\mathit{Z}}^{\left[  n\right]  }$ is the querelement%
\begin{equation}
\widetilde{\mathit{Z}}=\widetilde{\mathit{Z}}^{\left[  n\right]  }=\left(
\begin{array}
[c]{ccccc}%
0 & \widetilde{\mathit{z}}_{1} & \ldots & 0 & 0\\
0 & 0 & \widetilde{\mathit{z}}_{2} & \ldots & 0\\
0 & 0 & \ddots & \ddots & \vdots\\
\vdots & \vdots & \ddots & 0 & \widetilde{\mathit{z}}_{n-2}\\
\widetilde{\mathit{z}}_{n-1} & 0 & \ldots & 0 & 0
\end{array}
\right)  , \label{zv}%
\end{equation}
which satisfies (provided $\mathbb{D}$ is associative)%
\begin{equation}
\mathit{\mu}_{\mathit{Z}}^{\left[  n\right]  }\left[  \overset{n-1}%
{\overbrace{\mathit{Z},\mathit{Z}\ldots\mathit{Z}}}\widetilde{\mathit{Z}%
}\right]  =\overset{n-1}{\overbrace{\mathit{Z}\cdot\mathit{Z}\cdot\ldots
\cdot\mathit{Z}}}\cdot\widetilde{\mathit{Z}}=\mathit{Z},\ \ \ \ \ \ \forall
\mathit{Z}\in\mathbb{D}^{\left[  n\right]  }, \label{mnz}%
\end{equation}
and $\widetilde{\mathit{Z}}$ can be on any place, such that%
\begin{equation}
\widetilde{\mathit{z}}_{i}=\mathit{z}_{i-1}^{-1}\mathit{z}_{i-2}^{-1}%
\ldots\mathit{z}_{2}^{-1}\mathit{z}_{1}^{-1}\mathit{z}_{n-1}^{-1}%
\mathit{z}_{n-1}^{-1}\ldots\mathit{z}_{i+2}^{-1}\mathit{z}_{i+1}%
^{-1},\ \ \ \ \mathit{z}_{i}\in\mathbb{D\setminus}\left\{  0\right\}  .
\label{zi}%
\end{equation}

\end{Theorem}

\begin{proof}
The main relation (\ref{zi}) follows from (\ref{mnz}) in components (\ref{cy})
as the following cycle products%
\begin{align}
\overset{n}{\overbrace{\mathit{z}_{1}\mathit{z}_{2}\ldots\mathit{z}%
_{n-1}\widetilde{\mathit{z}}_{1}}}  &  =\mathit{z}_{1},\nonumber\\
\overset{n}{\overbrace{\mathit{z}_{2}\mathit{z}_{3}\ldots\mathit{z}%
_{n-1}\mathit{z}_{1}\widetilde{\mathit{z}}_{2}}}  &  =\mathit{z}%
_{2},\nonumber\\
&  \vdots\nonumber\\
\overset{n}{\overbrace{\mathit{z}_{n-1}\mathit{z}_{1}\ldots\mathit{z}%
_{n-2}\widetilde{\mathit{z}}_{n-1}}}  &  =\mathit{z}_{n-1}%
,\ \ \ \ \ \ \mathit{z}_{i},\widetilde{\mathit{z}}_{i}\in\mathbb{D\setminus
}\left\{  0\right\}  , \label{zd}%
\end{align}
are obtained by applying $\mathit{z}_{i}^{-1}$ (which exists in $\mathbb{D}$
for nonzero $\mathit{z}$ (\ref{z1})) from the left $\left(  n-1\right)  $
times (with suitable indices) to both sides of each equation in (\ref{zd}) to
obtain $\widetilde{\mathit{z}}_{i}$.
\end{proof}

\begin{Corollary}
Each $D$-dimensional division algebra over the reals $\mathbb{D}%
=\mathbb{C},\mathbb{H},\mathbb{O}$ (including $\mathbb{R}$ itself as the
one-dimensional case) has as its $n$ polyadic counterparts (where $n$ is
arbitrary) the nonderived $n$-ary non-division algebras $\mathbb{D}^{\left[
n\right]  }$ (\ref{dz}) of dimension $D(n-1)$ (having the polyadic unit
(\ref{en}) and the querelement (\ref{zv}) for invertible $\mathit{z}_{i}%
\in\mathbb{D\setminus}\left\{  0\right\}  $) constructed by the matrix
polyadization procedure.
\end{Corollary}

\begin{Theorem}
\label{theor-poly}The matrix polyadization procedure changes the invertibility
properties of the initial algebra, that is, the polyadization of a binary division
algebra $\mathbb{D}$ leads to a $n$-ary non-division algebra $\mathbb{D}%
^{\left[  n\right]  }$ for arbitrary $n$.
\end{Theorem}

\begin{proof}
The polyadization procedure is provided by monomial matrices which have a
determinant proportional to the product of nonzero entries. The nonzero
elements of $\mathbb{D}^{\left[  n\right]  }$ having some of $\mathit{z}%
_{i}=0$ are noninvertible, and therefore, $\mathbb{D}^{\left[  n\right]  }$ is
not a field.
\end{proof}

Nevertheless, a special subalgebra of $\mathbb{D}^{\left[  n\right]  }$ can be
a division $n$-ary algebra, that is, when $\mathit{Z}$-matrix is a monomial
matrix, having one non-zero entry in each row and each column (see, e.g.,
\cite{joyner}).

\begin{Theorem}
\label{theor-div}The elements of $\mathbb{D}^{\left[  n\right]  }$ which have
all invertible $\mathit{z}_{i}\in\mathbb{D\setminus}\left\{  0\right\}  $ (or
set of invertible $\mathit{Z}$-matrices $\det\mathit{Z}\neq0$) form a
subalgebra $\mathbb{D}_{\operatorname{div}}^{\left[  n\right]  }%
\subset\mathbb{D}^{\left[  n\right]  }$ which is the division $n$-ary algebra
corresponding to the division algebra $\mathbb{D}=\mathbb{R},\mathbb{C}%
,\mathbb{H},\mathbb{O}$.
\end{Theorem}

The simplest case is obtained by the polyadization of the reals.

\begin{Example}
[Five-ary real numbers]\label{exam-real}The $5$-ary associative algebra of real
numbers is $\mathbb{R}^{\left[  5\right]  }=\left\langle \left\{
\mathit{R}^{\left[  5\right]  }\right\}  \mid\left(  +\right)  ,\mathit{\mu
}_{\mathit{R}}^{\left[  5\right]  }\right\rangle $, where $\mathit{R}^{\left[
5\right]  }$ is the cyclic $4\times4$ block-shift matrix (\ref{zn}) with real
entries%
\begin{align}
\mathit{R}^{\left[  5\right]  }  &  =\left(
\begin{array}
[c]{cccc}%
0 & a_{1} & 0 & 0\\
0 & 0 & a_{2} & 0\\
0 & 0 & 0 & a_{3}\\
a_{4} & 0 & 0 & 0
\end{array}
\right)  ,\label{r5}\\
\ \ \ \det\mathit{R}^{\left[  5\right]  }  &  =a_{1}a_{2}a_{3}a_{4}%
,\ \ \ \ \ \ a_{i}\in\mathbb{R}, \label{r5d}%
\end{align}
and the multiplication $\mathit{\mu}_{\mathit{R}}^{\left[  5\right]  }$ is an
ordinary product of five matrices. Only with respect to the product of $5$
elements $\mathit{R}^{\left[  5\right]  }$ is the algebra closed, and
therefore $\mathbb{R}^{\left[  5\right]  }$ is nonderived. In components, we
have the braiding cyclic products (\ref{cy}) for $a_{i}$. If $a_{i}%
\in\mathbb{R}\setminus\left\{  0\right\}  $, the component equations for the
querelement (\ref{zd}) (after the cancellation of nonzero $a_{i}$) become%
\begin{align}
a_{2}a_{3}a_{4}\widetilde{a}_{1}  &  =1,\\
a_{3}a_{4}a_{1}\widetilde{a}_{2}  &  =1,\\
a_{4}a_{1}a_{2}\widetilde{a}_{3}  &  =1,\\
a_{1}a_{2}a_{3}\widetilde{a}_{4}  &  =1.
\end{align}
\hl{Thus}, the querelement for an invertible (\ref{r5}) is%
\begin{equation}
\widetilde{\mathit{R}}^{\left[  5\right]  }=\left(
\begin{array}
[c]{cccc}%
0 & \dfrac{1}{a_{2}a_{3}a_{4}} & 0 & 0\\
0 & 0 & \dfrac{1}{a_{3}a_{4}a_{1}} & 0\\
0 & 0 & 0 & \dfrac{1}{a_{4}a_{1}a_{2}}\\
\dfrac{1}{a_{1}a_{2}a_{3}} & 0 & 0 & 0
\end{array}
\right)  ,\ \ \ a_{i}\in\mathbb{R}\setminus\left\{  0\right\}  ,
\end{equation}
and therefore, the algebra $\mathbb{R}^{\left[  5\right]  }$ of $5$-ary real
numbers is a non-division algebra, because the elements with some $a_{i}=0$
are in $\mathbb{R}^{\left[  5\right]  }$, but they are noninvertible (due to
(\ref{r5d})), and so $\mathbb{R}^{\left[  5\right]  }$ is not a field. But the
subalgebra $\mathbb{R}_{\operatorname{div}}^{\left[  5\right]  }%
\subset\mathbb{R}^{\left[  5\right]  }$ of invertible (monomial) matrices
$\mathit{R}^{\left[  5\right]  }$ ($\det\mathit{R}^{\left[  5\right]  }\neq0$,
with all $a_{i}\neq0$) is a division $n$-ary algebra of reals (see
{Theorem} \ref{theor-div}).
\end{Example}

Now, we provide the example of the $4$-ary algebra of complex numbers, to
compare it with the dual numbers of the same arity in  {Example
}\ref{exam-dual}.

\begin{Example}
[4-ary complex numbers]\label{exam-complex}First, we establish notations, as
before. The commutative and associative two-dimensional algebra of complex
numbers is $\mathbb{C}=\left\langle \left\{  \mathit{z}\right\}  \mid\left(
+\right)  ,\left(  \cdot\right)  \right\rangle $, where $\mathit{z}%
=a+b\mathit{i}$ with $\mathit{i}^{2}=-1$, $a,b\in\mathbb{R}$. The binary
multiplication $\mu_{\mathsf{v}}^{\left[  2\right]  }=\mu_{compl}^{\left[
2\right]  }$ of pairs%
\begin{equation}
\mathsf{v}_{2}=\left(
\begin{array}
[c]{c}%
a\\
b
\end{array}
\right)  \in\mathbb{C}^{tu}=\left\langle \left\{  \mathsf{v}_{2}\right\}
\mid\left(  +\right)  ,\mu_{\mathsf{v}}^{\left[  2\right]  }\right\rangle
\label{v}%
\end{equation}
(where addition is component-wise, see \textit{Remark} \ref{rem-add}) is%
\begin{equation}
\mu_{\mathsf{v}}^{\left[  2\right]  }\left[  \left(
\begin{array}
[c]{c}%
a^{\prime}\\
b^{\prime}%
\end{array}
\right)  ,\left(
\begin{array}
[c]{c}%
a^{\prime\prime}\\
b^{\prime\prime}%
\end{array}
\right)  \right]  =\left(
\begin{array}
[c]{c}%
a^{\prime}a^{\prime\prime}-b^{\prime\prime}b^{\prime}\\
b^{\prime\prime}a^{\prime}+b^{\prime}a^{\prime\prime}%
\end{array}
\right)  \in\mathbb{C}^{tu},\ \ \ \ a^{\prime},a^{\prime\prime},b^{\prime
},b^{\prime\prime}\in\mathbb{R}.\label{m2}%
\end{equation}

\hl{The} algebra of pairs $\mathbb{C}^{tu}$ does not contain idempotents or zero
divisors, its multiplication agrees with one of $\mathbb{C}$ ($\mathit{z}%
^{\prime}\cdot\mathit{z}^{\prime\prime}=\mathit{z}$), and therefore, it is a
division algebra, being isomorphic to $\mathbb{C}$%
\begin{equation}
\mathbb{C}^{tu}\cong\mathbb{C}.\label{cc}%
\end{equation}

Using the matrix polyadization procedure, we construct the nonderived $4$-ary
algebra of complex numbers $\mathbb{C}^{\left[  4\right]  }=\left\langle
\left\{  \mathit{Z}^{\left[  4\right]  }\right\}  \mid\left(  +\right)
,\mathit{\mu}_{\mathit{Z}}^{\left[  4\right]  }\right\rangle $ of the
dimension $D^{\left[  4\right]  }=6$ (see (\ref{dn})) by introducing the
following $3\times3$ matrix (\ref{zn})%
\begin{equation}
\mathit{Z}=\mathit{Z}^{\left[  4\right]  }=\left(
\begin{array}
[c]{ccc}%
0 & \mathit{z}_{1} & 0\\
0 & 0 & \mathit{z}_{2}\\
\mathit{z}_{3} & 0 & 0
\end{array}
\right)  =\left(
\begin{array}
[c]{ccc}%
0 & a_{1}+b_{1}\mathit{i} & 0\\
0 & 0 & a_{2}+b_{2}\mathit{i}\\
a_{3}+b_{3}\mathit{i} & 0 & 0
\end{array}
\right)  \in\mathbb{C}^{\left[  4\right]  },\ \ \ \mathit{z}_{1}%
,\mathit{z}_{2},\mathit{z}_{3}\in\mathbb{C}.\label{z4c}%
\end{equation}
\hl{The} 4-ary product of $\mathit{Z}$-matrices (\ref{z4c}) is%
\vspace{19pt}
\begin{equation}
\mathit{\mu}_{\mathit{Z}}^{\left[  4\right]  }\left[  \mathit{Z}^{\prime
},\mathit{Z}^{\prime\prime},\mathit{Z}^{\prime\prime\prime},\mathit{Z}%
^{\prime\prime\prime\prime}\right]  =\mathit{Z}^{\prime}\mathit{Z}%
^{\prime\prime}\mathit{Z}^{\prime\prime\prime}\mathit{Z}^{\prime\prime
\prime\prime}.\label{m4z}%
\end{equation}
\hl{The} corresponding to (\ref{m4z}) cyclic product in the components (\ref{cy})
becomes%
\begin{align}
\mathit{z}_{1}^{\prime}\mathit{z}_{2}^{\prime\prime}\mathit{z}_{3}%
^{\prime\prime\prime}\mathit{z}_{1}^{\prime\prime\prime\prime} &
=\mathit{z}_{1},\nonumber\\
\mathit{z}_{2}^{\prime}\mathit{z}_{3}^{\prime\prime}\mathit{z}_{1}%
^{\prime\prime\prime}\mathit{z}_{2}^{\prime\prime\prime\prime} &
=\mathit{z}_{2},\nonumber\\
\mathit{z}_{3}^{\prime}\mathit{z}_{1}^{\prime\prime}\mathit{z}_{2}%
^{\prime\prime\prime}\mathit{z}_{3}^{\prime\prime\prime\prime} &
=\mathit{z}_{3},\ \ \ \ \ \ \mathit{z}_{i},\mathit{z}_{i}^{\prime}%
,\mathit{z}_{i}^{\prime\prime},\mathit{z}_{i}^{\prime\prime\prime}%
,\mathit{z}_{i}^{\prime\prime\prime\prime}\in\mathbb{C}.
\end{align}
In terms of six-tuples $\mathsf{v}_{6}$ over $\mathbb{R}$ (cf. (\ref{ab})), the
$4$-ary multiplication $\mu_{\mathsf{v}}^{\left[  4\right]  }$ in
$\mathbb{C}^{\left[  4\right]  ,tu}=\left\langle \left\{  \mathsf{v}%
_{6}\right\}  \mid\left(  +\right)  ,\mu_{\mathsf{v}}^{\left[  4\right]
}\right\rangle $ has the form (cf. (\ref{m2}))
{\footnotesize
\begin{align}
&  \mu_{\mathsf{v}}^{\left[  4\right]  }\left[  \left(
\begin{array}
[c]{c}%
a_{1}^{\prime}\\
b_{1}^{\prime}\\
a_{2}^{\prime}\\
b_{2}^{\prime}\\
a_{3}^{\prime}\\
b_{3}^{\prime}%
\end{array}
\right)  ,\left(
\begin{array}
[c]{c}%
a_{1}^{\prime\prime}\\
b_{1}^{\prime\prime}\\
a_{2}^{\prime\prime}\\
b_{2}^{\prime\prime}\\
a_{3}^{\prime\prime}\\
b_{3}^{\prime\prime}%
\end{array}
\right)  ,\left(
\begin{array}
[c]{c}%
a_{1}^{\prime\prime\prime}\\
b_{1}^{\prime\prime\prime}\\
a_{2}^{\prime\prime\prime}\\
b_{2}^{\prime\prime\prime}\\
a_{3}^{\prime\prime\prime}\\
b_{3}^{\prime\prime\prime}%
\end{array}
\right)  ,\left(
\begin{array}
[c]{c}%
a_{1}^{\prime\prime\prime\prime}\\
b_{1}^{\prime\prime\prime\prime}\\
a_{2}^{\prime\prime\prime\prime}\\
b_{2}^{\prime\prime\prime\prime}\\
a_{3}^{\prime\prime\prime\prime}\\
b_{3}^{\prime\prime\prime\prime}%
\end{array}
\right)  \right]  \nonumber\\
&  =\left(
\begin{array}
[c]{c}%
a_{1}^{\prime}a_{2}^{\prime\prime}a_{3}^{\prime\prime\prime}a_{1}%
^{\prime\prime\prime\prime}-a_{1}^{\prime}b_{2}^{\prime\prime}b_{3}%
^{\prime\prime\prime}a_{1}^{\prime\prime\prime\prime}-b_{1}^{\prime}%
a_{2}^{\prime\prime}b_{3}^{\prime\prime\prime}a_{1}^{\prime\prime\prime\prime
}-b_{1}^{\prime}b_{2}^{\prime\prime}a_{3}^{\prime\prime\prime}a_{1}%
^{\prime\prime\prime\prime}-a_{1}^{\prime}a_{2}^{\prime\prime}b_{3}%
^{\prime\prime\prime}b_{1}^{\prime\prime\prime\prime}-a_{1}^{\prime}%
b_{2}^{\prime\prime}a_{3}^{\prime\prime\prime}b_{1}^{\prime\prime\prime\prime
}-b_{1}^{\prime}a_{2}^{\prime\prime}a_{3}^{\prime\prime\prime}b_{1}%
^{\prime\prime\prime\prime}+b_{1}^{\prime}b_{2}^{\prime\prime}b_{3}%
^{\prime\prime\prime}b_{1}^{\prime\prime\prime\prime}\\
a_{1}^{\prime}a_{2}^{\prime\prime}b_{3}^{\prime\prime\prime}a_{1}%
^{\prime\prime\prime\prime}+a_{1}^{\prime}b_{2}^{\prime\prime}a_{3}%
^{\prime\prime\prime}a_{1}^{\prime\prime\prime\prime}+b_{1}^{\prime}%
a_{2}^{\prime\prime}a_{3}^{\prime\prime\prime}a_{1}^{\prime\prime\prime\prime
}+a_{1}^{\prime}a_{2}^{\prime\prime}a_{3}^{\prime\prime\prime}b_{1}%
^{\prime\prime\prime\prime}-b_{1}^{\prime}b_{2}^{\prime\prime}b_{3}%
^{\prime\prime\prime}a_{1}^{\prime\prime\prime\prime}-a_{1}^{\prime}%
b_{2}^{\prime\prime}b_{3}^{\prime\prime\prime}b_{1}^{\prime\prime\prime\prime
}-b_{1}^{\prime}a_{2}^{\prime\prime}b_{3}^{\prime\prime\prime}b_{1}%
^{\prime\prime\prime\prime}-b_{1}^{\prime}b_{2}^{\prime\prime}a_{3}%
^{\prime\prime\prime}b_{1}^{\prime\prime\prime\prime}\\
a_{2}^{\prime}a_{3}^{\prime\prime}a_{1}^{\prime\prime\prime}a_{2}%
^{\prime\prime\prime\prime}-a_{2}^{\prime}b_{3}^{\prime\prime}b_{1}%
^{\prime\prime\prime}a_{2}^{\prime\prime\prime\prime}-b_{2}^{\prime}%
b_{3}^{\prime\prime}a_{1}^{\prime\prime\prime}a_{2}^{\prime\prime\prime\prime
}-b_{2}^{\prime}a_{3}^{\prime\prime}b_{1}^{\prime\prime\prime}a_{2}%
^{\prime\prime\prime\prime}-a_{2}^{\prime}b_{3}^{\prime\prime}a_{1}%
^{\prime\prime\prime}b_{2}^{\prime\prime\prime\prime}-a_{2}^{\prime}%
a_{3}^{\prime\prime}b_{1}^{\prime\prime\prime}b_{2}^{\prime\prime\prime\prime
}-b_{2}^{\prime}a_{3}^{\prime\prime}a_{1}^{\prime\prime\prime}b_{2}%
^{\prime\prime\prime\prime}+b_{2}^{\prime}b_{3}^{\prime\prime}b_{1}%
^{\prime\prime\prime}b_{2}^{\prime\prime\prime\prime}\\
a_{2}^{\prime}b_{3}^{\prime\prime}a_{1}^{\prime\prime\prime}a_{2}%
^{\prime\prime\prime\prime}+a_{2}^{\prime}a_{3}^{\prime\prime}b_{1}%
^{\prime\prime\prime}a_{2}^{\prime\prime\prime\prime}+b_{2}^{\prime}%
a_{3}^{\prime\prime}a_{1}^{\prime\prime\prime}a_{2}^{\prime\prime\prime\prime
}+a_{2}^{\prime}a_{3}^{\prime\prime}a_{1}^{\prime\prime\prime}b_{2}%
^{\prime\prime\prime\prime}-b_{2}^{\prime}b_{3}^{\prime\prime}b_{1}%
^{\prime\prime\prime}a_{2}^{\prime\prime\prime\prime}-a_{2}^{\prime}%
b_{3}^{\prime\prime}b_{1}^{\prime\prime\prime}b_{2}^{\prime\prime\prime\prime
}-b_{2}^{\prime}b_{3}^{\prime\prime}a_{1}^{\prime\prime\prime}b_{2}%
^{\prime\prime\prime\prime}-b_{2}^{\prime}a_{3}^{\prime\prime}b_{1}%
^{\prime\prime\prime}b_{2}^{\prime\prime\prime\prime}\\
a_{3}^{\prime}a_{1}^{\prime\prime}a_{2}^{\prime\prime\prime}a_{3}%
^{\prime\prime\prime\prime}-b_{3}^{\prime}a_{1}^{\prime\prime}b_{2}%
^{\prime\prime\prime}a_{3}^{\prime\prime\prime\prime}-a_{3}^{\prime}%
b_{1}^{\prime\prime}b_{2}^{\prime\prime\prime}a_{3}^{\prime\prime\prime\prime
}-b_{3}^{\prime}b_{1}^{\prime\prime}a_{2}^{\prime\prime\prime}a_{3}%
^{\prime\prime\prime\prime}-a_{3}^{\prime}a_{1}^{\prime\prime}b_{2}%
^{\prime\prime\prime}b_{3}^{\prime\prime\prime\prime}-b_{3}^{\prime}%
a_{1}^{\prime\prime}a_{2}^{\prime\prime\prime}b_{3}^{\prime\prime\prime\prime
}-a_{3}^{\prime}b_{1}^{\prime\prime}a_{2}^{\prime\prime\prime}b_{3}%
^{\prime\prime\prime\prime}+b_{3}^{\prime}b_{1}^{\prime\prime}b_{2}%
^{\prime\prime\prime}b_{3}^{\prime\prime\prime\prime}\\
a_{3}^{\prime}a_{1}^{\prime\prime}b_{2}^{\prime\prime\prime}a_{3}%
^{\prime\prime\prime\prime}+b_{3}^{\prime}a_{1}^{\prime\prime}a_{2}%
^{\prime\prime\prime}a_{3}^{\prime\prime\prime\prime}+a_{3}^{\prime}%
b_{1}^{\prime\prime}a_{2}^{\prime\prime\prime}a_{3}^{\prime\prime\prime\prime
}+a_{3}^{\prime}a_{1}^{\prime\prime}a_{2}^{\prime\prime\prime}b_{3}%
^{\prime\prime\prime\prime}-b_{3}^{\prime}b_{1}^{\prime\prime}b_{2}%
^{\prime\prime\prime}a_{3}^{\prime\prime\prime\prime}-b_{3}^{\prime}%
a_{1}^{\prime\prime}b_{2}^{\prime\prime\prime}b_{3}^{\prime\prime\prime\prime
}-a_{3}^{\prime}b_{1}^{\prime\prime}b_{2}^{\prime\prime\prime}b_{3}%
^{\prime\prime\prime\prime}-b_{3}^{\prime}b_{1}^{\prime\prime}a_{2}%
^{\prime\prime\prime}b_{3}^{\prime\prime\prime\prime}%
\end{array}
\right)  ,
\end{align}
} which is nonderived and noncommutative due to braidings in the cyclic
product (\ref{z}). The polyadic unit ($4$-ary unit) in the $4$-ary algebra of
$6$-tuples is (see (\ref{mn}))%
\begin{align}
\mathsf{e}_{6} &  =\left(
\begin{array}
[c]{c}%
\mathit{1}\\
0\\
\mathit{1}\\
0\\
\mathit{1}\\
0
\end{array}
\right)  \in\mathbb{C}^{\left[  4\right]  ,tu},\\
\mu_{\mathsf{v}}^{\left[  4\right]  }\left[  \mathsf{e}_{6},\mathsf{e}%
_{6},\mathsf{e}_{6},\mathsf{v}_{6}\right]   &  =\mu_{\mathsf{v}}^{\left[
4\right]  }\left[  \mathsf{e}_{6},\mathsf{e}_{6},\mathsf{v}_{6},\mathsf{e}%
_{6}\right]  =\mu_{\mathsf{v}}^{\left[  4\right]  }\left[  \mathsf{e}%
_{6},\mathsf{v}_{6},\mathsf{e}_{6},\mathsf{e}_{6}\right]  =\mu_{\mathsf{v}%
}^{\left[  4\right]  }\left[  \mathsf{v}_{6},\mathsf{e}_{6},\mathsf{e}%
_{6},\mathsf{e}_{6}\right]  =\mathsf{v}_{6}.
\end{align}

\hl{The} querelement (\ref{mgg}) and (\ref{zv}) for invertible elements of the $4$-ary
algebra of complex numbers $\mathbb{C}^{\left[  4\right]  }$ has the matrix
form which follows from the equations (\ref{zd})%
\begin{equation}
\widetilde{\mathit{Z}}^{\left[  4\right]  }=\left(
\begin{array}
[c]{ccc}%
0 & \frac{1}{\mathit{z}_{2}\mathit{z}_{3}} & 0\\
0 & 0 & \frac{1}{\mathit{z}_{1}\mathit{z}_{3}}\\
\frac{1}{\mathit{z}_{1}\mathit{z}_{2}} & 0 & 0
\end{array}
\right)  \in\mathbb{C}^{\left[  4\right]  },\ \ \ \mathit{z}_{1}%
,\mathit{z}_{2},\mathit{z}_{3}\in\mathbb{C\setminus}\left\{  \mathit{0}%
\right\}  .
\end{equation}

Thus, $\mathbb{C}^{\left[  4\right]  }=\left\langle \left\{  \mathit{Z}%
^{\left[  4\right]  }\right\}  \mid\left(  +\right)  ,\mathit{\mu}^{\left[
4\right]  },\widetilde{\left(  \ \ \right)  }\right\rangle $ is the nonderived
$4$-ary non-division algebra over $\mathbb{R}$ obtained by the matrix
polyadization procedure from the algebra $\mathbb{C}$ of complex numbers. The
$4$-ary algebra of complex numbers $\mathbb{C}^{\left[  4\right]  }$ is not a
field, because it contains noninvertible nonzero elements (with some
$\mathit{z}_{i}=\mathit{0}$).

In $\mathbb{C}^{\left[  4\right]  ,tu}$ the querelement is given by the
following $6$-tuple%
\begin{equation}
\widetilde{\mathsf{v}}_{6}=\widetilde{\left(
\begin{array}
[c]{c}%
a_{1}\\
b_{1}\\
a_{2}\\
b_{2}\\
a_{3}\\
b_{3}%
\end{array}
\right)  }=\left(
\begin{array}
[c]{c}%
\dfrac{a_{2}a_{3}-b_{2}b_{3}}{\left(  a_{2}b_{3}+a_{3}b_{2}\right)
^{2}+\left(  a_{2}a_{3}-b_{2}b_{3}\right)  ^{2}}\\
-\dfrac{a_{2}b_{3}+a_{3}b_{2}}{\left(  a_{2}b_{3}+a_{3}b_{2}\right)
^{2}+\left(  a_{2}a_{3}-b_{2}b_{3}\right)  ^{2}}\\
\dfrac{a_{1}a_{3}-b_{1}b_{3}}{\left(  a_{1}b_{3}+a_{3}b_{1}\right)
^{2}+\left(  a_{1}a_{3}-b_{1}b_{3}\right)  ^{2}}\\
-\dfrac{a_{1}b_{3}+a_{3}b_{1}}{\left(  a_{1}b_{3}+a_{3}b_{1}\right)
^{2}+\left(  a_{1}a_{3}-b_{1}b_{3}\right)  ^{2}}\\
\dfrac{a_{1}a_{2}-b_{1}b_{2}}{\left(  a_{1}b_{2}+a_{2}b_{1}\right)
^{2}+\left(  a_{1}a_{2}-b_{1}b_{2}\right)  ^{2}}\\
-\dfrac{a_{1}b_{2}+a_{2}b_{1}}{\left(  a_{1}b_{2}+a_{2}b_{1}\right)
^{2}+\left(  a_{1}a_{2}-b_{1}b_{2}\right)  ^{2}}%
\end{array}
\right)  ,\ \ \ a_{i},b_{i}\in\mathbb{R}.
\end{equation}

Therefore, $\mathbb{C}^{\left[  4\right]  ,tu}$ is a division nonderived
$4$-ary algebra isomorphic to $\mathbb{C}^{\left[  4\right]  }$.

Thus, the application of the matrix polyadization procedure to the
commutative, associative, unital, two-dimensional, division algebra of complex
numbers $\mathbb{C}$ gives the noncommutative, $4$-ary nonderived, totally
associative, unital, six-dimensional, non-division algebra $\mathbb{C}%
^{\left[  4\right]  }$ over $\mathbb{R}$ (with the corresponding isomorphic
$4$-ary non-division algebra of $6$-tuples $\mathbb{C}^{\left[  4\right]
,tu}$), which we call the $4$-ary complex numbers. The subalgebra
$\mathbb{C}_{\operatorname{div}}^{\left[  4\right]  }\subset\mathbb{C}%
^{\left[  4\right]  }$ of invertible matrices $\mathit{Z}^{\left[  4\right]
}$ ($\det\mathit{Z}^{\left[  4\right]  }\neq\mathit{0}$, with all
$\mathit{z}_{i}\neq\mathit{0}$) is the division $4$-ary algebra of complex
numbers (by {Theorem} \ref{theor-div}).
\end{Example}

We then consider the polyadization of the noncommutative quaternion algebra
$\mathbb{H}$.

\begin{Example}
[Ternary quaternions]\label{exam-quat}The associative four-dimensional algebra
of quaternions is given by%
\begin{align}
\mathbb{H} &  =\left\langle \left\{  \mathit{q}\right\}  \mid\left(  +\right)
,\left(  \cdot\right)  \right\rangle ,\ \ \ \ \ \mathit{q}=a+b\mathit{i}%
+c\mathit{j}+d\mathit{k},\nonumber\\
\mathit{ij} &  =\mathit{k},\mathit{ij}=-\mathit{ji},\left(  \text{+cycled}%
\right)  ,\mathit{i}^{2}=\mathit{j}^{2}=\mathit{k}^{2}=\mathit{ijk}%
=-1,\ \ a,b,c,d\in\mathbb{R}.\label{qa}%
\end{align}

\hl{The} binary multiplication $\mu_{\mathsf{v}}^{\left[  4\right]  }$ of the
quadruples%
\begin{equation}
\mathsf{v}_{4}=\left(
\begin{array}
[c]{c}%
a\\
b\\
c\\
d
\end{array}
\right)  \in\mathbb{H}^{tu}=\left\langle \left\{  \mathsf{v}_{4}\right\}
\mid\left(  +\right)  ,\mu_{\mathsf{v}}^{\left[  2\right]  }\right\rangle
\end{equation}
is (we present the binary product in our notation here for completeness to
compare with the ternary case below)%
\begin{equation}
\mu_{\mathsf{v}}^{\left[  4\right]  }\left[  \left(
\begin{array}
[c]{c}%
a^{\prime}\\
b^{\prime}\\
c^{\prime}\\
d^{\prime}%
\end{array}
\right)  ,\left(
\begin{array}
[c]{c}%
a^{\prime\prime}\\
b^{\prime\prime}\\
c^{\prime\prime}\\
d^{\prime\prime}%
\end{array}
\right)  \right]  =\left(
\begin{array}
[c]{c}%
a^{\prime}a^{\prime\prime}-b^{\prime}b^{\prime\prime}-c^{\prime}%
c^{\prime\prime}-d^{\prime}d^{\prime\prime}\\
a^{\prime}b^{\prime\prime}+b^{\prime}a^{\prime\prime}-d^{\prime}%
c^{\prime\prime}+c^{\prime}d^{\prime\prime}\\
a^{\prime}c^{\prime\prime}+c^{\prime}a^{\prime\prime}+d^{\prime}%
b^{\prime\prime}-b^{\prime}d^{\prime\prime}\\
a^{\prime}d^{\prime\prime}+b^{\prime}c^{\prime\prime}-c^{\prime}%
b^{\prime\prime}+d^{\prime}a^{\prime\prime}%
\end{array}
\right)  \in\mathbb{H}^{tu},\label{mv4}%
\end{equation}
where $a^{\prime},a^{\prime\prime},b^{\prime},b^{\prime\prime},c^{\prime
},c^{\prime\prime},d^{\prime},d^{\prime\prime}\in\mathbb{R}$. The binary
algebra of quadruples $\mathbb{H}^{tu}$ is a noncommutative division algebra
(isomorphic to $\mathbb{H}$), without idempotents or zero divisors.

By the above matrix polyadization procedure, we construct the nonderived
ternary algebra of quaternions $\mathbb{H}^{\left[  3\right]  }=\left\langle
\left\{  \mathit{Q}^{\left[  3\right]  }\right\}  \mid\left(  +\right)
,\mathit{\mu}_{\mathit{Q}}^{\left[  3\right]  }\right\rangle $ of the
dimension $D^{\left[  3\right]  }=8$ (see (\ref{dn})) by introducing the
following $2\times2$ matrix (\ref{zn})%
\begin{equation}
\mathit{Q}^{\left[  3\right]  }=\left(
\begin{array}
[c]{cc}%
0 & \mathit{q}_{1}\\
\mathit{q}_{2} & 0
\end{array}
\right)  =\left(
\begin{array}
[c]{cc}%
0 & a_{1}+b_{1}\mathit{i}+c_{1}\mathit{j}+d_{1}\mathit{k}\\
a_{2}+b_{2}\mathit{i}+c_{2}\mathit{j}+d_{2}\mathit{k} & 0
\end{array}
\right)  \in\mathbb{H}^{\left[  3\right]  },\ \ \ \mathit{q}_{1}%
,\mathit{q}_{2}\in\mathbb{H}.
\end{equation}

\hl{The} nonderived ternary product of $\mathit{Q}$-matrices (\ref{z4c}) is%
\begin{equation}
\mathit{\mu}_{\mathit{Q}}^{\left[  3\right]  }\left[  \mathit{Q}^{\prime
},\mathit{Q}^{\prime\prime},\mathit{Q}^{\prime\prime\prime}\right]
=\mathit{Q}^{\prime}\mathit{Q}^{\prime\prime}\mathit{Q}^{\prime\prime\prime}.
\end{equation}

\hl{The} cyclic product corresponding to (\ref{m4z}) in components (\ref{cy})
becomes%
\begin{align}
\mathit{q}_{1}^{\prime}\mathit{q}_{2}^{\prime\prime}\mathit{q}_{1}%
^{\prime\prime\prime} &  =\mathit{q}_{1},\nonumber\\
\mathit{q}_{2}^{\prime}\mathit{q}_{1}^{\prime\prime}\mathit{q}_{2}%
^{\prime\prime\prime} &  =\mathit{q}_{2},\ \ \ \ \ \mathit{q}_{i}%
,\mathit{q}_{i}^{\prime},\mathit{q}_{i}^{\prime\prime},\mathit{q}_{i}%
^{\prime\prime\prime}\in\mathbb{H}.\label{q}%
\end{align}
\hl{In} {terms of eight-tuples $\mathsf{v}_{8}$ over $\mathbb{R}$ (cf. (\ref{ab})), the
ternary multiplication $\mu_{\mathsf{v}}^{\left[  3\right]  }$ in
$\mathbb{H}^{\left[  3\right]  ,tu}=\left\langle \left\{  \mathsf{v}%
_{8}\right\}  \mid\left(  +\right)  ,\mu_{\mathsf{v}}^{\left[  3\right]
}\right\rangle $ has the form (cf. (\ref{m2}))}

{\tiny
\begin{align}
&  \mu_{\mathsf{v}}^{\left[  3\right]  }\left[  \mathsf{v}_{8}^{\prime
},\mathsf{v}_{8}^{\prime\prime},\mathsf{v}_{8}^{\prime\prime\prime}\right]
=\mu_{\mathsf{v}}^{\left[  3\right]  }\left[  \left(
\begin{array}
[c]{c}%
a_{1}^{\prime}\\
b_{1}^{\prime}\\
c_{1}^{\prime}\\
d_{1}^{\prime}\\
a_{2}^{\prime}\\
b_{2}^{\prime}\\
c_{2}^{\prime}\\
d_{2}^{\prime}%
\end{array}
\right)  ,\left(
\begin{array}
[c]{c}%
a_{1}^{\prime\prime}\\
b_{1}^{\prime\prime}\\
c_{1}^{\prime\prime}\\
d_{1}^{\prime\prime}\\
a_{2}^{\prime\prime}\\
b_{2}^{\prime\prime}\\
c_{2}^{\prime\prime}\\
d_{2}^{\prime\prime}%
\end{array}
\right)  ,\left(
\begin{array}
[c]{c}%
a_{1}^{\prime\prime\prime}\\
b_{1}^{\prime\prime\prime}\\
c_{1}^{\prime\prime\prime}\\
d_{1}^{\prime\prime\prime}\\
a_{2}^{\prime\prime\prime}\\
b_{2}^{\prime\prime\prime}\\
c_{2}^{\prime\prime\prime}\\
d_{2}^{\prime\prime\prime}%
\end{array}
\right)  \right]  =\nonumber\\
&  \left(
\begin{array}
[c]{c}%
\left(  a_{1}^{\prime}a_{2}^{\prime\prime}-b_{1}^{\prime}b_{2}^{\prime\prime
}-c_{1}^{\prime}c_{2}^{\prime\prime}-d_{1}^{\prime}d_{2}^{\prime\prime
}\right)  a_{1}^{\prime\prime\prime}-\left(  a_{1}^{\prime}b_{2}^{\prime
\prime}+b_{1}^{\prime}a_{2}^{\prime\prime}+c_{1}^{\prime}d_{2}^{\prime\prime
}-d_{1}^{\prime}c_{2}^{\prime\prime}\right)  b_{1}^{\prime\prime\prime
}-\left(  a_{1}^{\prime}c_{2}^{\prime\prime}+c_{1}^{\prime}a_{2}^{\prime
\prime}+d_{1}^{\prime}b_{2}^{\prime\prime}-b_{1}^{\prime}d_{2}^{\prime\prime
}\right)  c_{1}^{\prime\prime\prime}-\left(  a_{1}^{\prime}d_{2}^{\prime
\prime}+d_{1}^{\prime}a_{2}^{\prime\prime}+b_{1}^{\prime}c_{2}^{\prime\prime
}-c_{1}^{\prime}b_{2}^{\prime\prime}\right)  d_{1}^{\prime\prime\prime}\\
\left(  a_{1}^{\prime}b_{2}^{\prime\prime}+b_{1}^{\prime}a_{2}^{\prime\prime
}+c_{1}^{\prime}d_{2}^{\prime\prime}-d_{1}^{\prime}c_{2}^{\prime\prime
}\right)  a_{1}^{\prime\prime\prime}+\left(  a_{1}^{\prime}a_{2}^{\prime
\prime}-b_{1}^{\prime}b_{2}^{\prime\prime}-c_{1}^{\prime}c_{2}^{\prime\prime
}-d_{1}^{\prime}d_{2}^{\prime\prime}\right)  b_{1}^{\prime\prime\prime
}-\left(  a_{1}^{\prime}d_{2}^{\prime\prime}+d_{1}^{\prime}a_{2}^{\prime
\prime}+b_{1}^{\prime}c_{2}^{\prime\prime}-c_{1}^{\prime}b_{2}^{\prime\prime
}\right)  c_{1}^{\prime\prime\prime}+\left(  a_{1}^{\prime}c_{2}^{\prime
\prime}+c_{1}^{\prime}a_{2}^{\prime\prime}+d_{1}^{\prime}b_{2}^{\prime\prime
}-b_{1}^{\prime}d_{2}^{\prime\prime}\right)  d_{1}^{\prime\prime\prime}\\
\left(  a_{1}^{\prime}c_{2}^{\prime\prime}+c_{1}^{\prime}a_{2}^{\prime\prime
}+d_{1}^{\prime}b_{2}^{\prime\prime}-b_{1}^{\prime}d_{2}^{\prime\prime
}\right)  a_{1}^{\prime\prime\prime}+\left(  a_{1}^{\prime}d_{2}^{\prime
\prime}+d_{1}^{\prime}a_{2}^{\prime\prime}+b_{1}^{\prime}c_{2}^{\prime\prime
}-c_{1}^{\prime}b_{2}^{\prime\prime}\right)  b_{1}^{\prime\prime\prime
}+\left(  a_{1}^{\prime}a_{2}^{\prime\prime}-b_{1}^{\prime}b_{2}^{\prime
\prime}-c_{1}^{\prime}c_{2}^{\prime\prime}-d_{1}^{\prime}d_{2}^{\prime\prime
}\right)  c_{1}^{\prime\prime\prime}-\left(  a_{1}^{\prime}b_{2}^{\prime
\prime}+b_{1}^{\prime}a_{2}^{\prime\prime}+c_{1}^{\prime}d_{2}^{\prime\prime
}-d_{1}^{\prime}c_{2}^{\prime\prime}\right)  d_{1}^{\prime\prime\prime}\\
\left(  a_{1}^{\prime}d_{2}^{\prime\prime}+d_{1}^{\prime}a_{2}^{\prime\prime
}+b_{1}^{\prime}c_{2}^{\prime\prime}-c_{1}^{\prime}b_{2}^{\prime\prime
}\right)  a_{1}^{\prime\prime\prime}-\left(  a_{1}^{\prime}c_{2}^{\prime
\prime}+c_{1}^{\prime}a_{2}^{\prime\prime}+d_{1}^{\prime}b_{2}^{\prime\prime
}-b_{1}^{\prime}d_{2}^{\prime\prime}\right)  b_{1}^{\prime\prime\prime
}+\left(  a_{1}^{\prime}b_{2}^{\prime\prime}+b_{1}^{\prime}a_{2}^{\prime
\prime}+c_{1}^{\prime}d_{2}^{\prime\prime}-d_{1}^{\prime}c_{2}^{\prime\prime
}\right)  c_{1}^{\prime\prime\prime}+\left(  a_{1}^{\prime}a_{2}^{\prime
\prime}-b_{1}^{\prime}b_{2}^{\prime\prime}-c_{1}^{\prime}c_{2}^{\prime\prime
}-d_{1}^{\prime}d_{2}^{\prime\prime}\right)  d_{1}^{\prime\prime\prime}\\
\left(  a_{2}^{\prime}a_{1}^{\prime\prime}-b_{2}^{\prime}b_{1}^{\prime\prime
}-c_{2}^{\prime}c_{1}^{\prime\prime}-d_{2}^{\prime}d_{1}^{\prime\prime
}\right)  a_{2}^{\prime\prime\prime}-\left(  a_{2}^{\prime}b_{1}^{\prime
\prime}+b_{2}^{\prime}a_{1}^{\prime\prime}+c_{2}^{\prime}d_{1}^{\prime\prime
}-d_{1}^{\prime}c_{2}^{\prime\prime}\right)  b_{2}^{\prime\prime\prime
}-\left(  a_{2}^{\prime}c_{1}^{\prime\prime}+c_{2}^{\prime}a_{1}^{\prime
\prime}+d_{2}^{\prime}b_{1}^{\prime\prime}-b_{2}^{\prime}d_{1}^{\prime\prime
}\right)  c_{2}^{\prime\prime\prime}+\left(  a_{2}^{\prime}d_{1}^{\prime
\prime}+d_{2}^{\prime}a_{1}^{\prime\prime}+b_{2}^{\prime}c_{1}^{\prime\prime
}-c_{2}^{\prime}b_{1}^{\prime\prime}\right)  d_{2}^{\prime\prime\prime}\\
\left(  a_{1}^{\prime}b_{2}^{\prime\prime}+b_{1}^{\prime}a_{2}^{\prime\prime
}+c_{1}^{\prime}d_{2}^{\prime\prime}-d_{1}^{\prime}c_{2}^{\prime\prime
}\right)  a_{2}^{\prime\prime\prime}-\left(  a_{2}^{\prime}b_{1}^{\prime
\prime}+b_{2}^{\prime}a_{1}^{\prime\prime}+c_{2}^{\prime}d_{1}^{\prime\prime
}-d_{2}^{\prime}c_{1}^{\prime\prime}\right)  b_{2}^{\prime\prime\prime
}-\left(  a_{2}^{\prime}d_{1}^{\prime\prime}+d_{2}^{\prime}a_{1}^{\prime
\prime}+b_{2}^{\prime}c_{1}^{\prime\prime}-c_{2}^{\prime}b_{1}^{\prime\prime
}\right)  c_{2}^{\prime\prime\prime}+\left(  a_{2}^{\prime}c_{1}^{\prime
\prime}+c_{2}^{\prime}a_{1}^{\prime\prime}+d_{2}^{\prime}b_{1}^{\prime\prime
}-b_{2}^{\prime}d_{1}^{\prime\prime}\right)  d_{2}^{\prime\prime\prime}\\
\left(  a_{2}^{\prime}c_{1}^{\prime\prime}+c_{2}^{\prime}a_{1}^{\prime\prime
}+d_{2}^{\prime}b_{1}^{\prime\prime}-b_{2}^{\prime}d_{1}^{\prime\prime
}\right)  a_{2}^{\prime\prime\prime}+\left(  a_{2}^{\prime}d_{1}^{\prime
\prime}+d_{2}^{\prime}a_{1}^{\prime\prime}+b_{2}^{\prime}c_{1}^{\prime\prime
}-c_{2}^{\prime}b_{1}^{\prime\prime}\right)  b_{2}^{\prime\prime\prime
}+\left(  a_{2}^{\prime}a_{1}^{\prime\prime}-b_{2}^{\prime}b_{1}^{\prime
\prime}-c_{2}^{\prime}c_{1}^{\prime\prime}-d_{2}^{\prime}d_{1}^{\prime\prime
}\right)  c_{2}^{\prime\prime\prime}-\left(  a_{2}^{\prime}b_{1}^{\prime
\prime}+b_{2}^{\prime}a_{1}^{\prime\prime}+c_{2}^{\prime}d_{1}^{\prime\prime
}-d_{2}^{\prime}c_{1}^{\prime\prime}\right)  d_{2}^{\prime\prime\prime}\\
\left(  a_{2}^{\prime}d_{1}^{\prime\prime}+d_{2}^{\prime}a_{1}^{\prime\prime
}+b_{2}^{\prime}c_{1}^{\prime\prime}-c_{2}^{\prime}b_{1}^{\prime\prime
}\right)  a_{2}^{\prime\prime\prime}-\left(  a_{2}^{\prime}c_{1}^{\prime
\prime}+c_{2}^{\prime}a_{1}^{\prime\prime}+d_{2}^{\prime}b_{1}^{\prime\prime
}-b_{2}^{\prime}d_{1}^{\prime\prime}\right)  b_{2}^{\prime\prime\prime
}+\left(  a_{2}^{\prime}b_{1}^{\prime\prime}+b_{2}^{\prime}a_{1}^{\prime
\prime}+c_{2}^{\prime}d_{1}^{\prime\prime}-d_{2}^{\prime}c_{1}^{\prime\prime
}\right)  c_{2}^{\prime\prime\prime}+\left(  a_{2}^{\prime}a_{1}^{\prime
\prime}-b_{2}^{\prime}b_{1}^{\prime\prime}-c_{2}^{\prime}c_{1}^{\prime\prime
}-d_{2}^{\prime}d_{1}^{\prime\prime}\right)  d_{2}^{\prime\prime\prime}%
\end{array}
\right)  \label{h3}%
\end{align}
} where $a_{i}^{\prime},a_{i}^{\prime\prime},a_{i}^{\prime\prime\prime}%
,b_{i}^{\prime},b_{i}^{\prime\prime},b_{i}^{\prime\prime\prime},c_{i}^{\prime
},c_{i}^{\prime\prime},c_{i}^{\prime\prime\prime},d_{i}^{\prime},d_{i}%
^{\prime\prime},d_{i}^{\prime\prime\prime}\in\mathbb{R}$.

\begin{Remark}
It is important to note that, although the ternary multiplication of $8$-tuples
(\ref{h3}) is nonderived and noncommutative, it is not the ordinary product of
three quaternion pairs, but corresponds to the nontrivial cyclic braided
products of (\ref{q}).
\end{Remark}

The polyadic unit (ternary unit) $\mathsf{e}_{8}$ in the ternary algebra of
$8$-tuples is (see (\ref{mn}))%
\begin{align}
\mathsf{e}_{8} &  =\left(
\begin{array}
[c]{c}%
\mathit{1}\\
0\\
0\\
0\\
\mathit{1}\\
0\\
0\\
0
\end{array}
\right)  \in\mathbb{H}^{\left[  3\right]  ,tu},\\
\mu_{\mathsf{v}}^{\left[  3\right]  }\left[  \mathsf{e}_{8},\mathsf{e}%
_{8},\mathsf{v}_{8}\right]   &  =\mu_{\mathsf{v}}^{\left[  3\right]  }\left[
\mathsf{e}_{8},\mathsf{v}_{8},\mathsf{e}_{8}\right]  =\mu_{\mathsf{v}%
}^{\left[  3\right]  }\left[  \mathsf{v}_{8},\mathsf{e}_{8},\mathsf{e}%
_{8}\right]  =\mathsf{v}_{8}.
\end{align}

\hl{The} querelement (\ref{mgg}) and (\ref{zv}) of the nonderived noncommutative
$8$-dimensional ternary algebra of quaternions $\mathbb{H}^{\left[  3\right]
}$ has the matrix form which follows from the general equations (\ref{zd})%
\begin{equation}
\widetilde{\mathit{Q}}^{\left[  3\right]  }=\left(
\begin{array}
[c]{cc}%
0 & \mathit{q}_{2}^{-1}\\
\mathit{q}_{1}^{-1} & 0
\end{array}
\right)  \in\mathbb{H}^{\left[  3\right]  },\ \ \ \mathit{q}_{1}%
,\mathit{q}_{2}\in\mathbb{H\setminus}\left\{  0\right\}  .
\end{equation}

Therefore, $\mathbb{H}^{\left[  3\right]  }=\left\langle \left\{
\mathit{Q}^{\left[  3\right]  }\right\}  \mid\left(  +\right)  ,\mathit{\mu
}_{\mathit{Q}}^{\left[  3\right]  },\widetilde{\left(  \ \ \right)
}\right\rangle $ is the nonderived noncommutative ternary non-division algebra
obtained by the matrix polyadization procedure from the algebra $\mathbb{H}$
of quaternions. Because the elements $\mathit{Q}^{\left[  3\right]  }$ with
$\mathit{q}_{1}=0$ or $\mathit{q}_{2}=0$ are nonzero, but noninvertible,
$\mathbb{H}^{\left[  3\right]  }$ is not a field.

In $\mathbb{H}^{\left[  3\right]  ,tu}$ the querelement is given by the
following $8$-tuple%
\begin{equation}
\widetilde{\mathsf{v}}_{8}=\widetilde{\left(
\begin{array}
[c]{c}%
a_{1}\\
b_{1}\\
c_{1}\\
d_{1}\\
a_{2}\\
b_{2}\\
c_{2}\\
d_{2}%
\end{array}
\right)  }=\left(
\begin{array}
[c]{c}%
\dfrac{a_{1}}{a_{1}^{2}+b_{1}^{2}+c_{1}^{2}+d_{1}^{2}}\\
-\dfrac{b_{1}}{a_{1}^{2}+b_{1}^{2}+c_{1}^{2}+d_{1}^{2}}\\
-\dfrac{c_{1}}{a_{1}^{2}+b_{1}^{2}+c_{1}^{2}+d_{1}^{2}}\\
-\dfrac{d_{1}}{a_{1}^{2}+b_{1}^{2}+c_{1}^{2}+d_{1}^{2}}\\
\dfrac{a_{2}}{a_{2}^{2}+b_{2}^{2}+c_{2}^{2}+d_{2}^{2}}\\
-\dfrac{b_{2}}{a_{2}^{2}+b_{2}^{2}+c_{2}^{2}+d_{2}^{2}}\\
-\dfrac{c_{2}}{a_{2}^{2}+b_{2}^{2}+c_{2}^{2}+d_{2}^{2}}\\
-\dfrac{d_{2}}{a_{2}^{2}+b_{2}^{2}+c_{2}^{2}+d_{2}^{2}}%
\end{array}
\right)  ,\ \ \ a_{1,2}^{2}+b_{1,2}^{2}+c_{1,2}^{2}+d_{1,2}^{2}\neq
0,\ a_{i},b_{i},c_{i},d_{i}\in\mathbb{R}.
\end{equation}

Therefore, $\mathbb{H}^{\left[  3\right]  ,tu}$ is a non-division ternary
algebra isomorphic to $\mathbb{H}^{\left[  3\right]  }$.

To conclude, the application of the matrix polyadization procedure to the
noncommutative, associative, unital, four-dimensional, division algebra of
quaternions $\mathbb{H}$ gives the noncommutative, nonderived ternary, totally
associative, unital, $8$-dimensional, non-division algebra $\mathbb{H}%
^{\left[  3\right]  }$ over $\mathbb{R}$ (with the corresponding isomorphic
ternary non-division algebra of $8$-tuples $\mathbb{H}^{\left[  3\right]
,tu}$), which we call the ternary quaternions. The subalgebra $\mathbb{H}%
_{\operatorname{div}}^{\left[  3\right]  }\subset\mathbb{H}^{\left[  3\right]
}$ of invertible matrices $\mathit{Q}^{\left[  3\right]  }$ ($\det
\mathit{Q}^{\left[  3\right]  }\neq\mathit{0}$, with all $\mathit{q}_{i}%
\neq\mathit{0}$) is the division ternary algebra of quaternions (see
{Theorem} \ref{theor-div}).
\end{Example}

\section{\textsc{Polyadic Norms}}

The division algebras $\mathbb{D}=\mathbb{R},\mathbb{C},\mathbb{H},\mathbb{O}$
are normed as vector spaces, and the corresponding Euclidean $2$-norm is
multiplicative (\ref{zz}), such that the corresponding mapping is a binary
homomorphism. It would be worthwhile to define a polyadic analog of the binary
norm $\left\Vert \ \right\Vert $ having similar properties.

\begin{Definition}
We define the polyadic ($n$-ary) norm $\left\Vert \ \ \right\Vert ^{\left[
n\right]  }$ for the $n$-ary algebra $\mathbb{D}^{\left[  n\right]  }$, that
is obtained from $\mathbb{D}$ by the matrix polyadization procedure
(\ref{zn}), as the product (in $\mathbb{R}$) of the component norms%
\begin{equation}
\left\Vert \mathit{Z}\right\Vert ^{\left[  n\right]  }=\left\Vert
\mathit{z}_{1}\right\Vert \left\Vert \mathit{z}_{2}\right\Vert \ldots
\left\Vert \mathit{z}_{n-1}\right\Vert \in\mathbb{R},\ \ \ \ \ \mathit{Z}%
\in\mathbb{D}^{\left[  n\right]  },\ \ \mathit{z}_{i}\in\mathbb{D}.
\label{znn}%
\end{equation}

\end{Definition}

\begin{Corollary}
The polyadic norm (\ref{znn}) is zero for the noninvertible elements of
$\mathbb{D}^{\left[  n\right]  }$, having some $\mathit{z}_{i}=0$.
\end{Corollary}

Therefore, it is worthwhile to consider a polyadic norm for invertible
elements of $\mathbb{D}^{\left[  n\right]  }$ only.

\begin{Proposition}
The division $n$-ary subalgebras $\mathbb{D}_{\operatorname{div}}^{\left[
n\right]  }\subset\mathbb{D}^{\left[  n\right]  }$ are normed $n$-ary algebras
with respect to the polyadic norm $\left\Vert \ \ \right\Vert ^{\left[
n\right]  }$ (\ref{znn}).
\end{Proposition}

Let us consider some properties of the polyadic norm (\ref{znn}).

\begin{Proposition}
The polyadic norm $\left\Vert \ \ \right\Vert ^{\left[  n\right]  }$
introduced above (\ref{znn}) has the following properties (for invertible
$\mathit{Z}\in\mathbb{D}^{\left[  n\right]  }$)
\begin{align}
\left\Vert \mathit{E}^{\left[  n\right]  }\right\Vert ^{\left[  n\right]  }
&  =1\in\mathbb{R},\ \ \ \ \mathit{E}^{\left[  n\right]  }\in\mathbb{D}%
^{\left[  n\right]  },\label{nlz}\\
\left\Vert \lambda\mathit{Z}\right\Vert ^{\left[  n\right]  }  &  =\left\vert
\lambda\right\vert ^{n-1}\left\Vert \mathit{Z}\right\Vert ^{\left[  n\right]
},\ \ \ \ \lambda\in\mathbb{R},\label{nl}\\
\left\Vert \mathit{Z}^{\prime}+\mathit{Z}^{\prime\prime}\right\Vert ^{\left[
n\right]  }  &  \leq\left\Vert \mathit{Z}\right\Vert ^{\left[  n\right]
}+\left\Vert \mathit{Z}\right\Vert ^{\left[  n\right]  },\ \ \mathit{Z}%
\in\mathbb{D}^{\left[  n\right]  }, \label{nt}%
\end{align}
where $\mathit{E}^{\left[  n\right]  }$ is the polyadic unit in the $n$-ary
algebra $\mathbb{D}^{\left[  n\right]  }$ (\ref{en}).
\end{Proposition}

\begin{proof}
The first property is obvious, the second one follows from the definition
(\ref{znn}), and the linearity of the ordinary Euclidean norm in $\mathbb{D}$
(\ref{11}). The polyadic triangle inequality (\ref{nt}) follows from the binary
triangle inequality (\ref{12}), because of the binary addition of $\mathit{Z}%
$-matrices of the cyclic block-shift form (\ref{zn}).
\end{proof}

The norms satisfying (\ref{nl}) are called norms of higher degree, and they
were investigated for the binary case in \cite{pum2011}.

The most important property of any (binary) norm is its multiplicativity
(\ref{zz}).

\begin{Theorem}
The polyadic norm $\left\Vert \ \ \right\Vert ^{\left[  n\right]  }$ defined
in (\ref{znn}) is $n$-ary multiplicative (such that the corresponding map
$\mathbb{D}^{\left[  n\right]  }\rightarrow\mathbb{R}$ is an $n$-ary
homomorphism)%
\begin{equation}
\left\Vert \mathit{\mu}_{\mathit{Z}}^{\left[  n\right]  }\left[  \overset
{n}{\overbrace{\mathit{Z}^{\prime},\mathit{Z}^{\prime\prime}\ldots
\mathit{Z}^{\prime\prime\prime}}}\right]  \right\Vert ^{\left[  n\right]
}=\overset{n}{\overbrace{\left\Vert \mathit{Z}^{\prime}\right\Vert ^{\left[
n\right]  }\cdot\left\Vert \mathit{Z}^{\prime\prime}\right\Vert ^{\left[
n\right]  }\cdot\ldots\cdot\left\Vert \mathit{Z}^{\prime\prime\prime
}\right\Vert ^{\left[  n\right]  }}},\ \ \ \ \mathit{Z}^{\prime}%
,\mathit{Z}^{\prime\prime},\mathit{Z}^{\prime\prime\prime}\in\mathbb{D}%
_{\operatorname{div}}^{\left[  n\right]  }. \label{mzz}%
\end{equation}

\end{Theorem}

\begin{proof}
Consider the component form (\ref{cy}) of each multiplier $\mathit{Z}$ in
(\ref{mzz}), then use the definition (\ref{znn}) and commutativity (as they
are in $\mathbb{R}$) and the multiplicativity (\ref{zz}) of the ordinary
binary norms $\left\Vert \ \right\Vert $ to rearrange the products of norms
from l.h.s. to r.h.s. in (\ref{mzz}). That is,
{
\begin{align}
&  \overset{n-1}{\overbrace{\left\Vert \mathit{z}_{1}^{\prime}\mathit{z}%
_{2}^{\prime\prime}\ldots\mathit{z}_{n-1}^{\prime\prime\prime}\mathit{z}%
_{1}^{\prime\prime\prime\prime}\right\Vert \cdot\left\Vert \mathit{z}%
_{2}^{\prime}\mathit{z}_{3}^{\prime\prime}\ldots\mathit{z}_{1}^{\prime
\prime\prime}\mathit{z}_{2}^{\prime\prime\prime\prime}\right\Vert \cdot
\ldots\cdot\left\Vert \mathit{z}_{n-1}^{\prime}\mathit{z}_{1}^{\prime\prime
}\ldots\mathit{z}_{n-2}^{\prime\prime\prime}\mathit{z}_{n-1}^{\prime
\prime\prime\prime}\right\Vert }}\nonumber\\
=  &  \overset{n}{\overbrace{\left(  \left\Vert \mathit{z}_{1}^{\prime
}\right\Vert \left\Vert \mathit{z}_{2}^{\prime}\right\Vert \ldots\left\Vert
\mathit{z}_{n-1}^{\prime}\right\Vert \right)  \cdot\left(  \left\Vert
\mathit{z}_{1}^{\prime\prime}\right\Vert \left\Vert \mathit{z}_{2}%
^{\prime\prime}\right\Vert \ldots\left\Vert \mathit{z}_{n-1}^{\prime\prime
}\right\Vert \right)  \cdot\ldots\cdot\left(  \left\Vert \mathit{z}%
_{1}^{\prime\prime\prime}\right\Vert \left\Vert \mathit{z}_{2}^{\prime
\prime\prime}\right\Vert \ldots\left\Vert \mathit{z}_{n-1}^{\prime\prime
\prime}\right\Vert \right)  \cdot\left(  \left\Vert \mathit{z}_{1}%
^{\prime\prime\prime\prime}\right\Vert \left\Vert \mathit{z}_{2}^{\prime
\prime\prime\prime}\right\Vert \ldots\left\Vert \mathit{z}_{n-1}^{\prime
\prime\prime\prime}\right\Vert \right)  }}.
\end{align}
}
\end{proof}

\begin{Remark}
The $n$-ary multiplicativity (\ref{mzz}) of the polyadic norm introduced in
(\ref{znn}) is independent of the concrete form of the binary norm $\left\Vert
\ \ \right\Vert $, and only the multiplicativity of the latter is needed.
\end{Remark}

\begin{Proposition}
The polyadic norm of the querelement in $n$-ary division subalgebra
$\mathbb{D}_{\operatorname{div}}^{\left[  n\right]  }$ is%
\begin{equation}
\left\Vert \widetilde{\mathit{Z}}\right\Vert ^{\left[  n\right]  }=\frac
{1}{\left(  \left\Vert \mathit{Z}\right\Vert ^{\left[  n\right]  }\right)
^{n-2}}\in\mathbb{R}_{>0},\ \ \ \forall\widetilde{\mathit{Z}}\in
\mathbb{D}_{\operatorname{div}}^{\left[  n\right]  }. \label{z2}%
\end{equation}

\end{Proposition}

\begin{proof}
It follows from the component relations for the querelements $\left\Vert
\widetilde{\mathit{z}}_{i}\right\Vert $ (\ref{zd}) and multiplicativity of the
binary norm that%
\begin{align}
\overset{n}{\overbrace{\left\Vert \mathit{z}_{1}\right\Vert \left\Vert
\mathit{z}_{2}\right\Vert \ldots\left\Vert \mathit{z}_{n-1}\right\Vert
\left\Vert \widetilde{\mathit{z}}_{1}\right\Vert }}  &  =\left\Vert
\mathit{z}_{1}\right\Vert ,\nonumber\\
\overset{n}{\overbrace{\left\Vert \mathit{z}_{2}\right\Vert \left\Vert
\mathit{z}_{3}\right\Vert \ldots\left\Vert \mathit{z}_{n-1}\right\Vert
\left\Vert \mathit{z}_{1}\right\Vert \left\Vert \widetilde{\mathit{z}}%
_{2}\right\Vert }}  &  =\left\Vert \mathit{z}_{2}\right\Vert ,\nonumber\\
&  \vdots\nonumber\\
\overset{n}{\overbrace{\left\Vert \mathit{z}_{n-1}\right\Vert \left\Vert
\mathit{z}_{1}\right\Vert \ldots\left\Vert \mathit{z}_{n-2}\right\Vert
\left\Vert \widetilde{\mathit{z}}_{n-1}\right\Vert }}  &  =\left\Vert
\mathit{z}_{n-1}\right\Vert ,\ \ \ \ \ \ \mathit{z}_{i},\widetilde{\mathit{z}%
}_{i}\in\mathbb{D},\ \ \left\Vert \mathit{z}_{i}\right\Vert ,\left\Vert
\widetilde{\mathit{z}}_{i}\right\Vert \in\mathbb{R}_{>0}. \label{zz1}%
\end{align}
\hl{Multiplying} all the equations in (\ref{zz1}) and using the definition of the
polyadic norm (\ref{znn}), together with the component form (\ref{zn}) and the
querelement (\ref{zv}), we obtain%
\begin{equation}
\left(  \left\Vert \mathit{Z}\right\Vert ^{\left[  n\right]  }\right)
^{n-1}\left\Vert \widetilde{\mathit{Z}}\right\Vert ^{\left[  n\right]
}=\left\Vert \mathit{Z}\right\Vert ^{\left[  n\right]  },
\end{equation}
from which follows (\ref{z2}).
\end{proof}

\begin{Example}
In the division, the $4$-ary algebra of complex numbers $\mathbb{C}%
_{\operatorname{div}}^{\left[  4\right]  }$ from  {Example}
\ref{exam-complex} the polyadic norm becomes%
\begin{equation}
\left\Vert \mathit{Z}\right\Vert ^{\left[  4\right]  }=\sqrt{\left(  a_{1}%
^{2}+b_{1}^{2}\right)  \left(  a_{2}^{2}+b_{2}^{2}\right)  \left(  a_{3}%
^{2}+b_{3}^{2}\right)  },\ \ \ a_{i},b_{i}\in\mathbb{R}.
\end{equation}
\hl{The} polyadic norm of the querelement $\widetilde{\mathit{Z}}$ in
$\mathbb{C}^{\left[  4\right]  }$ (\ref{z2}) is%
\begin{equation}
\left\Vert \widetilde{\mathit{Z}}\right\Vert ^{\left[  4\right]  }=\frac
{1}{\left(  a_{1}^{2}+b_{1}^{2}\right)  \left(  a_{2}^{2}+b_{2}^{2}\right)
\left(  a_{3}^{2}+b_{3}^{2}\right)  },\ \ \ \ \ a_{i}^{2}+b_{i}^{2}%
\neq0,\ \ \ \ a_{i},b_{i}\in\mathbb{R}.
\end{equation}

\end{Example}

\begin{Example}
In the division ternary algebra of quaternions $\mathbb{H}_{\operatorname{div}%
}^{\left[  3\right]  }$ from  {Example} \ref{exam-quat} the polyadic
norm becomes%
\begin{equation}
\left\Vert \mathit{Q}\right\Vert ^{\left[  3\right]  }=\sqrt{\left(  a_{1}%
^{2}+b_{1}^{2}+c_{1}^{2}+d_{1}^{2}\right)  \left(  a_{2}^{2}+b_{2}^{2}%
+c_{2}^{2}+d_{2}^{2}\right)  },\ \ \ a_{i},b_{i}\in\mathbb{R}.
\end{equation}
\hl{The} polyadic norm of the querelement $\widetilde{\mathit{Q}}$ in
$\mathbb{H}_{\operatorname{div}}^{\left[  3\right]  }$ (\ref{z2}) is%
\begin{equation}
\left\Vert \widetilde{\mathit{Q}}\right\Vert ^{\left[  3\right]  }=\frac
{1}{\sqrt{\left(  a_{1}^{2}+b_{1}^{2}+c_{1}^{2}+d_{1}^{2}\right)  \left(
a_{2}^{2}+b_{2}^{2}+c_{2}^{2}+d_{2}^{2}\right)  }},\ \ \ \ \ \ a_{i}^{2}%
+b_{i}^{2}+c_{i}^{2}+d_{i}^{2}\neq0,\ \ \ \ a_{i},b_{i}\in\mathbb{R}.
\end{equation}

\end{Example}

Further properties of the polyadic norm $\left\Vert \mathit{\ \ }\right\Vert
^{\left[  n\right]  }$ can be investigated for invertible elements of concrete
$n$-ary algebras.

\section{\textsc{Polyadic Analog of the Cayley--Dickson Construction}}

The standard method of obtaining the higher hypercomplex algebras is the
Cayley--Dickson construction \cite{scha54,kor2012,fla2019}. It is well known
that all four binary division algebras $\mathbb{D}=\mathbb{R},\mathbb{C}%
,\mathbb{H},\mathbb{O}$ can be built in this way \cite{kan/sol}. Here, we
generalize the Cayley--Dickson construction to the polyadic ($n$-ary) division
algebras introduced in the previous section. As a result, the number of
polyadic division algebras becomes infinite (as opposed to just four in the
binary case), because of the arbitrary initial and final arities of the
algebras under consideration. For illustration, we present several low arity
examples, since higher arity cases become too cumbersome and difficult to see.
First, we recall in brief (just to install our notation) the ordinary (binary)
Cayley--Dickson doubling process (in our notation, which is convenient for the
polyadization procedure).

\subsection{Abstract (Tuple) Approach}

Consider the sequence of algebras $\mathbb{A}_{\ell}$, $\ell\geq0$, over the
reals, starting from $\mathbb{A}_{0}=\mathbb{R}$. The main idea is to repeat
the doubling process of the complex number construction using pairs (doubles)
(\ref{m2}) and taking into account the isomorphism (\ref{cc}) at each stage
$\mathbb{A}_{\ell}\rightarrow\mathbb{A}_{\ell+1}$. Let us denote the binary
algebra over $\mathbb{R}$ on the $\ell$-th stage with the underlying set $\left\{
\mathbb{A}_{\ell}\right\}  $ as%
\begin{equation}
\mathbb{A}_{\ell}=\left\langle \left\{  \mathbb{A}_{\ell}\right\}  \mid\left(
+\right)  ,\mu_{\ell}^{\left[  2\right]  },\left(  \ast_{\ell}\right)
\right\rangle ,\label{alm}%
\end{equation}
where $\left(  \ast_{\ell}\right)  $ is involution in $\mathbb{A}_{\ell}$ and
$\mu_{\ell}^{\left[  2\right]  }:\mathbb{A}_{\ell}\otimes\mathbb{A}_{\ell
}\rightarrow\mathbb{A}_{\ell}$ is its binary multiplication, and we also will
write $\mu_{\ell}^{\left[  2\right]  }\equiv\left(  \cdot_{\ell}\right)  $.
The corresponding algebra of doubles ($2$-tuples)%
\begin{equation}
\mathsf{v}_{2\left(  \ell\right)  }=\left(
\begin{array}
[c]{c}%
\mathit{a}_{\left(  \ell\right)  }\\
\mathit{b}_{\left(  \ell\right)  }%
\end{array}
\right)  ,\ \ \ \mathit{a}_{\left(  \ell\right)  },\mathit{b}_{\left(
\ell\right)  }\in\mathbb{A}_{\ell},\label{v2}%
\end{equation}
is denoted by%
\begin{equation}
\mathbb{A}_{\ell}^{tu}=\left\langle \left\{  \mathsf{v}_{2\left(  \ell\right)
}\right\}  \mid\left(  +\right)  ,\mu_{\mathsf{v}\left(  \ell\right)
}^{\left[  2\right]  },\left(  \ast_{\ell}^{tu}\right)  \right\rangle
,\label{alt}%
\end{equation}
where $\left(  \ast_{\ell}^{tu}\right)  $ is involution in $\mathbb{A}_{\ell
}^{tu}$ and $\mu_{\mathsf{v}\left(  \ell\right)  }^{\left[  2\right]
}:\mathbb{A}_{\ell}^{tu}\otimes\mathbb{A}_{\ell}^{tu}\rightarrow
\mathbb{A}_{\ell}^{tu}$ is the binary product of doubles. If $\ell=0$, then
the conjugation is the identity map, as it should be for the reals
$\mathbb{R}$. The addition and scalar multiplication are made componentwise in
the standard way.

In this notation, the Cayley--Dickson doubling process is defined by the
recurrent multiplication formula%
\begin{align}
&  \mu_{\mathsf{v}\left(  \ell+1\right)  }^{\left[  2\right]  }\left[  \left(
\begin{array}
[c]{c}%
\mathit{a}_{\left(  \ell\right)  }^{\prime}\\
\mathit{b}_{\left(  \ell\right)  }^{\prime}%
\end{array}
\right)  ,\left(
\begin{array}
[c]{c}%
\mathit{a}_{\left(  \ell\right)  }^{\prime\prime}\\
\mathit{b}_{\left(  \ell\right)  }^{\prime\prime}%
\end{array}
\right)  \right]  =\left(
\begin{array}
[c]{c}%
\mu_{\ell}^{\left[  2\right]  }\left[  \mathit{a}_{\left(  \ell\right)
}^{\prime},\mathit{a}_{\left(  \ell\right)  }^{\prime\prime}\right]
-\mu_{\ell}^{\left[  2\right]  }\left[  \left(  \mathit{b}_{\left(
\ell\right)  }^{\prime\prime}\right)  ^{\ast_{\ell}},\mathit{b}_{\left(
\ell\right)  }^{\prime}\right]  \\
\mu_{\ell}^{\left[  2\right]  }\left[  \mathit{b}_{\left(  \ell\right)
}^{\prime\prime},\mathit{a}_{\left(  \ell\right)  }^{\prime}\right]
+\mu_{\ell}^{\left[  2\right]  }\left[  \mathit{b}_{\left(  \ell\right)
}^{\prime},\left(  \mathit{a}_{\left(  \ell\right)  }^{\prime\prime}\right)
^{\ast_{\ell}}\right]
\end{array}
\right)  \label{mv}\\
&  =\left(
\begin{array}
[c]{c}%
\mathit{a}_{\left(  \ell\right)  }^{\prime}\cdot_{\ell}\mathit{a}_{\left(
\ell\right)  }^{\prime\prime}-\left(  \mathit{b}_{\left(  \ell\right)
}^{\prime\prime}\right)  ^{\ast_{\ell}}\cdot_{\ell}\mathit{b}_{\left(
\ell\right)  }^{\prime}\\
\mathit{b}_{\left(  \ell\right)  }^{\prime\prime}\cdot_{\ell}\mathit{a}%
_{\left(  \ell\right)  }^{\prime}+\mathit{b}_{\left(  \ell\right)  }^{\prime
}\cdot_{\ell}\left(  \mathit{a}_{\left(  \ell\right)  }^{\prime\prime}\right)
^{\ast_{\ell}}%
\end{array}
\right)  \equiv\left(
\begin{array}
[c]{c}%
\mathit{a}_{\left(  \ell+1\right)  }\\
\mathit{b}_{\left(  \ell+1\right)  }%
\end{array}
\right)  \in\mathbb{A}_{\ell+1}^{tu},\label{mv1}%
\end{align}
and the recurrent conjugation%
\begin{equation}
\left(
\begin{array}
[c]{c}%
\mathit{a}_{\left(  \ell\right)  }\\
\mathit{b}_{\left(  \ell\right)  }%
\end{array}
\right)  ^{\ast_{\ell+1}^{tu}}=\left(
\begin{array}
[c]{c}%
\mathit{a}_{\left(  \ell\right)  }^{\ast_{\ell}}\\
-\mathit{b}_{\left(  \ell\right)  }%
\end{array}
\right)  \in\mathbb{A}_{\ell+1}^{tu},\ \ \ \ \ \ \ \mathit{a}_{\left(
\ell\right)  },\mathit{b}_{\left(  \ell\right)  }\in\mathbb{A}_{\ell
}.\label{abl}%
\end{equation}

Then, we use the isomorphism (\ref{cc}) which now becomes%
\begin{equation}
\mathbb{A}_{\ell+1}^{tu}\cong\mathbb{A}_{\ell+1}.\label{ata}%
\end{equation}

To go to the next level of recursion from that obtained so far, $\mathbb{A}%
_{\ell+1}$, we use (\ref{mv})--(\ref{ata}) while changing $\ell\rightarrow
\ell+1$. The dimension of the algebra $\mathbb{A}_{\ell}$ is%
\begin{equation}
D(\ell)=2^{\ell}, \label{dl}%
\end{equation}
such that each element can be presented in the form of $2^{\ell}$-tuple of
reals%
\begin{equation}
\left(
\begin{array}
[c]{c}%
a_{\left(  \ell\right)  ,1}\\
a_{\left(  \ell\right)  ,2}\\
\vdots\\
a_{\left(  \ell\right)  ,2^{\ell}}%
\end{array}
\right)  \in\mathbb{A}_{\ell},
\end{equation}
and the conjugated $2^{\ell}$-tuple becomes%
\begin{equation}
\left(
\begin{array}
[c]{c}%
a_{\left(  \ell\right)  ,1}\\
-a_{\left(  \ell\right)  2}\\
\vdots\\
-a_{\left(  \ell\right)  ,2^{\ell}}%
\end{array}
\right)  ,\ \ \ \ a_{\left(  \ell\right)  ,k}\in\mathbb{R}.
\end{equation}

For clarity, we intentionally mark elements and operations at the $\ell$-th
level explicitly, because they really are different for different $\ell$. The
example with $\ell=0$ just obtains the algebra of complex numbers
$\mathbb{A}_{1}\equiv\mathbb{C}$ from the algebra of reals $\mathbb{A}%
_{0}\equiv\mathbb{R}$ (\ref{v})--(\ref{cc}). All the binary division algebras
$\mathbb{D}=\mathbb{R},\mathbb{C},\mathbb{H},\mathbb{O}$ can be obtained by
the Cayley--Dickson construction \cite{kan/sol}.

\subsection{Concrete (Hyperembedding) Approach\label{subsec-concrete}}

Alternatively, one can reparametrize the pairs (\ref{v2}) satisfying the
complex-like multiplication (\ref{mv}), as the field extension $\widehat
{\mathbb{A}}_{\ell}=\mathbb{A}_{\ell}\left(  \mathit{i}_{\ell}\right)  $ by
one complex-like unit $\mathit{i}_{\ell}$ on each $\ell$-th stage of
iteration, $\ell\geq0$. The Cayley--Dickson doubling process is given by the
iterations%
\begin{equation}
\widehat{\mathbb{A}}_{\ell+1}=\left\langle \left\{  \mathbb{A}_{\ell}\left(
\mathit{i}_{\ell}\right)  \right\}  \mid\left(  +\right)  ,\widehat{\mu}%
_{\ell+1}^{\left[  2\right]  },\left(  \ast_{\ell+1}\right)  \right\rangle
\left(  \mathit{i}_{\ell+1}\right)  ,\ \ \ \mathit{i}_{\ell}^{2}%
=-1,\ \ \ \mathit{i}_{\ell+1}^{2}=-1,\ \ \ \mathit{i}_{\ell+1}\mathit{i}%
_{\ell}=-\mathit{i}_{\ell}\mathit{i}_{\ell+1}. \label{aa}%
\end{equation}

\hl{In} a component form, it is given by%
\begin{equation}
\mathit{z}_{\left(  \ell+1\right)  }=\mathit{z}_{\left(  \ell\right)
,1}+\mathit{z}_{\left(  \ell\right)  ,2}\cdot_{\ell+1}\mathit{i}_{\ell
+1}=\mathit{z}_{\left(  \ell\right)  ,1}+\widehat{\mu}_{\ell+1}^{\left[
2\right]  }\left[  \mathit{z}_{\left(  \ell\right)  ,2},\mathit{i}_{\ell
+1}\right]  \in\widehat{\mathbb{A}}_{\ell+1},\ \ \ \mathit{z}_{\left(
\ell\right)  ,1},\mathit{z}_{\left(  \ell\right)  ,2}\in\widehat{\mathbb{A}%
}_{\ell}. \label{zl}%
\end{equation}

The product in $\widehat{\mathbb{A}}_{\ell+1}$ can be expressed through the
product of the previous stage from $\widehat{\mathbb{A}}_{\ell}$ and
conjugation by using the anticommutation of the imaginary units from different
stages $\mathit{i}_{\ell+1}\mathit{i}_{\ell}=-\mathit{i}_{\ell}\mathit{i}%
_{\ell+1}$ (see (\ref{aa})). Thus, we obtain the standard complex-like
multiplication (see (\ref{m2}) and (\ref{mv1})) on each $\ell$-th stage of the
Cayley--Dickson doubling process
\vspace{-6pt}
 \begin{align}
&  \mathit{z}_{\left(  \ell+1\right)  }^{\prime}\cdot_{\ell+1}\mathit{z}%
_{\left(  \ell+1\right)  }^{\prime\prime}\nonumber\\
&  =\left(  \mathit{z}_{\left(  \ell\right)  ,1}^{\prime}\cdot_{\ell
}\mathit{z}_{\left(  \ell\right)  ,1}^{\prime\prime}-\left(  \mathit{z}%
_{\left(  \ell\right)  ,2}^{\prime\prime}\right)  ^{\ast_{\ell}}\cdot_{\ell
}\mathit{z}_{\left(  \ell\right)  ,2}^{\prime}\right)  +\left(  \mathit{z}%
_{\left(  \ell\right)  ,2}^{\prime\prime}\cdot_{\ell}\mathit{z}_{\left(
\ell\right)  ,1}^{\prime}+\mathit{z}_{\left(  \ell\right)  ,2}^{\prime}%
\cdot_{\ell}\left(  \mathit{z}_{\left(  \ell\right)  ,1}^{\prime\prime
}\right)  ^{\ast_{\ell}}\right)  \cdot_{\ell+1}\mathit{i}_{\ell+1},
\label{zzl}%
\end{align}
or with the manifest form of different stage multiplications $\widehat{\mu
}_{\ell}^{\left[  2\right]  }$ and $\widehat{\mu}_{\ell+1}^{\left[  2\right]
}$ (needed to go on higher arities)%
\begin{align}
\widehat{\mu}_{\ell+1}^{\left[  2\right]  }\left[  \mathit{z}_{\left(
\ell+1\right)  }^{\prime},\mathit{z}_{\left(  \ell+1\right)  }^{\prime\prime
}\right]   &  =\left(  \widehat{\mu}_{\ell}^{\left[  2\right]  }\left[
\mathit{z}_{\left(  \ell\right)  ,1}^{\prime},\mathit{z}_{\left(  \ell\right)
,1}^{\prime\prime}\right]  -\widehat{\mu}_{\ell}^{\left[  2\right]  }\left[
\left(  \mathit{z}_{\left(  \ell\right)  ,2}^{\prime\prime}\right)
^{\ast_{\ell}},\mathit{z}_{\left(  \ell\right)  ,2}^{\prime}\right]  \right)
\nonumber\\
&  +\widehat{\mu}_{\ell+1}^{\left[  2\right]  }\left[  \left(  \widehat{\mu
}_{\ell}^{\left[  2\right]  }\left[  \mathit{z}_{\left(  \ell\right)
,2}^{\prime\prime},\mathit{z}_{\left(  \ell\right)  ,1}^{\prime}\right]
+\widehat{\mu}_{\ell}^{\left[  2\right]  }\left[  \mathit{z}_{\left(
\ell\right)  ,2}^{\prime},\left(  \mathit{z}_{\left(  \ell\right)  ,1}%
^{\prime\prime}\right)  ^{\ast_{\ell}}\right]  \right)  ,\mathit{i}_{\ell
+1}\right]  .
\end{align}

Let us consider the example of quaternion construction in our notation.

\begin{Example}
\label{exam-quat1}In the case of quaternions $\ell=1$ and $\widehat
{\mathbb{A}}_{\ell+1}=\widehat{\mathbb{A}}_{2}=\mathbb{H}$, while
$\widehat{\mathbb{A}}_{\ell}=\widehat{\mathbb{A}}_{1}=\mathbb{C}$. The
recurrent relations (\ref{zl}) for $\ell=0,1$ become%
\begin{align}
\mathit{z}_{\left(  2\right)  }  &  =\mathit{z}_{\left(  1\right)
,1}+\mathit{z}_{\left(  1\right)  ,2}\cdot_{2}\mathit{i}_{2},\ \ \ \mathit{z}%
_{\left(  1\right)  ,1},\mathit{z}_{\left(  1\right)  ,2}\in\widehat
{\mathbb{A}}_{1}=\mathbb{C},\ \mathit{z}_{\left(  2\right)  },\mathit{i}%
_{2}\in\mathbb{H},\label{zz2}\\
\mathit{z}_{\left(  1\right)  ,1}  &  =\mathit{z}_{\left(  0\right)
,11}+\mathit{z}_{\left(  0\right)  ,12}\cdot_{1}\mathit{i}_{1}%
,\ \ \ \mathit{z}_{\left(  0\right)  ,11},\mathit{z}_{\left(  0\right)
,12}\in\widehat{\mathbb{A}}_{0}=\mathbb{R},\ \ \mathit{z}_{\left(  1\right)
,1},\mathit{i}_{1}\in\mathbb{C},\label{z11}\\
\mathit{z}_{\left(  1\right)  ,2}  &  =\mathit{z}_{\left(  0\right)
,21}+\mathit{z}_{\left(  0\right)  ,22}\cdot_{1}\mathit{i}_{1}%
,\ \ \ \mathit{z}_{\left(  0\right)  ,21},\mathit{z}_{\left(  0\right)
,22}\in\widehat{\mathbb{A}}_{0}=\mathbb{R},\ \ \mathit{z}_{\left(  1\right)
,2},\mathit{i}_{1}\in\mathbb{C}. \label{z12}%
\end{align}

\hl{After} the substitution of (\ref{z11}) and (\ref{z12}) into (\ref{zz2}), we obtain the
expression of quaternions with real coefficients $\mathit{z}_{\left(
0\right)  ,\alpha\beta}\in\mathbb{R}$, $\alpha,\beta=1,2$, and two imagery
units (from different parts of $\mathbb{H}$) $\mathit{i}_{1}\in
\mathbb{C\setminus R}\subset\mathbb{H}$ and $\mathit{i}_{2}\in\mathbb{H}%
\setminus\mathbb{C}$%
\begin{align}
\mathit{z}_{\left(  2\right)  }  &  =\mathit{z}_{\left(  1\right)
,1}+\mathit{z}_{\left(  1\right)  ,2}\cdot_{2}\mathit{i}_{2}=\left(
\mathit{z}_{\left(  0\right)  ,11}+\mathit{z}_{\left(  0\right)  ,12}\cdot
_{1}\mathit{i}_{1}\right)  +\left(  \mathit{z}_{\left(  0\right)
,21}+\mathit{z}_{\left(  0\right)  ,22}\cdot_{1}\mathit{i}_{1}\right)
\cdot_{2}\mathit{i}_{2}\nonumber\\
&  =\mathit{z}_{\left(  0\right)  ,11}+\mathit{z}_{\left(  0\right)  ,12}%
\cdot_{1}\mathit{i}_{1}+\mathit{z}_{\left(  0\right)  ,21}\cdot_{2}%
\mathit{i}_{2}+\left(  \mathit{z}_{\left(  0\right)  ,22}\cdot_{1}%
\mathit{i}_{1}\right)  \cdot_{2}\mathit{i}_{2}. \label{z22}%
\end{align}
\hl{To} return to the standard notation (\ref{qa}), we put $\mathit{z}_{\left(
2\right)  }=\mathit{q}$, $\mathit{z}_{\left(  0\right)  ,11}=a$,
$\mathit{z}_{\left(  0\right)  ,12}=b$, $\mathit{z}_{\left(  0\right)  ,21}%
=c$, $\mathit{z}_{\left(  0\right)  ,22}=d$, $\mathit{i}_{1}=\mathit{i}%
\in\mathbb{C}$, $\mathit{i}_{2}=\mathit{j}\in\mathbb{H\setminus C}$,
$\mathit{i}_{1}\cdot_{2}\mathit{i}_{2}=\mathit{k}\in\mathbb{H}$ and obtain
$\mathit{q}=a+b\mathit{i}+c\mathit{j}+d\mathit{k}$. Similarly, the
complex-like multiplication (\ref{zzl}) can be also applied twice with
$\ell=0,1$ to obtain the quaternion multiplication in terms of real
coefficients (\ref{mv4}).
\end{Example}

\subsection{Polyadic Cayley--Dickson Process}

Now, we provide a generalization of the Cayley-Dickson construction to the
polyadic case in the framework of the second (embedding) approach using the
field extension formalism (see {Section} \ref{subsec-concrete}). The
main iteration relation (\ref{aa}) will now contain, instead of binary
hypercomplex algebras $\widehat{\mathbb{A}}_{\ell}$, $n$-ary algebras
$\widehat{\mathbb{A}}_{\ell}^{\left[  n_{\ell}\right]  }$ at each stage.

\begin{Definition}
The polyadic Cayley--Dickson process is defined as the following iteration on
$n$-ary algebras
\begin{align}
\widehat{\mathbb{A}}_{\ell+1}^{\left[  n_{\ell+1}\right]  }  &  =\left\langle
\left\{  \mathbb{A}_{\ell}^{\left[  n_{\ell}\right]  }\left(  \mathit{i}%
_{\ell}\right)  \right\}  \mid\left(  +\right)  ,\widehat{\mu}_{\ell
+1}^{\left[  n_{\ell+1}\right]  }\right\rangle \left(  \mathit{i}_{\ell
+1}\right)  ,\nonumber\\
\mathit{i}_{\ell}^{2}  &  =-1,\ \ \ \mathit{i}_{\ell+1}^{2}%
=-1,\ \ \ \mathit{i}_{\ell+1}\mathit{i}_{\ell}=-\mathit{i}_{\ell}%
\mathit{i}_{\ell+1},\ \ \mathit{i}_{\ell}\in\mathbb{A}_{\ell+1}^{\left[
n_{\ell+1}\right]  },\ \ n_{\ell}\geq2,\ \ \ell\geq0. \label{anl}%
\end{align}

\end{Definition}

The concrete representation of the $\ell$-th stage $n$-ary algebras
$\widehat{\mathbb{A}}_{\ell}^{\left[  n_{\ell}\right]  }$ is not important for
the general recurrence formula (\ref{anl}). Nevertheless, here we will use
matrix polyadization to obtain higher $n$-ary non-division algebras
({Section} \ref{sec-matr}). First, we will need the obvious

\begin{Lemma}
\label{lem-monom}The embedding of a block-monomial matrix into a
block-monomial matrix gives a block-monomial matrix.
\end{Lemma}

\begin{Corollary}
If the binary Cayley--Dickson construction gives a division algebra, then the
corresponding polyadic Cayley--Dickson process gives a nonderived $n$-ary
non-division algebra, because of its noninvertible elements.
\end{Corollary}

Thus, the structure of the general algebra obtained by the polyadic
Cayley--Dickson process is richer than a one block-shift monomial matrix
(\ref{zn}), it is the \textquotedblleft tower\textquotedblright\ of such
monomial matrices of size $\left(  n_{\ell}-1\right)  \times\left(  n_{\ell
}-1\right)  $ on $\ell$-th stage embedded one into another. The connection
between near arities (and matrix sizes) is%
\begin{equation}
n_{\ell+1}=\varkappa\left(  n_{\ell}-1\right)  +1, \label{nl1}%
\end{equation}
where $\varkappa$ is the polyadic power (\ref{al}) \cite{duplij2022}.

\begin{Proposition}
If the number of stages of the polyadic Cayley--Dickson process is $\ell$, then
the dimension of the final algebra is (cf. (\ref{dl}))%
\begin{equation}
D_{CD}(\ell)=D(\mathbb{A}_{CD}^{\left[  n_{0},n_{1},\ldots,n_{\ell}\right]
})=2^{\ell}\left(  n_{0}-1\right)  \left(  n_{1}-1\right)  \ldots\left(
n_{\ell}-1\right)  \label{dnn}%
\end{equation}
where $n_{i}$ are the arities of the intermediate algebras. The size of the
final matrix becomes%
\begin{equation}
\left(  n_{0}-1\right)  \left(  n_{1}-1\right)  \ldots\left(  n_{\ell
}-1\right)  \times\left(  n_{0}-1\right)  \left(  n_{1}-1\right)
\ldots\left(  n_{\ell}-1\right)  .
\end{equation}

\end{Proposition}

\begin{proof}
On the $\ell$-th stage of the polyadic Cayley--Dickson process, each
block-monomial $\left(  n_{\ell}-1\right)  \times\left(  n_{\ell}-1\right)  $
matrix has $\left(  n_{\ell}-1\right)  $ nonzero blocks of the cycle-shift
shape (\ref{zn}). The $0$-th stage corresponds to reals, which gives
$2^{0}\left(  n_{0}-1\right)  $ parameters. Then, the simple field extension
$\ell=1$ (\ref{zl}) with blocks of reals gives $\left(  n_{0}-1\right)
\cdot2\left(  n_{1}-1\right)  $ parameters and so on. Thus, the $\ell$-th
stage is given by (\ref{dnn}).
\end{proof}

More concretely, for the nonderived $n$-ary algebras, we have the dimensions%
\begin{equation}%
\begin{array}
[c]{ll}%
\mathbb{R}^{\left[  n_{0}\right]  }, & D(\mathbb{R}^{\left[  n_{0}\right]
})=\left(  n_{0}-1\right)  ,\\
\mathbb{C}_{CD}^{\left[  n_{0},n_{1}\right]  }=\mathbb{R}_{CD}^{\left[
n_{0}\right]  }\left(  \mathit{i}_{1}\right)  , & D(\mathbb{C}_{CD}^{\left[
n_{0},n_{1}\right]  })=2\left(  n_{0}-1\right)  \left(  n_{1}-1\right)  ,\\
\mathbb{H}_{CD}^{\left[  n_{0},n_{1},n_{2}\right]  }=\mathbb{C}_{CD}^{\left[
n_{0},n_{1}\right]  }\left(  \mathit{i}_{2}\right)  , & D(\mathbb{H}%
_{CD}^{\left[  n_{0},n_{1},n_{2}\right]  })=2^{2}\left(  n_{0}-1\right)
\left(  n_{1}-1\right)  \left(  n_{2}-1\right)  ,\\
\mathbb{O}_{CD}^{\left[  n_{0},n_{1},n_{2},n_{3}\right]  }=\mathbb{H}%
_{CD}^{\left[  n_{0},n_{1},n_{2}\right]  }\left(  \mathit{i}_{3}\right)  , &
D(\mathbb{O}_{CD}^{\left[  n_{0},n_{1},n_{2},n_{3}\right]  })=2^{3}\left(
n_{0}-1\right)  \left(  n_{1}-1\right)  \left(  n_{2}-1\right)  \left(
n_{3}-1\right)  ,\\
\vdots & \vdots\\
\mathbb{A}_{CD,\ell}^{\left[  n_{0},n_{1},\ldots,n_{\ell}\right]  }%
=\mathbb{A}_{CD}^{\left[  n_{0},n_{1},\ldots,n_{\ell-1}\right]  }\left(
\mathit{i}_{\ell}\right)  , & D(\mathbb{A}_{CD}^{\left[  n_{0},n_{1}%
,\ldots,n_{\ell}\right]  })=2^{\ell}\left(  n_{0}-1\right)  \left(
n_{1}-1\right)  \ldots\left(  n_{\ell}-1\right)  .
\end{array}
\label{ad}%
\end{equation}

The starting algebra $\mathbb{R}^{\left[  n_{0}\right]  }$ is presented by the
$\mathit{Z}$-matrix (\ref{zn}) with $n=n_{0}$ and real entries ($\mathbb{A}%
=\mathbb{R}$).

\begin{Definition}
The sequence of embedded cyclic shift block matrices in (\ref{ad}) will be
called a polyadic block-shift tower.
\end{Definition}

\begin{Remark}
The shapes of the final matrices (\ref{ad}) are different from the cyclic
shift weighted matrices (\ref{zn}), but nevertheless, all the intermediate
blocks are cyclic shift block matrices. Adjacent (in $\ell$) matrices do not
have an arbitrary size, but are connected by (\ref{nl1}).
\end{Remark}

\begin{Example}
To clarify the above general formulas, we provide the shape of polyadic
quaternion matrices with $\ell=2$, and $n_{0}=5$, $n_{1}=3$, $n_{2}=4$,
as {
\begin{equation}
\left(
\begin{tabular}
[c]{cccccccccccccccccccccccc}\cline{9-16}%
$\cdot$ & $\cdot$ & $\cdot$ & $\cdot$ & $\cdot$ & $\cdot$ & $\cdot$ & $\cdot$
& \multicolumn{1}{||c}{$\cdot$} & $\cdot$ & $\cdot$ & $\cdot$ &
\multicolumn{1}{|c}{$\cdot$} & $\star$ & $\cdot$ & $\cdot$ &
\multicolumn{1}{||c}{$\cdot$} & $\cdot$ & $\cdot$ & $\cdot$ & $\cdot$ &
$\cdot$ & $\cdot$ & $\cdot$\\
$\cdot$ & $\cdot$ & $\cdot$ & $\cdot$ & $\cdot$ & $\cdot$ & $\cdot$ & $\cdot$
& \multicolumn{1}{||c}{$\cdot$} & $\cdot$ & $\cdot$ & $\cdot$ &
\multicolumn{1}{|c}{$\cdot$} & $\cdot$ & $\star$ & $\cdot$ &
\multicolumn{1}{||c}{$\cdot$} & $\cdot$ & $\cdot$ & $\cdot$ & $\cdot$ &
$\cdot$ & $\cdot$ & $\cdot$\\
$\cdot$ & $\cdot$ & $\cdot$ & $\cdot$ & $\cdot$ & $\cdot$ & $\cdot$ & $\cdot$
& \multicolumn{1}{||c}{$\cdot$} & $\cdot$ & $\cdot$ & $\cdot$ &
\multicolumn{1}{|c}{$\cdot$} & $\cdot$ & $\cdot$ & $\star$ &
\multicolumn{1}{||c}{$\cdot$} & $\cdot$ & $\cdot$ & $\cdot$ & $\cdot$ &
$\cdot$ & $\cdot$ & $\cdot$\\
$\cdot$ & $\cdot$ & $\cdot$ & $\cdot$ & $\cdot$ & $\cdot$ & $\cdot$ & $\cdot$
& \multicolumn{1}{||c}{$\cdot$} & $\cdot$ & $\cdot$ & $\cdot$ &
\multicolumn{1}{|c}{$\star$} & $\cdot$ & $\cdot$ & $\cdot$ &
\multicolumn{1}{||c}{$\cdot$} & $\cdot$ & $\cdot$ & $\cdot$ & $\cdot$ &
$\cdot$ & $\cdot$ & $\cdot$\\
$\cdot$ & $\cdot$ & $\cdot$ & $\cdot$ & $\cdot$ & $\cdot$ & $\cdot$ & $\cdot$
& \multicolumn{1}{||c}{$\cdot$} & $\star$ & $\cdot$ & $\cdot$ &
\multicolumn{1}{|c}{$\cdot$} & $\cdot$ & $\cdot$ & $\cdot$ &
\multicolumn{1}{||c}{$\cdot$} & $\cdot$ & $\cdot$ & $\cdot$ & $\cdot$ &
$\cdot$ & $\cdot$ & $\cdot$\\
$\cdot$ & $\cdot$ & $\cdot$ & $\cdot$ & $\cdot$ & $\cdot$ & $\cdot$ & $\cdot$
& \multicolumn{1}{||c}{$\cdot$} & $\cdot$ & $\star$ & $\cdot$ &
\multicolumn{1}{|c}{$\cdot$} & $\cdot$ & $\cdot$ & $\cdot$ &
\multicolumn{1}{||c}{$\cdot$} & $\cdot$ & $\cdot$ & $\cdot$ & $\cdot$ &
$\cdot$ & $\cdot$ & $\cdot$\\
$\cdot$ & $\cdot$ & $\cdot$ & $\cdot$ & $\cdot$ & $\cdot$ & $\cdot$ & $\cdot$
& \multicolumn{1}{||c}{$\cdot$} & $\cdot$ & $\cdot$ & $\star$ &
\multicolumn{1}{|c}{$\cdot$} & $\cdot$ & $\cdot$ & $\cdot$ &
\multicolumn{1}{||c}{$\cdot$} & $\cdot$ & $\cdot$ & $\cdot$ & $\cdot$ &
$\cdot$ & $\cdot$ & $\cdot$\\
$\cdot$ & $\cdot$ & $\cdot$ & $\cdot$ & $\cdot$ & $\cdot$ & $\cdot$ & $\cdot$
& \multicolumn{1}{||c}{$\star$} & $\cdot$ & $\cdot$ & $\cdot$ &
\multicolumn{1}{|c}{$\cdot$} & $\cdot$ & $\cdot$ & $\cdot$ &
\multicolumn{1}{||c}{$\cdot$} & $\cdot$ & $\cdot$ & $\cdot$ & $\cdot$ &
$\cdot$ & $\cdot$ & $\cdot$\\\cline{9-24}%
$\cdot$ & $\cdot$ & $\cdot$ & $\cdot$ & $\cdot$ & $\cdot$ & $\cdot$ & $\cdot$
& $\cdot$ & $\cdot$ & $\cdot$ & $\cdot$ & $\cdot$ & $\cdot$ & $\cdot$ &
$\cdot$ & \multicolumn{1}{||c}{$\cdot$} & $\cdot$ & $\cdot$ & $\cdot$ &
\multicolumn{1}{|c}{$\cdot$} & $\star$ & $\cdot$ & \multicolumn{1}{c||}{$\cdot
$}\\
$\cdot$ & $\cdot$ & $\cdot$ & $\cdot$ & $\cdot$ & $\cdot$ & $\cdot$ & $\cdot$
& $\cdot$ & $\cdot$ & $\cdot$ & $\cdot$ & $\cdot$ & $\cdot$ & $\cdot$ &
$\cdot$ & \multicolumn{1}{||c}{$\cdot$} & $\cdot$ & $\cdot$ & $\cdot$ &
\multicolumn{1}{|c}{$\cdot$} & $\cdot$ & $\star$ & \multicolumn{1}{c||}{$\cdot
$}\\
$\cdot$ & $\cdot$ & $\cdot$ & $\cdot$ & $\cdot$ & $\cdot$ & $\cdot$ & $\cdot$
& $\cdot$ & $\cdot$ & $\cdot$ & $\cdot$ & $\cdot$ & $\cdot$ & $\cdot$ &
$\cdot$ & \multicolumn{1}{||c}{$\cdot$} & $\cdot$ & $\cdot$ & $\cdot$ &
\multicolumn{1}{|c}{$\cdot$} & $\cdot$ & $\cdot$ & \multicolumn{1}{c||}{$\star
$}\\
$\cdot$ & $\cdot$ & $\cdot$ & $\cdot$ & $\cdot$ & $\cdot$ & $\cdot$ & $\cdot$
& $\cdot$ & $\cdot$ & $\cdot$ & $\cdot$ & $\cdot$ & $\cdot$ & $\cdot$ &
$\cdot$ & \multicolumn{1}{||c}{$\cdot$} & $\cdot$ & $\cdot$ & $\cdot$ &
\multicolumn{1}{|c}{$\star$} & $\cdot$ & $\cdot$ & \multicolumn{1}{c||}{$\cdot
$}\\
$\cdot$ & $\cdot$ & $\cdot$ & $\cdot$ & $\cdot$ & $\cdot$ & $\cdot$ & $\cdot$
& $\cdot$ & $\cdot$ & $\cdot$ & $\cdot$ & $\cdot$ & $\cdot$ & $\cdot$ &
$\cdot$ & \multicolumn{1}{||c}{$\cdot$} & $\star$ & $\cdot$ & $\cdot$ &
\multicolumn{1}{|c}{$\cdot$} & $\cdot$ & $\cdot$ & \multicolumn{1}{c||}{$\cdot
$}\\
$\cdot$ & $\cdot$ & $\cdot$ & $\cdot$ & $\cdot$ & $\cdot$ & $\cdot$ & $\cdot$
& $\cdot$ & $\cdot$ & $\cdot$ & $\cdot$ & $\cdot$ & $\cdot$ & $\cdot$ &
$\cdot$ & \multicolumn{1}{||c}{$\cdot$} & $\cdot$ & $\star$ & $\cdot$ &
\multicolumn{1}{|c}{$\cdot$} & $\cdot$ & $\cdot$ & \multicolumn{1}{c||}{$\cdot
$}\\
$\cdot$ & $\cdot$ & $\cdot$ & $\cdot$ & $\cdot$ & $\cdot$ & $\cdot$ & $\cdot$
& $\cdot$ & $\cdot$ & $\cdot$ & $\cdot$ & $\cdot$ & $\cdot$ & $\cdot$ &
$\cdot$ & \multicolumn{1}{||c}{$\cdot$} & $\cdot$ & $\cdot$ & $\star$ &
\multicolumn{1}{|c}{$\cdot$} & $\cdot$ & $\cdot$ & \multicolumn{1}{c||}{$\cdot
$}\\
$\cdot$ & $\cdot$ & $\cdot$ & $\cdot$ & $\cdot$ & $\cdot$ & $\cdot$ & $\cdot$
& $\cdot$ & $\cdot$ & $\cdot$ & $\cdot$ & $\cdot$ & $\cdot$ & $\cdot$ &
$\cdot$ & \multicolumn{1}{||c}{$\star$} & $\cdot$ & $\cdot$ & $\cdot$ &
\multicolumn{1}{|c}{$\cdot$} & $\cdot$ & $\cdot$ & \multicolumn{1}{c||}{$\cdot
$}\\\cline{1-8}\cline{17-24}%
\multicolumn{1}{||c}{$\cdot$} & $\cdot$ & $\cdot$ & $\cdot$ &
\multicolumn{1}{|c}{$\cdot$} & $\star$ & $\cdot$ & $\cdot$ &
\multicolumn{1}{||c}{$\cdot$} & $\cdot$ & $\cdot$ & $\cdot$ & $\cdot$ &
$\cdot$ & $\cdot$ & $\cdot$ & $\cdot$ & $\cdot$ & $\cdot$ & $\cdot$ & $\cdot$
& $\cdot$ & $\cdot$ & $\cdot$\\
\multicolumn{1}{||c}{$\cdot$} & $\cdot$ & $\cdot$ & $\cdot$ &
\multicolumn{1}{|c}{$\cdot$} & $\cdot$ & $\star$ & $\cdot$ &
\multicolumn{1}{||c}{$\cdot$} & $\cdot$ & $\cdot$ & $\cdot$ & $\cdot$ &
$\cdot$ & $\cdot$ & $\cdot$ & $\cdot$ & $\cdot$ & $\cdot$ & $\cdot$ & $\cdot$
& $\cdot$ & $\cdot$ & $\cdot$\\
\multicolumn{1}{||c}{$\cdot$} & $\cdot$ & $\cdot$ & $\cdot$ &
\multicolumn{1}{|c}{$\cdot$} & $\cdot$ & $\cdot$ & $\star$ &
\multicolumn{1}{||c}{$\cdot$} & $\cdot$ & $\cdot$ & $\cdot$ & $\cdot$ &
$\cdot$ & $\cdot$ & $\cdot$ & $\cdot$ & $\cdot$ & $\cdot$ & $\cdot$ & $\cdot$
& $\cdot$ & $\cdot$ & $\cdot$\\
\multicolumn{1}{||c}{$\cdot$} & $\cdot$ & $\cdot$ & $\cdot$ &
\multicolumn{1}{|c}{$\star$} & $\cdot$ & $\cdot$ & $\cdot$ &
\multicolumn{1}{||c}{$\cdot$} & $\cdot$ & $\cdot$ & $\cdot$ & $\cdot$ &
$\cdot$ & $\cdot$ & $\cdot$ & $\cdot$ & $\cdot$ & $\cdot$ & $\cdot$ & $\cdot$
& $\cdot$ & $\cdot$ & $\cdot$\\\cline{1-8}%
\multicolumn{1}{||c}{$\cdot$} & $\star$ & $\cdot$ & $\cdot$ &
\multicolumn{1}{|c}{$\cdot$} & $\cdot$ & $\cdot$ & $\cdot$ &
\multicolumn{1}{||c}{$\cdot$} & $\cdot$ & $\cdot$ & $\cdot$ & $\cdot$ &
$\cdot$ & $\cdot$ & $\cdot$ & $\cdot$ & $\cdot$ & $\cdot$ & $\cdot$ & $\cdot$
& $\cdot$ & $\cdot$ & $\cdot$\\
\multicolumn{1}{||c}{$\cdot$} & $\cdot$ & $\star$ & $\cdot$ &
\multicolumn{1}{|c}{$\cdot$} & $\cdot$ & $\cdot$ & $\cdot$ &
\multicolumn{1}{||c}{$\cdot$} & $\cdot$ & $\cdot$ & $\cdot$ & $\cdot$ &
$\cdot$ & $\cdot$ & $\cdot$ & $\cdot$ & $\cdot$ & $\cdot$ & $\cdot$ & $\cdot$
& $\cdot$ & $\cdot$ & $\cdot$\\
\multicolumn{1}{||c}{$\cdot$} & $\cdot$ & $\cdot$ & $\star$ &
\multicolumn{1}{|c}{$\cdot$} & $\cdot$ & $\cdot$ & $\cdot$ &
\multicolumn{1}{||c}{$\cdot$} & $\cdot$ & $\cdot$ & $\cdot$ & $\cdot$ &
$\cdot$ & $\cdot$ & $\cdot$ & $\cdot$ & $\cdot$ & $\cdot$ & $\cdot$ & $\cdot$
& $\cdot$ & $\cdot$ & $\cdot$\\
\multicolumn{1}{||c}{$\star$} & $\cdot$ & $\cdot$ & $\cdot$ &
\multicolumn{1}{|c}{$\cdot$} & $\cdot$ & $\cdot$ & $\cdot$ &
\multicolumn{1}{||c}{$\cdot$} & $\cdot$ & $\cdot$ & $\cdot$ & $\cdot$ &
$\cdot$ & $\cdot$ & $\cdot$ & $\cdot$ & $\cdot$ & $\cdot$ & $\cdot$ & $\cdot$
& $\cdot$ & $\cdot$ & $\cdot$\\\cline{1-8}%
\end{tabular}
\ \right)  \label{q24}%
\end{equation}
}

We have a $24\times24$ monomial matrix, $24=\left(  5-1\right)  \left(
3-1\right)  \left(  4-1\right)  $ by (\ref{ad}), where dots denote zeroes and
stars denote nonzero entries. The dimension of this polyadic quaternion
algebra represented by (\ref{q24}) is $2^{2}\cdot24=96$, because of the two
simple field extensions (\ref{zz2})--(\ref{z22}). This is a $13$-ary
quaternion non-division algebra. If there were to be no further stages of the
Cayley--Dickson process, and the $24\times24$ monomial matrix would not be
composed, having the form (\ref{zn}), then it would give a $25$-ary algebra.
\end{Example}

Thus, we present the polyadic Cayley-Dickson process in the component form
(\ref{zl}), but instead of the elements $\mathit{z}_{\left(  \ell\right)  }$ of
the $\ell$-th stage, we place the $\mathit{Z}$-matrices (\ref{zn}) of suitable
sizes. The resulting matrix is still monomial, and so its determinant is
proportional to the product of elements. The subalgebra of invertible matrices
corresponds to a division $n$-ary algebra.

\begin{Theorem}
The invertible elements (which are described by $\mathit{Z}$-matrices with a
nonzero determinant) of the polyadic Cayley--Dickson construction
$\mathbb{D}_{CD}^{\left[  n\right]  }$ (for the first stages $\ell=0,1,2,3$)
(\ref{ad}) form the subalgebra $\mathbb{D}_{CD,\operatorname{div}}^{\left[
n\right]  }\subset\mathbb{D}_{CD}^{\left[  n\right]  }$ which is a polyadic
division $n$-ary algebra $\mathbb{D}_{CD,\operatorname{div}}^{\left[
n\right]  }$ corresponding to the binary division algebras $\mathbb{D}%
=\mathbb{R},\mathbb{C},\mathbb{H},\mathbb{O}$.
\end{Theorem}

To conclude, the matrix polyadization procedure applied to division algebras
leads, in general, to non-division algebras, because it is presented by
monomial matrices (representing nonzero noninvertible elements) which become
noninvertible when at least one entry vanishes. However, the subalgebras of
invertible elements can be considered as new $n$-ary division algebras.

\section{\textsc{Polyadic Product of Vectors}}

Now, we will show that the matrix polyadization procedure ({Section}
\ref{sec-matr}) is connected with a new product of vectors in a vector space,
which we will consider below.

First, we recall some properties of monomial and related matrices
\cite{joyner}. An arbitrary monomial (or generalized permutation) matrix
$\mathit{M}_{mon}$ (over a field $\mathbb{F}^{\times}=\mathbb{F}%
\setminus\left\{0\right\}  $, and so having nonzero entries) can be
presented as a product of an invertible diagonal matrix $\mathit{M}_{diag}$
and a permutation matrix $\mathit{M}_{per}$%
\begin{equation}
\mathit{M}_{mon}=\mathit{M}_{diag}\mathit{M}_{per}.\label{mm}%
\end{equation}

The set of all $m\times m$ monomial matrices form a binary subgroup
$\mathcal{G}_{mon}$ of the general linear group $GL_{m}\left(  \mathbb{F}%
\right)  $, while the set of monomial matrices of the special fixed shape of
the cyclic shift weighted matrices (\ref{zn}), form a nonderived $\left(
m+1\right)  $-ary group (see {Section} \ref{sec-matr}). Abstractly,
(\ref{mm}) can be considered as the matrix representation of the wreath
product of $\mathbb{F}^{\times}$ and the symmetric group $S_{m}$, because the
group of diagonal matrices is isomorphic to $\left(  \mathbb{F}^{\times
}\right)  ^{m}$ \cite{joyner}. In general, a monomial representation of a
binary group $\mathcal{G}$ is a homomorphism to some binary subgroup of
$\mathcal{G}_{mon}$.

On the other hand, there exists the procedure of vectorization (see, e.g.,
\cite{hackbusch,mag/neu}) which establishes a correspondence between (for
instance square) $m\times m$ matrices and $m^{2}$-tuples (which can be
interpreted as coordinate expressions of $m^{2}\times1$ column vectors in a
vector space) derived by stacking the columns of the matrix, such that%
\begin{equation}
\operatorname*{vec}\nolimits_{\mathbb{F}}:\mathbb{F}^{m\times m}%
\rightarrow\mathbb{F}^{m^{2}}. \label{vf}%
\end{equation}

In general, vectorization is a homomorphism. In particular, for diagonal
matrices we have%
\begin{equation}
\operatorname*{vec}\nolimits_{\mathbb{F}}\left(  \mathit{M}_{diag}^{\prime
}\mathit{M}_{diag}^{\prime\prime}\right)  =\operatorname*{vec}%
\nolimits_{\mathbb{F}}\left(  \mathit{M}_{diag}^{\prime}\right)
\odot\operatorname*{vec}\nolimits_{\mathbb{F}}\left(  \mathit{M}%
_{diag}^{\prime\prime}\right)  , \label{vmm}%
\end{equation}
where $\odot$ is the element-wise product of $m^{2}$-tuples. For the monomial
matrices (\ref{mm}), we define the following modification of vectorization.

\begin{Definition}
Reduced vectorization is the mapping of a monomial matrix into $m$-tuple%
\begin{equation}
\operatorname*{vec}\nolimits_{\mathbb{F}}^{\times}:\mathbb{F}^{m\times
m}\rightarrow\mathbb{F}^{m}, \label{vf1}%
\end{equation}
which is derived from vectorization (\ref{vf}) by omitting zeroes on both sides.
\end{Definition}

\begin{Remark}
\label{rem-mon}The relation (\ref{vmm}) is not valid for general monomial
matrices (\ref{mm}). Nevertheless, we will show that, for a special class of
monomial matrices, the cycled shift weighted matrices (\ref{zn}), the
homomorphism (\ref{vmm}) can be obtained for non-binary products on both
sides, and that will allow us to define a new kind of product of vectors (in
any vector space).
\end{Remark}

Let $\mathfrak{V}_{\mathbb{F}}$ be a finite-dimensional vector space (over a
field $\mathbb{F}$) of dimension $m\geq2$, such that%
\begin{equation}
\mathfrak{V}_{\mathbb{F}}\ni\mathbf{v}=\mathbf{v}\left(  x\right)  =\sum
_{i=1}^{m}x_{i}\mathbf{e}_{i},\ \ \ \ \ x_{i}\in\mathbb{F}, \label{fv}%
\end{equation}
where $\mathbf{e}_{i}$ are canonical basis vectors.

In $\mathfrak{V}_{\mathbb{F}}$ (satisfying the standard axioms), the only
closed binary operation between vectors themselves is addition. Introducing
another closed binary operation between vectors, which satisfies two-sided
distributivity with the addition and compatibility with scalars (from
$\mathbb{F}$), gives the bilinear product. A vector space $\mathfrak{V}%
_{\mathbb{F}}$ (we do not consider its concrete realization) endowed with a
bilinear product becomes an algebra over a field $\mathbb{F}$ (for further
details and review, see, e.g., \cite{brown,spindler}).

The reduced vectorization (\ref{vf1}) is not a binary homomorphism (see
{Remark} { \ref{rem-mon}}), but can be a polyadic homomorphism,
if we consider a special kind of monomial $m\times m$ matrices, the cyclic
shift weighted matrices $\mathit{M}_{shf}$ of the fixed shape (\ref{zn}).

To show this, we introduce a new product of vectors in a vector space
$\mathfrak{V}_{\mathbb{F}}$.

\begin{Definition}
In $m$-dimensional vector space $\mathfrak{V}_{\mathbb{F}}$, we define $\left(
m+1\right)  $-ary (polyadic) product of vectors%
\begin{align}
&  \mathit{\mu}_{\star}^{\left[  m+1\right]  }\left[  \mathbf{v}\left(
x^{\prime}\right)  ,\mathbf{v}\left(  x^{\prime\prime}\right)  ,\ldots
,\mathbf{v}\left(  x^{\prime\prime\prime}\right)  \mathbf{v}\left(
x^{\prime\prime\prime\prime}\right)  \right]  =\left(  \overset{m+1}%
{\overbrace{x_{1}^{\prime}x_{2}^{\prime\prime}\ldots x_{m}^{\prime\prime
\prime}x_{1}^{\prime\prime\prime\prime}}}\right)  \mathbf{e}_{1}\nonumber\\
&  +\left(  \overset{m+1}{\overbrace{x_{2}^{\prime}x_{3}^{\prime\prime}\ldots
x_{1}^{\prime\prime\prime}x_{2}^{\prime\prime\prime\prime}}}\right)
\mathbf{e}_{2}+\ldots+\left(  \overset{m+1}{\overbrace{x_{m}^{\prime}%
x_{1}^{\prime\prime}\ldots x_{m-1}^{\prime\prime\prime}x_{m}^{\prime
\prime\prime\prime}}}\right)  \mathbf{e}_{m},\ \ \ x_{i}^{\prime}%
,x_{i}^{\prime\prime},\ldots,x_{i}^{\prime\prime\prime}x_{i}^{\prime
\prime\prime\prime}\in\mathbb{F}. \label{mvv}%
\end{align}

\end{Definition}

\begin{Remark}
The polyadic product of vectors (\ref{mvv}) in components is not elementwise,
but cyclic braided-like. Such products appeared in higher regular semigroups
and braid groups \cite{duplij2022}, as well as in semisupermanifold theory
\cite{dup2018}.
\end{Remark}

Consider the reduced vectorization $\operatorname*{vec}\nolimits_{\mathbb{F}%
}^{\times}$ (\ref{vf1}) of the cyclic shift weighted matrices given by
$\mathit{M}_{shf}\left(  x\right)  \mapsto\mathbf{v}\left(  x\right)  $ or
(cf. (\ref{zn}))%
\begin{align}
\mathit{M}_{shf}\left(  x\right)   &  =\left(
\begin{array}
[c]{ccccc}%
0 & x_{1} & \ldots & 0 & 0\\
0 & 0 & x_{2} & \ldots & 0\\
0 & 0 & \ddots & \ddots & \vdots\\
\vdots & \vdots & \ddots & 0 & x_{m-1}\\
x_{m} & 0 & \ldots & 0 & 0
\end{array}
\right)  \nonumber\\[10pt]
&  \mapsto\mathbf{v}\left(  x\right)  =x_{1}\mathbf{e}_{1}+x_{2}\mathbf{e}%
_{2}+\ldots+x_{m-1}\mathbf{e}_{1}+x_{m}\mathbf{e}_{m},\ x_{i}\in
\mathbb{F}.\label{ms}%
\end{align}

\begin{Proposition}
The reduced vectorization $\operatorname*{vec}\nolimits_{\mathbb{F}}^{\times}$
of the cyclic shift weighted matrices (\ref{ms}) is a polyadic or $\left(
m+1\right)  $-ary homomorphism.
\end{Proposition}

\begin{proof}
The product of $\left(  m+1\right)  $ matrices $\mathit{M}_{shf}\left(
x\right)  $ (but not fewer) is a cyclic shift weighted matrix of the form such
that the set $\left\{  \mathit{M}_{shf}\left(  x\right)  \right\}  $ is a
nonderived $\left(  m+1\right)  $-ary semigroup (cf. (\ref{cy}))%
\begin{align}
&  \overset{m+1}{\overbrace{\mathit{M}_{shf}\left(  x^{\prime}\right)
\mathit{M}_{shf}\left(  x^{\prime\prime}\right)  \ldots\mathit{M}_{shf}\left(
x^{\prime\prime\prime}\right)  \mathit{M}_{shf}\left(  x^{\prime\prime
\prime\prime}\right)  }}\nonumber\\
&  =\left(
\begin{array}
[c]{ccccc}%
0 & \overset{m+1}{\overbrace{x_{1}^{\prime}x_{2}^{\prime\prime}\ldots
x_{m}^{\prime\prime\prime}x_{1}^{\prime\prime\prime\prime}}} & \ldots & 0 &
0\\
0 & 0 & \overset{m+1}{\overbrace{x_{2}^{\prime}x_{3}^{\prime\prime}\ldots
x_{1}^{\prime\prime\prime}x_{2}^{\prime\prime\prime\prime}}} & \ldots & 0\\
0 & 0 & \ddots & \ddots & \vdots\\
\vdots & \vdots & \ddots & 0 & \overset{m+1}{\overbrace{x_{m-1}^{\prime}%
x_{m}^{\prime\prime}\ldots x_{m-2}^{\prime\prime\prime}x_{m-1}^{\prime
\prime\prime\prime}}}\\
\overset{m+1}{\overbrace{x_{m}^{\prime}x_{1}^{\prime\prime}\ldots
x_{m-1}^{\prime\prime\prime}x_{m}^{\prime\prime\prime\prime}}} & 0 & \ldots &
0 & 0
\end{array}
\right)  \label{mmm}\\
&  \mapsto\mathit{\mu}_{\star}^{\left[  m+1\right]  }\left[  \mathbf{v}\left(
x^{\prime}\right)  ,\mathbf{v}\left(  x^{\prime\prime}\right)  ,\ldots
,\mathbf{v}\left(  x^{\prime\prime\prime}\right)  \mathbf{v}\left(
x^{\prime\prime\prime\prime}\right)  \right]  .\nonumber
\end{align}

\hl{The} statement now follows from comparing (\ref{mmm}) with (\ref{mvv}) and
(\ref{ms}).
\end{proof}

\begin{Definition}
A vector space $\mathfrak{V}_{\mathbb{F}}$ equipped with the $\left(
m+1\right)  $-ary (polyadic) product of vectors (\ref{mvv}) becomes $\left(
m+1\right)  $-ary algebra $\mathfrak{A}^{\left[  m+1\right]  }\left(
\mathbb{F}\right)  $ over a field $\mathbb{F}$.
\end{Definition}

\begin{Proposition}
The $\left(  m+1\right)  $-ary algebra $\mathfrak{A}^{\left[  m+1\right]
}\left(  \mathbb{F}\right)  $ is totally (polyadic) associative.
\end{Proposition}

\begin{proof}
It follows from the associativity of binary matrix multiplication
$\mathit{M}_{shf}\left(  x\right)  $ (\ref{mmm}) and the definition of reduced
vectorization (\ref{ms}).
\end{proof}

We define the structure constants of $\mathfrak{A}^{\left[  m+1\right]
}\left(  \mathbb{F}\right)  $ through the basis vectors $\mathbf{e}_{i}$ by%
\begin{equation}
\mathit{\mu}_{\star}^{\left[  m+1\right]  }\left[  \mathbf{e}_{i_{1}%
},\mathbf{e}_{i_{2}},\ldots,\mathbf{e}_{i_{m}},\mathbf{e}_{i_{m+1}}\right]
=\sum_{j=1}^{m}f_{i_{1},i_{2},\ldots,i_{m},i_{m+1}}^{\ j}\mathbf{e}%
_{j},\ \ \ \ \ \ \ \ f_{i_{1},i_{2},\ldots,i_{m},i_{m+1}}^{\ j}\in\mathbb{F}.
\label{me}%
\end{equation}

\begin{Theorem}
The structure constants of the $\left(  m+1\right)  $-ary algebra
$\mathfrak{A}^{\left[  m+1\right]  }\left(  \mathbb{F}\right)  $ are%
\begin{equation}
f_{1,2,3,\ldots,m-1,m,1}^{\ 1}=1,\ \ f_{2,3,4,\ldots m,1,2}^{\ 2}%
=1,\ldots,f_{m-1,m,1,2\ldots,m-3,m-2,m-1}^{\ m-1}=1,\ \ f_{m,1,2\ldots
m-2,m-1,m}^{\ m}=1,
\end{equation}

while $f_{i_{1},i_{2},\ldots,i_{m},i_{m+1}}^{\ j}$ having other combinations
of indices are zero.
\end{Theorem}

\begin{proof}
It follows from the explicit form of the $\left(  m+1\right)  $-ary product
(\ref{mvv}) and from the multilinearity of $\mathit{\mu}_{\star}^{\left[
m+1\right]  }$, that is, the polyadic distributivity with addition (on each
place)
\begin{align}
&  \mathit{\mu}_{\star}^{\left[  m+1\right]  }\left[  \mathbf{v}_{1}\left(
x^{\prime}\right)  +\mathbf{v}_{2}\left(  x^{\prime}\right)  ,\mathbf{v}%
\left(  x^{\prime\prime}\right)  ,\ldots,\mathbf{v}\left(  x^{\prime
\prime\prime}\right)  ,\mathbf{v}\left(  x^{\prime\prime\prime\prime}\right)
\right] \nonumber\\
&  =\mathit{\mu}_{\star}^{\left[  m+1\right]  }\left[  \mathbf{v}_{1}\left(
x^{\prime}\right)  ,\mathbf{v}\left(  x^{\prime\prime}\right)  ,\ldots
,\mathbf{v}\left(  x^{\prime\prime\prime}\right)  ,\mathbf{v}\left(
x^{\prime\prime\prime\prime}\right)  \right]  +\mathit{\mu}_{\star}^{\left[
m+1\right]  }\left[  \mathbf{v}_{2}\left(  x^{\prime}\right)  ,\mathbf{v}%
\left(  x^{\prime\prime}\right)  ,\ldots,\mathbf{v}\left(  x^{\prime
\prime\prime}\right)  ,\mathbf{v}\left(  x^{\prime\prime\prime\prime}\right)
\right]  , \label{md}%
\end{align}

and compatibility with scalars%
 \begin{align}
&  \mathit{\mu}_{\star}^{\left[  m+1\right]  }\left[  \lambda^{\prime
}\mathbf{v}\left(  x^{\prime}\right)  ,\lambda^{\prime\prime}\mathbf{v}\left(
x^{\prime\prime}\right)  ,\ldots,\lambda^{\prime\prime\prime}\mathbf{v}\left(
x^{\prime\prime\prime}\right)  ,\lambda^{\prime\prime\prime\prime}%
\mathbf{v}\left(  x^{\prime\prime\prime\prime}\right)  \right] \nonumber\\
&  =\left(  \lambda^{\prime}\lambda^{\prime\prime}\ldots\lambda^{\prime
\prime\prime}\lambda^{\prime\prime\prime\prime}\right)  \mathit{\mu}_{\star
}^{\left[  m+1\right]  }\left[  \mathbf{v}_{1}\left(  x^{\prime}\right)
,\mathbf{v}\left(  x^{\prime\prime}\right)  ,\ldots,\mathbf{v}\left(
x^{\prime\prime\prime}\right)  ,\mathbf{v}\left(  x^{\prime\prime\prime\prime
}\right)  \right]  ,\ \ \ \ \ \lambda^{\prime},\lambda^{\prime\prime}%
,\ldots\lambda^{\prime\prime\prime},\lambda^{\prime\prime\prime\prime}%
\in\mathbb{F}. \label{mc}%
\end{align}

We insert the expansions (\ref{fv}) into (\ref{mvv}) and then use the
definition of the structure constants (\ref{me}) taking into account
(\ref{md}) and (\ref{mc}) to obtain the result.
\end{proof}

\begin{Corollary}
If $\mathit{M}_{shf}\left(  x\right)  $ is monomial (having no zero entries
$x_{i}\in\mathbb{F\setminus}\left\{  0\right\}  $), the vectorization
$\operatorname*{vec}\nolimits_{\mathbb{F}}^{\times}$ (\ref{vf1}) becomes a
polyadic isomorphism.
\end{Corollary}

\begin{Proposition}
If $x_{i}\in\mathbb{F\setminus}\left\{  0\right\}  $, the set of vectors with
multiplication (\ref{mvv}) becomes a nonderived $\left(  m+1\right)  $-ary
group with the quervector $\mathbf{\tilde{v}}$ (\ref{mgg}) (the polyadic
analog the \textquotedblleft inverse\textquotedblright) having the form%
\begin{equation}
\widetilde{\mathbf{v}}\left(  x\right)  =\left(  \frac{1}{x_{2}x_{3}\ldots
x_{m}}\right)  \mathbf{e}_{1}+\left(  \frac{1}{x_{3}x_{4}\ldots x_{m}x_{1}%
}\right)  \mathbf{e}_{2}+\ldots+\left(  \frac{1}{x_{m}x_{1}\ldots x_{m-1}%
}\right)  \mathbf{e}_{m},\ x_{i}\in\mathbb{F\setminus}\left\{  0\right\}  .
\label{vx}%
\end{equation}

\end{Proposition}

\begin{proof}
To obtain the quervector $\mathbf{\tilde{v}}$, we use the nonderived $\left(
m+1\right)  $-ary product (\ref{mvv}) and the definition (\ref{mgg}), also
equations (\ref{zd}) in the notation (\ref{mmm}).
\end{proof}

\begin{Corollary}
The corresponding nonderived $\left(  m+1\right)  $-ary algebra $\mathfrak{A}%
^{\left[  m+1\right]  }\left(  \mathbb{F}\right)  $ with $x_{i}\in
\mathbb{F\setminus}\left\{  0\right\}  $ becomes a division algebra
$\mathfrak{A}_{\operatorname{div}}^{\left[  m+1\right]  }\left(
\mathbb{F}\right)  $ which is polyadically isomorphic to the division algebra
of monomial $m\times m$ matrices, and the cyclic shift weighted matrices
(\ref{ms}), see (\ref{zn}) and {Theorem \ref{theor-div}}.
\end{Corollary}

Let us consider a simple example.

\begin{Example}
In any three-dimensional vector space $\mathfrak{V}_{\mathbb{R}_{3}}$ over
$\mathbb{R}_{3}$ (of any nature, with no additional structures, such as an
inner product, etc., needed), we can define the nonderived associative $4$-ary
product of vectors%
\begin{equation}
\mathit{\mu}_{\star}^{\left[  4\right]  }\left[  \mathbf{v}\left(  x^{\prime
}\right)  ,\mathbf{v}\left(  x^{\prime\prime}\right)  ,\mathbf{v}\left(
x^{\prime\prime\prime}\right)  ,\mathbf{v}\left(  x^{\prime\prime\prime\prime
}\right)  \right]  =\left(  x_{1}^{\prime}x_{2}^{\prime\prime}x_{3}%
^{\prime\prime\prime}x_{1}^{\prime\prime\prime\prime}\right)  \mathbf{e}%
_{1}+\left(  x_{2}^{\prime}x_{3}^{\prime\prime}x_{1}^{\prime\prime\prime}%
x_{2}^{\prime\prime\prime\prime}\right)  \mathbf{e}_{2}+\left(  x_{3}^{\prime
}x_{1}^{\prime\prime}x_{2}^{\prime\prime\prime}x_{3}^{\prime\prime\prime
\prime}\right)  \mathbf{e}_{3}. \label{m4v}%
\end{equation}

\hl{The} vector space $\mathfrak{V}_{\mathbb{R}_{3}}$ equipped with the product
(\ref{m4v}) becomes a nonderived $4$-ary algebra $\mathfrak{A}^{\left[
4\right]  }\left(  \mathbb{R}\right)  $ over $\mathbb{R}$. Its nonzero
structure constants (\ref{me}) are%
\begin{equation}
f_{1231}^{\ 1}=1,\ \ \ \ f_{2312}^{\ 2}=1,\ \ \ \ f_{3123}^{\ 3}=1.
\end{equation}

In the general case, $\mathfrak{A}^{\left[  4\right]  }\left(  \mathbb{R}%
\right)  $ contains zero divisors, but for all nonzero coordinates $x_{i}%
\in\mathbb{R\setminus}\left\{  0\right\}  $, the algebra becomes the division
algebra $\mathfrak{A}_{\operatorname{div}}^{\left[  4\right]  }\left(
\mathbb{R}\right)  $. The polyadic unit is $\mathbf{v}\left(  1\right)
=\mathbf{e}_{1}+\mathbf{e}_{2}+\mathbf{e}_{3}$, and each element
$\mathbf{v}\left(  x\right)  $ has a quervector $\widetilde{\mathbf{v}}\left(
x\right)  $ (polyadic analog of the binary inverse) determined by the equation
(followed from (\ref{mgg}))%
\begin{equation}
\mathit{\mu}_{\star}^{\left[  4\right]  }\left[  \mathbf{v}\left(  x\right)
,\mathbf{v}\left(  x\right)  ,\mathbf{v}\left(  x\right)  ,\widetilde
{\mathbf{v}}\left(  x\right)  \right]  =\mathbf{v}\left(  x\right)  ,
\label{m4q}%
\end{equation}
such that%
\begin{equation}
\widetilde{\mathbf{v}}\left(  x\right)  =\frac{1}{x_{2}x_{3}}\mathbf{e}%
_{1}+\frac{1}{x_{3}x_{1}}\mathbf{e}_{2}+\frac{1}{x_{1}x_{2}}\mathbf{e}%
_{3},\ \ \ x_{i}\in\mathbb{R\setminus}\left\{  0\right\}  .
\end{equation}

As a simple numerical example, take four vectors with all nonvanishing
coordinates, then each one has a quervector with respect to the $4$-ary
product of vectors $\widetilde{\mathbf{v}}$ (\ref{m4q})%
\begin{align}
\mathbf{v}^{\prime}  &  =\mathbf{v}\left(  x^{\prime}\right)  =2\mathbf{e}%
_{1}+3\mathbf{e}_{2}+4\mathbf{e}_{3},\ \ \ \ \ \ \widetilde{\mathbf{v}%
}^{\prime}=\frac{1}{12}\mathbf{e}_{1}+\frac{1}{8}\mathbf{e}_{2}+\frac{1}%
{6}\mathbf{e}_{3},\label{v1}\\
\mathbf{v}^{\prime\prime}  &  =\mathbf{v}\left(  x^{\prime\prime}\right)
=3\mathbf{e}_{1}+1\mathbf{e}_{2}+5\mathbf{e}_{3},\ \ \ \ \ \ \widetilde
{\mathbf{v}}^{\prime\prime}=\frac{1}{5}\mathbf{e}_{1}+\frac{1}{15}%
\mathbf{e}_{2}+\frac{1}{3}\mathbf{e}_{3},\\
\mathbf{v}^{\prime\prime\prime}  &  =\mathbf{v}\left(  x^{\prime\prime\prime
}\right)  =4\mathbf{e}_{1}+3\mathbf{e}_{2}+5\mathbf{e}_{3},\ \ \ \text{\ \ }%
\ \widetilde{\mathbf{v}}^{\prime\prime\prime}=\frac{1}{15}\mathbf{e}_{1}%
+\frac{1}{20}\mathbf{e}_{2}+\frac{1}{12}\mathbf{e}_{3},\\
\mathbf{v}^{\prime\prime\prime\prime}  &  =\mathbf{v}\left(  x^{\prime
\prime\prime\prime}\right)  =5\mathbf{e}_{1}+3\mathbf{e}_{2}+2\mathbf{e}%
_{3},\ \ \ \ \ \ \widetilde{\mathbf{v}}^{\prime\prime\prime\prime}=\frac{1}%
{6}\mathbf{e}_{1}+\frac{1}{10}\mathbf{e}_{2}+\frac{1}{15}\mathbf{e}_{3}.
\label{v4}%
\end{align}

The $4$-ary product (\ref{m4v}) of the vectors (\ref{v1})--(\ref{v4}) is%
\begin{equation}
\mathit{\mu}_{\star}^{\left[  4\right]  }\left[  \mathbf{v}^{\prime
},\mathbf{v}^{\prime\prime},\mathbf{v}^{\prime\prime\prime},\mathbf{v}%
^{\prime\prime\prime\prime}\right]  =50\mathbf{e}_{1}+180\mathbf{e}%
_{2}+75\mathbf{e}_{3}.
\end{equation}

\end{Example}

Thus, we have shown that, in any $m$-dimensional vector space $\mathfrak{V}%
_{\mathbb{F}}$ over a field $\mathbb{F}$, we can introduce a new nonderived
polyadic (or $\left(  m+1\right)  $-ary) product of vectors $\mathit{\mu
}_{\star}^{\left[  m+1\right]  }$ (\ref{mvv}). The corresponding $\left(
m+1\right)  $-ary algebra is associative and becomes a division algebra for
all vectors with nonvanishing coordinates, and it is isomorphic to the
polyadic division algebras obtained by the matrix polyadization procedure
({Section \ref{sec-polydiv}}).

\section{\textsc{Polyadic Imaginary Division Algebras}}

Here, we introduce a non-matrix polyadization procedure which allows us to
obtain ternary division algebras from ordinary binary normed division
algebras. Let us exploit the notations of the concrete hyperembedding approach
from {Section} \ref{subsec-concrete}. First, we show that some
version of ternary division algebra structure can be obtained by a new
iterative process without introducing additional variables.

\begin{Definition}
We define the ternary \textquotedblleft imaginary tower\textquotedblright\ of
algebras by%
\begin{equation}
\overline{\mathbb{A}}_{CD,\ell+1}^{\left[  3\right]  }=\mathbb{A}_{CD}\left(
\mathit{i}_{\ell}\right)  \cdot\mathit{i}_{\ell+1}=\left\langle \left\{
\mathit{z}_{\left(  \ell\right)  }\left(  \mathit{i}_{\ell}\right)
\cdot\mathit{i}_{\ell+1}\right\}  \mid\left(  +\right)  ,\overline{\mu}%
_{\ell+1}^{\left[  3\right]  }\right\rangle , \mathit{i}_{\ell}^{2}%
=-1,\mathit{i}_{\ell+1}\mathit{i}_{\ell}=-\mathit{i}_{\ell}\mathit{i}_{\ell
+1}, \label{a3}%
\end{equation}
where $\ell\geq0$.
\end{Definition}

The multiplication $\overline{\mu}_{\ell+1}^{\left[  3\right]  }$ in
(\ref{a3}) is nonderived ternary, because $\mathit{i}_{\ell}^{2}%
\nsim\mathit{i}_{\ell}$, but $\mathit{i}_{\ell}^{3}=-\mathit{i}_{\ell}$ for
$\ell\geq1$.

\begin{Theorem}
If the initial algebra $\mathbb{A}_{CD}\left(  \mathit{i}_{\ell}\right)  $ is
a normed division algebra, then its iterated imaginary version $\overline
{\mathbb{A}}_{CD,\ell+1}^{\left[  3\right]  }$ (\ref{a3}) is the ternary
division algebra of the same (initial) dimension $D(\ell)$ and norm.
\end{Theorem}

\begin{Proposition}
The ternary algebras $\overline{\mathbb{A}}_{CD,\ell+1}^{\left[  3\right]  }$
are not subalgebras of the initial algebras $\mathbb{A}_{CD}\left(
\mathit{i}_{\ell}\right)  $.
\end{Proposition}

\begin{proof}
This is because the multiplications in the above algebras have different
arities, despite the underlying sets of imaginary algebras being subsets of
the corresponding initial algebras.
\end{proof}

Let us now present the concrete expressions for the initial division algebras.

\subsection{Complex Number Ternary Division Algebra}

The first algebra (with $\ell=0$) is the ternary division algebra of pure
imaginary complex numbers having the dimension $D\left(  0\right)  =D\left(
\mathbb{R}\right)  =1$%
\begin{equation}
\overline{\mathbb{A}}_{CD,1}^{\left[  3\right]  }\equiv\overline{\mathbb{C}%
}^{\left[  3\right]  }=\mathbb{R}\cdot\mathit{i}_{1},\ \ \mathit{z}_{\left(
1\right)  }=b\mathit{i}_{1}\in\mathbb{C},\ \ \ \mathit{i}_{1}^{2}%
=-1,\ \ b\in\mathbb{R},
\end{equation}
which is unitless. The multiplication in $\overline{\mathbb{C}}^{\left[
3\right]  }$ is ternary nonderived and commutative%
\begin{equation}
\ \overline{\mu}_{1}^{\left[  3\right]  }\left[  \mathit{z}_{\left(  1\right)
}^{\prime},\mathit{z}_{\left(  1\right)  }^{\prime\prime},\mathit{z}_{\left(
1\right)  }^{\prime\prime\prime}\right]  =-b^{\prime}b^{\prime\prime}%
b^{\prime\prime\prime}\mathit{i}_{1}\in\overline{\mathbb{C}}^{\left[
3\right]  }. \label{m3}%
\end{equation}
The norm is the absolute value in $\mathbb{C}$ (module) $\left\Vert
\mathit{z}_{\left(  1\right)  }\right\Vert _{\left(  1\right)  }=\left\vert
b\right\vert $, and it is ternary multiplicative (from (\ref{m3}))%
\begin{equation}
\left\Vert \overline{\mu}_{1}^{\left[  3\right]  }\left[  \mathit{z}_{\left(
1\right)  }^{\prime},\mathit{z}_{\left(  1\right)  }^{\prime\prime}%
,\mathit{z}_{\left(  1\right)  }^{\prime\prime\prime}\right]  \right\Vert
_{\left(  1\right)  }=\left\Vert \mathit{z}_{\left(  1\right)  }^{\prime
}\right\Vert _{\left(  1\right)  }\left\Vert \mathit{z}_{\left(  1\right)
}^{\prime\prime}\right\Vert _{\left(  1\right)  }\left\Vert \mathit{z}%
_{\left(  1\right)  }^{\prime\prime\prime}\right\Vert _{\left(  1\right)  }%
\in\mathbb{R}, \label{mn3}%
\end{equation}
such that the corresponding map $\overline{\mathbb{C}}^{\left[  3\right]
}\rightarrow\mathbb{R}$ is a ternary homomorphism. The querelement of
$\mathit{z}_{\left(  1\right)  }$ (\ref{mgg}) is now defined by%
\begin{equation}
\overline{\mu}_{1}^{\left[  3\right]  }\left[  \mathit{z}_{\left(  1\right)
},\mathit{z}_{\left(  1\right)  },\widetilde{\mathit{z}}_{\left(  1\right)
}\right]  =\mathit{z}_{\left(  1\right)  }, \label{m31}%
\end{equation}
which gives%
\begin{equation}
\widetilde{\mathit{z}}_{\left(  1\right)  }=\frac{1}{\mathit{z}_{\left(
1\right)  }}=-\frac{\mathit{i}_{1}}{b},\ \ \ b\in\mathbb{R\setminus}\left\{
0\right\}  .
\end{equation}

Thus, $\overline{\mathbb{C}}^{\left[  3\right]  }$ is a ternary commutative
algebra, which is indeed a division algebra, because each nonzero element has
a querelement, i.e., it is invertible, so all equations of type (\ref{m31})
with different $\mathit{z}^{\prime}s$ have a solution.

\subsection{Half-Quaternion Ternary Division Algebra}

The next iteration case ($\ell=1$) of the \textquotedblleft imaginary
tower\textquotedblright\ (\ref{a3}) is more complicated and leads to pure
imaginary ternary quaternions of dimension $D\left(  1\right)  =D\left(
\mathbb{C}\right)  =2$%
\begin{align}
\overline{\mathbb{A}}_{CD,2}^{\left[  3\right]  }  &  \equiv\overline
{\mathbb{H}}^{\left[  3\right]  }=\mathbb{C}\left(  \mathit{i}_{1}\right)
\cdot\mathit{i}_{2},\ \ \ \ \ \mathit{z}_{\left(  2\right)  }=\left(
c+d\mathit{i}_{1}\right)  \mathit{i}_{2}=c\mathit{i}_{2}+d\mathit{i}%
_{1}\mathit{i}_{2}\in\overline{\mathbb{H}}^{\left[  3\right]  },\nonumber\\
\mathit{i}_{1}^{2}  &  =\mathit{i}_{2}^{2}=-1,\ \ \ \ \mathit{i}_{1}%
\mathit{i}_{2}=-\mathit{i}_{2}\mathit{i}_{1},\ \ \ \mathit{i}_{1}\in
\mathbb{C},\ \ \ \mathit{i}_{2}\in\mathbb{H\setminus C},\ \ \ \ \ c,d\in
\mathbb{R}. \label{ah3}%
\end{align}

\hl{We} can informally call $\overline{\mathbb{H}}^{\left[  3\right]  }$ the
ternary algebra of imaginary \textquotedblleft
half-quaternions\textquotedblright, because in the standard notation (see
 Example \ref{exam-quat1}) $\mathit{z}_{\left(  2\right)  }$\ from
(\ref{ah3}) is $\mathit{q}=a+b\mathit{i}+c\mathit{j}+d\mathit{k}%
\overset{a=0,b=0}{\longrightarrow}\mathit{q}_{half}=c\mathit{j}+d\mathit{k}$.
The nonderived ternary algebra $\overline{\mathbb{H}}^{\left[  3\right]  }$ is
obviously unitless. The multiplication of the half-quaternions $\overline
{\mathbb{H}}^{\left[  3\right]  }$%
\begin{align}
\ \overline{\mu}_{2}^{\left[  3\right]  }\left[  \mathit{z}_{\left(  2\right)
}^{\prime},\mathit{z}_{\left(  2\right)  }^{\prime\prime},\mathit{z}_{\left(
2\right)  }^{\prime\prime\prime}\right]   &  =\mathit{z}_{\left(  2\right)
}^{\prime}\cdot\mathit{z}_{\left(  2\right)  }^{\prime\prime}\cdot
\mathit{z}_{\left(  2\right)  }^{\prime\prime\prime}=\left(  d^{\prime
}c^{\prime\prime}d^{\prime\prime\prime}-c^{\prime}d^{\prime\prime}%
d^{\prime\prime\prime}-d^{\prime}d^{\prime\prime}c^{\prime\prime\prime
}-c^{\prime}c^{\prime}c^{\prime\prime\prime}\right)  \mathit{i}_{2}\nonumber\\
&  +\left(  c^{\prime}d^{\prime\prime}c^{\prime\prime\prime}-d^{\prime
}c^{\prime\prime}c^{\prime\prime\prime}-c^{\prime}c^{\prime\prime}%
d^{\prime\prime\prime}-d^{\prime}d^{\prime\prime}d^{\prime\prime\prime
}\right)  \mathit{i}_{1}\mathit{i}_{2}\in\overline{\mathbb{H}}^{\left[
3\right]  }, \label{m33}%
\end{align}
is ternary nonderived (i.e., the algebra is closed with respect to the product of three
elements, but not fewer), noncommutative, and totally ternary associative
(\ref{ma})%
\begin{align}
\overline{\mu}_{2}^{\left[  3\right]  }\left[  \overline{\mu}_{2}^{\left[
3\right]  }\left[  \mathit{z}_{\left(  2\right)  }^{\prime},\mathit{z}%
_{\left(  2\right)  }^{\prime\prime},\mathit{z}_{\left(  2\right)  }%
^{\prime\prime\prime}\right]  \mathit{z}_{\left(  2\right)  }^{\prime
\prime\prime\prime},\mathit{z}_{\left(  2\right)  }^{\prime\prime\prime
\prime\prime}\right]   &  =\overline{\mu}_{2}^{\left[  3\right]  }\left[
\mathit{z}_{\left(  2\right)  }^{\prime},\overline{\mu}_{2}^{\left[  3\right]
}\left[  \mathit{z}_{\left(  2\right)  }^{\prime\prime},\mathit{z}_{\left(
2\right)  }^{\prime\prime\prime}\mathit{z}_{\left(  2\right)  }^{\prime
\prime\prime\prime}\right]  ,\mathit{z}_{\left(  2\right)  }^{\prime
\prime\prime\prime\prime}\right] \nonumber\\
&  =\overline{\mu}_{2}^{\left[  3\right]  }\left[  \mathit{z}_{\left(
2\right)  }^{\prime},\mathit{z}_{\left(  2\right)  }^{\prime\prime}%
,\overline{\mu}_{2}^{\left[  3\right]  }\left[  \mathit{z}_{\left(  2\right)
}^{\prime\prime\prime}\mathit{z}_{\left(  2\right)  }^{\prime\prime
\prime\prime},\mathit{z}_{\left(  2\right)  }^{\prime\prime\prime\prime\prime
}\right]  \right]  . \label{aq}%
\end{align}

The norm is defined by the absolute value in the quaternion algebra
$\mathbb{H}$ in the standard way%
\begin{equation}
\left\Vert \mathit{z}_{\left(  2\right)  }\right\Vert _{\left(  2\right)
}=\left\vert \mathit{z}_{\left(  2\right)  }\right\vert =\sqrt{c^{2}+d^{2}}.
\label{n2}%
\end{equation}

The norm (\ref{n2}) is ternary multiplicative%
\begin{equation}
\left\Vert \overline{\mu}_{2}^{\left[  3\right]  }\left[  \mathit{z}_{\left(
2\right)  }^{\prime},\mathit{z}_{\left(  2\right)  }^{\prime\prime}%
,\mathit{z}_{\left(  2\right)  }^{\prime\prime\prime}\right]  \right\Vert
_{\left(  2\right)  }=\left\Vert \mathit{z}_{\left(  2\right)  }^{\prime
}\right\Vert _{\left(  2\right)  }\left\Vert \mathit{z}_{\left(  2\right)
}^{\prime\prime}\right\Vert _{\left(  2\right)  }\left\Vert \mathit{z}%
_{\left(  2\right)  }^{\prime\prime\prime}\right\Vert _{\left(  2\right)  }%
\in\mathbb{R}, \label{m23}%
\end{equation}
such that the corresponding map $\overline{\mathbb{H}}^{\left[  3\right]
}\rightarrow\mathbb{R}$ is a ternary homomorphism.

\begin{Proposition}
The ternary analog of the sum of two squares' identity is%
\begin{align}
&  \left(  d^{\prime}c^{\prime\prime}d^{\prime\prime\prime}-c^{\prime
}d^{\prime\prime}d^{\prime\prime\prime}-c^{\prime}c^{\prime}c^{\prime
\prime\prime}-d^{\prime}d^{\prime\prime}c^{\prime\prime\prime}\right)
^{2}+\left(  c^{\prime}d^{\prime\prime}c^{\prime\prime\prime}-d^{\prime
}c^{\prime\prime}c^{\prime\prime\prime}-c^{\prime}c^{\prime\prime}%
d^{\prime\prime\prime}-d^{\prime}d^{\prime\prime}d^{\prime\prime\prime
}\right)  ^{2}\nonumber\\
&  =\left(  (c^{\prime})^{2}+(d^{\prime})^{2}\right)  \left(  (c^{\prime
\prime})^{2}+(d^{\prime\prime})^{2}\right)  \left(  (c^{\prime\prime\prime
})^{2}+(d^{\prime\prime\prime})^{2}\right)  . \label{s}%
\end{align}

\end{Proposition}

\begin{proof}
It follows from the ternary multiplication formula (\ref{m33}) and the
multiplicativity (\ref{m23}) of the norm (\ref{n2}).
\end{proof}

\begin{Remark}
The ternary sum of two squares' identity (\ref{s}) cannot be derived from the
binary sum of two squares' identity (Diophantus, Fibonacci) or from Euler's sum
of four squares' identity, while it can be considered an intermediate (triple) identity.
\end{Remark}

The querelement of $\mathit{z}_{\left(  2\right)  }$ (\ref{mgg}) is now
defined by%
\begin{equation}
\overline{\mu}_{2}^{\left[  3\right]  }\left[  \mathit{z}_{\left(  2\right)
},\mathit{z}_{\left(  2\right)  },\widetilde{\mathit{z}}_{\left(  2\right)
}\right]  =\mathit{z}_{\left(  2\right)  },
\end{equation}
which gives%
\begin{equation}
\widetilde{\mathit{z}}_{\left(  2\right)  }=-\frac{c\mathit{i}_{2}%
+d\mathit{i}_{1}\mathit{i}_{2}}{c^{2}+d^{2}}\in\overline{\mathbb{H}}^{\left[
3\right]  },\ \ \ c^{2}+d^{2}\neq0,\ \ \ c,d\in\mathbb{R}. \label{zq2}%
\end{equation}
\hl{It} follows from (\ref{zq2}) that $\overline{\mathbb{H}}^{\left[  3\right]  }$
(\ref{ah3}) is a noncommutative nonderived ternary algebra, which is indeed a
division algebra, because each nonzero element $\mathit{z}_{\left(  2\right)
}$ has a querelement $\widetilde{\mathit{z}}_{\left(  2\right)  }$, i.e., it is
invertible, and equations of type (\ref{m31}) have solutions.

\subsection{Half-Octonion Ternary Division Algebra}

The next case ($\ell=2$) gives pure imaginary ternary octonions of dimension
$D\left(  2\right)  =D\left(  \mathbb{H}\right)  =4$%

\begin{align}
\overline{\mathbb{A}}_{CD,3}^{\left[  3\right]  } &  \equiv\overline
{\mathbb{O}}^{\left[  3\right]  }=\mathbb{H}\left(  \mathit{i}_{1}%
,\mathit{i}_{2}\right)  \cdot\mathit{i}_{3},\mathit{z}_{\left(  3\right)
}=\left(  a+b\mathit{i}_{1}+c\mathit{i}_{2}+d\mathit{i}_{1}\mathit{i}%
_{2}\right)  \mathit{i}_{3}=a\mathit{i}_{3}+b\mathit{i}_{1}\mathit{i}%
_{3}+c\mathit{i}_{2}\mathit{i}_{3}+d\mathit{i}_{1}\mathit{i}_{2}\mathit{i}%
_{3},\nonumber\\
\mathit{i}_{1}^{2} &  =\mathit{i}_{2}^{2}=\mathit{i}_{3}^{2}%
=-1,\ \ \ \ \mathit{i}_{1}\mathit{i}_{2}=-\mathit{i}_{2}\mathit{i}%
_{1},\ \ \ \ \mathit{i}_{1}\mathit{i}_{3}=-\mathit{i}_{3}\mathit{i}%
_{1},\ \ \ \ \mathit{i}_{2}\mathit{i}_{3}=-\mathit{i}_{3}\mathit{i}%
_{2},\ \ \ \nonumber\\
\mathit{i}_{1} &  \in\mathbb{C},\ \ \ \mathit{i}_{2}\in\mathbb{H\setminus
C},\ \ \ \mathit{i}_{3}\in\mathbb{O\setminus H},\ \ \ \ \ a,b,c,d\in
\mathbb{R}.\label{o3}%
\end{align}
\hl{We} informally can call $\overline{\mathbb{O}}^{\left[  3\right]  }$ the
nonderived ternary algebra of imaginary ``half-octonions'', which is obviously unitless. In the
standard notation (with seven imaginary units $\mathit{e}_{1}\ldots
\mathit{e}_{7}$), the \textquotedblleft half-octonion\textquotedblright\ is%
\begin{equation}
\mathit{o}_{half}=\mathit{z}_{\left(  3\right)  }=a\mathit{e}_{4}%
+b\mathit{e}_{5}+c\mathit{e}_{6}+d\mathit{e}_{7}.\label{oh}%
\end{equation}

\begin{Remark}
It is well known (see, e.g., \cite{bae2002,dra/man}), that the algebra of
octonions $\mathbb{O}$ is not a field, but a special nonassociative loop
(quasigroup with an identity), the Moufang loop (see, e.g.,
\cite{bru/kle,pum2014}). Because the ternary algebra $\overline{\mathbb{O}%
}^{\left[  3\right]  }$ of imaginary \textquotedblleft
half-octonions\textquotedblright\ is unitless, it cannot be a ternary loop
\cite{ono/urs}.
\end{Remark}

Recall that the binary algebra of ordinary octonions $\mathbb{O=}\left\langle
\left\{  \mathit{z}\right\}  \mid\left(  +\right)  ,\left(  \bullet\right)
\right\rangle $ is not associative, and therefore, a triple product in
$\mathbb{O}$ is not unique, because $\left(  \mathit{z}^{\prime}%
\bullet\mathit{z}^{\prime\prime}\right)  \bullet\mathit{z}^{\prime\prime
\prime}\neq\mathit{z}^{\prime}\bullet\left(  \mathit{z}^{\prime\prime}%
\bullet\mathit{z}^{\prime\prime\prime}\right)  $, $\mathit{z}\in\mathbb{O}$.
We introduce the ternary multiplication for the \textquotedblleft
half-octonions\textquotedblright\ $\mathit{z}_{\left(  3\right)  }\in
\overline{\mathbb{O}}^{\left[  3\right]  }$ (\ref{oh})\ in the unique way as
the following arithmetic mean%
\begin{equation}
\ \overline{\mu}_{3}^{\left[  3\right]  }\left[  \mathit{z}_{\left(  3\right)
}^{\prime},\mathit{z}_{\left(  3\right)  }^{\prime\prime},\mathit{z}_{\left(
3\right)  }^{\prime\prime\prime}\right]  =\dfrac{\left(  \mathit{z}_{\left(
3\right)  }^{\prime}\bullet\mathit{z}_{\left(  3\right)  }^{\prime\prime
}\right)  \bullet\mathit{z}_{\left(  3\right)  }^{\prime\prime\prime
}+\mathit{z}_{\left(  3\right)  }^{\prime}\bullet\left(  \mathit{z}_{\left(
3\right)  }^{\prime\prime}\bullet\mathit{z}_{\left(  3\right)  }^{\prime
\prime\prime}\right)  }{2}. \label{o33}%
\end{equation}

\begin{Proposition}
The nonderived nonunital ternary algebra of imaginary \textquotedblleft
half-octonions\textquotedblright\ (\ref{o3})%
\begin{equation}
\overline{\mathbb{O}}^{\left[  3\right]  }=\left\langle \left\{
\mathit{z}_{\left(  3\right)  }\right\}  \mid\left(  +\right)  ,\overline{\mu
}_{3}^{\left[  3\right]  }\right\rangle
\end{equation}
is closed under the multiplication (\ref{o33}) and ternary associative.
\end{Proposition}

\begin{proof}
The closeness of the product $\overline{\mu}_{3}^{\left[  3\right]  }$ is
obvious, and the ternary total associativity of $\overline{\mu}_{3}^{\left[
3\right]  }$ (similar to (\ref{aq})) follows from (\ref{o33}).
\end{proof}

In components, the multiplication (\ref{o33}) becomes{
\begin{align}
&  \overline{\mu}_{3}^{\left[  3\right]  }\left[  \mathit{z}_{\left(
3\right)  }^{\prime},\mathit{z}_{\left(  3\right)  }^{\prime\prime}%
,\mathit{z}_{\left(  3\right)  }^{\prime\prime\prime}\right] \nonumber\\
&  =\left(  a^{\prime\prime}\left(  b^{\prime}b^{\prime\prime\prime}%
+c^{\prime}c^{\prime\prime\prime}+d^{\prime}d^{\prime\prime\prime}\right)
-a^{\prime\prime\prime}\left(  b^{\prime}b^{\prime\prime}+c^{\prime}%
c^{\prime\prime}+d^{\prime}d^{\prime\prime}\right)  +a^{\prime}\left(
b^{\prime\prime}b^{\prime\prime\prime}+c^{\prime\prime}c^{\prime\prime\prime
}+d^{\prime\prime}d^{\prime\prime\prime}\right)  -a^{\prime}a^{\prime\prime
}a^{\prime\prime\prime}\right)  \mathit{i}_{3}\nonumber\\
&  +\left(  b^{\prime\prime}\left(  a^{\prime}a^{\prime\prime\prime}%
+c^{\prime}c^{\prime\prime\prime}+d^{\prime}d^{\prime\prime\prime}\right)
+b^{\prime\prime\prime}\left(  a^{\prime}a^{\prime\prime}+c^{\prime}%
c^{\prime\prime}+d^{\prime}d^{\prime\prime}\right)  -b^{\prime}\left(
a^{\prime\prime}a^{\prime\prime\prime}+c^{\prime\prime}c^{\prime\prime\prime
}+d^{\prime\prime}d^{\prime\prime\prime}\right)  -b^{\prime}b^{\prime\prime
}b^{\prime\prime\prime}\right)  \mathit{i}_{1}\mathit{i}_{3}\nonumber\\
&  +\left(  c^{\prime\prime}\left(  a^{\prime}a^{\prime\prime\prime}%
+b^{\prime}b^{\prime\prime\prime}+d^{\prime}d^{\prime\prime\prime}\right)
-c^{\prime\prime\prime}\left(  a^{\prime}a^{\prime\prime}+b^{\prime}%
b^{\prime\prime}+d^{\prime}d^{\prime\prime}\right)  -c^{\prime}\left(
a^{\prime\prime}a^{\prime\prime\prime}+b^{\prime\prime}b^{\prime\prime\prime
}+d^{\prime\prime}d^{\prime\prime\prime}\right)  -c^{\prime}c^{\prime\prime
}c^{\prime\prime\prime}\right)  \mathit{i}_{2}\mathit{i}_{3}\nonumber\\
&  +\left(  d^{\prime\prime}\left(  a^{\prime}a^{\prime\prime\prime}%
+b^{\prime}b^{\prime\prime\prime}+c^{\prime}c^{\prime\prime\prime}\right)
-d^{\prime\prime\prime}\left(  a^{\prime}a^{\prime\prime}+b^{\prime}%
b^{\prime\prime}+c^{\prime}c^{\prime\prime}\right)  -d^{\prime}\left(
a^{\prime\prime}a^{\prime\prime\prime}+b^{\prime\prime}b^{\prime\prime\prime
}+c^{\prime\prime}c^{\prime\prime\prime}\right)  -d^{\prime}d^{\prime\prime
}d^{\prime\prime\prime}\right)  \mathit{i}_{1}\mathit{i}_{2}\mathit{i}_{3}.
\label{m3z}%
\end{align}
}

The querelement $\widetilde{\mathit{z}}_{\left(  3\right)  }$ (\ref{mgg}) of
$\mathit{z}_{\left(  3\right)  }$ from $\overline{\mathbb{O}}^{\left[
3\right]  }$ is defined by%
\begin{equation}
\overline{\mu}_{3}^{\left[  3\right]  }\left[  \mathit{z}_{\left(  3\right)
},\mathit{z}_{\left(  3\right)  },\widetilde{\mathit{z}}_{\left(  3\right)
}\right]  =\mathit{z}_{\left(  3\right)  },
\end{equation}
where $\widetilde{\mathit{z}}_{\left(  3\right)  }$ can be on any place. In
components, we obtain (cf. half-quaternions (\ref{zq2}))%
\begin{equation}
\widetilde{\mathit{z}}_{\left(  3\right)  }=-\frac{a\mathit{i}_{3}%
+b\mathit{i}_{1}\mathit{i}_{3}+c\mathit{i}_{2}\mathit{i}_{3}+d\mathit{i}%
_{1}\mathit{i}_{2}\mathit{i}_{3}}{a^{2}+b^{2}+c^{2}+d^{2}}\in\overline
{\mathbb{O}}^{\left[  3\right]  },\ \ \ a^{2}+b^{2}+c^{2}+d^{2}\neq
0,\ \ \ a,b,c,d\in\mathbb{R}.\label{zq3}%
\end{equation}
It follows from (\ref{zq3}) that $\overline{\mathbb{O}}^{\left[  3\right]  }$
is a division algebra, because each nonzero element $\mathit{z}_{\left(
3\right)  }$ has a querelement $\widetilde{\mathit{z}}_{\left(  3\right)  }$,
i.e., it is invertible.

Thus, we have constructed the noncommutative nonderived ternary division
algebra of half-octonions $\overline{\mathbb{O}}^{\left[  3\right]  }$, which
is nonunital and totally associative.

\medskip

\textbf{Acknowledgements.} The author is grateful to Mike Hewitt, Thomas Nordahl, Susanne Pumpl\"un, Vladimir Tkach, and Raimund Vogl for useful discussions and valuable help.

\pagestyle{emptyf}
\mbox{}


\vskip 1.5cm
\begin{thebibliography}{}

\bibitem[\protect\citeauthoryear{Abramov, Kerner, and
  Le$\;$Roy}{\textcolor{blue}{\sc Abramov et~al.}}{1997}]{abr/ker/roy}
{\textcolor{blue}{\sc Abramov, V., R.~Kerner, and B.~Le$\;$Roy}} (1997).
\newblock Hypersymmetry: a ${{\Bbb Z}}_3$-graded generalization of
  supersymmetry.
\newblock {\em J.~Math. Phys.\/}~{\bf 38}, 1650--1669.

\bibitem[\protect\citeauthoryear{Albert}{\textcolor{blue}{\sc
  Albert}}{1942}]{alb42}
{\textcolor{blue}{\sc Albert, A.~A.}} (1942).
\newblock Quadratic forms permitting composition.
\newblock {\em Ann. Math.\/}~{\bf 43} (1), 161--177.

\bibitem[\protect\citeauthoryear{Ataguema, Makhlouf, and
  Silvestrov}{\textcolor{blue}{\sc Ataguema et~al.}}{2008}]{ata/mak/sil}
{\textcolor{blue}{\sc Ataguema, H., A.~Makhlouf, and S.~Silvestrov}} (2008).
\newblock Generalization of $n$-ary {N}ambu algebras and beyond.
\newblock math.RA/0812.4058.

\bibitem[\protect\citeauthoryear{Baez}{\textcolor{blue}{\sc
  Baez}}{2002}]{bae2002}
{\textcolor{blue}{\sc Baez, J.~C.}} (2002).
\newblock The octonions.
\newblock {\em Bull. Am. Math. Soc.\/}~{\bf 39} (2), 145--205.

\bibitem[\protect\citeauthoryear{Bagger and Lambert}{\textcolor{blue}{\sc
  Bagger and Lambert}}{2008}]{bag/lam1}
{\textcolor{blue}{\sc Bagger, J. and N.~Lambert}} (2008).
\newblock Gauge symmetry and supersymmetry of multiple {M2}-branes.
\newblock {\em Phys. Rev.\/}~{\bf D77}, 065008.

\bibitem[\protect\citeauthoryear{Bars and G\"unaydin}{\textcolor{blue}{\sc Bars
  and G\"unaydin}}{1979}]{bar/gun1}
{\textcolor{blue}{\sc Bars, I. and M.~G\"unaydin}} (1979).
\newblock Construction of {L}ie algebras and {L}ie superalgebras from ternary
  algebras.
\newblock {\em J. Math. Phys.\/}~{\bf 20}, 1977--1985.

\bibitem[\protect\citeauthoryear{Belousov}{\textcolor{blue}{\sc
  Belousov}}{1972}]{belousov}
{\textcolor{blue}{\sc Belousov, V.~D.}} (1972).
\newblock {\em $n$-{A}ry Quasigroups}.
\newblock Kishinev: Shtintsa.

\bibitem[\protect\citeauthoryear{Borevich and Shafarevich}{\textcolor{blue}{\sc
  Borevich and Shafarevich}}{1966}]{bor/sch}
{\textcolor{blue}{\sc Borevich, A.~I. and I.~R. Shafarevich}} (1966).
\newblock {\em Number Theory}.
\newblock New York-London.

\bibitem[\protect\citeauthoryear{Brown}{\textcolor{blue}{\sc
  Brown}}{1991}]{brown}
{\textcolor{blue}{\sc Brown, W.~C.}} (1991).
\newblock {\em Matrices and Vector Spaces}.
\newblock New York: Marcel Dekker.

\bibitem[\protect\citeauthoryear{Bruck and Kleinfeld}{\textcolor{blue}{\sc
  Bruck and Kleinfeld}}{1951}]{bru/kle}
{\textcolor{blue}{\sc Bruck, R.~H. and E.~Kleinfeld}} (1951).
\newblock The structure of alternative division rings.
\newblock {\em Proc. Amer. Math. Soc.\/}~{\bf 2} (6), 878--890.

\bibitem[\protect\citeauthoryear{Burlakov and Burlakov}{\textcolor{blue}{\sc
  Burlakov and Burlakov}}{2020}]{bur/bur}
{\textcolor{blue}{\sc Burlakov, I.~M. and M.~P. Burlakov}} (2020).
\newblock Geometric structures over hypercomplex algebras.
\newblock {\em J. Math. Sci.\/}~{\bf 245} (4), 538--552.

\bibitem[\protect\citeauthoryear{Carlsson}{\textcolor{blue}{\sc
  Carlsson}}{1980}]{car4}
{\textcolor{blue}{\sc Carlsson, R.}} (1980).
\newblock ${N}$-ary algebras.
\newblock {\em Nagoya Math. J.\/}~{\bf 78} (1), 45--56.

\bibitem[\protect\citeauthoryear{Castro}{\textcolor{blue}{\sc
  Castro}}{2010}]{cast1}
{\textcolor{blue}{\sc Castro, C.}} (2010).
\newblock On $n$-ary algebras, branes and polyvector gauge theories in
  noncommutative {C}lifford spaces.
\newblock {\em J. Phys.\/}~{\bf A43}, 365201.

\bibitem[\protect\citeauthoryear{Chien and Nakazato}{\textcolor{blue}{\sc Chien
  and Nakazato}}{2013}]{chi/nak}
{\textcolor{blue}{\sc Chien, M.-T. and H.~Nakazato}} (2013).
\newblock Hyperbolic forms associated with cyclic weighted shift matrices.
\newblock {\em Linear Algebra Appl.\/}~{\bf 439} (11), 3541--3554.

\bibitem[\protect\citeauthoryear{de$\;$Azcarraga and
  Izquierdo}{\textcolor{blue}{\sc de$\;$Azcarraga and
  Izquierdo}}{2010}]{azc/izq}
{\textcolor{blue}{\sc de$\;$Azcarraga, J. and J.~M. Izquierdo}} (2010).
\newblock $n$-{A}ry algebras: {A} review with applications.
\newblock {\em J. Phys.\/}~{\bf A43}, 293001.

\bibitem[\protect\citeauthoryear{Dickson}{\textcolor{blue}{\sc
  Dickson}}{1919}]{dic19}
{\textcolor{blue}{\sc Dickson, L.~E.}} (1919).
\newblock On quaternions and their generalization and the history of the eight
  square theorem.
\newblock {\em Ann. Math.\/}~{\bf 20} (3), 155--171.

\bibitem[\protect\citeauthoryear{D\"ornte}{\textcolor{blue}{\sc
  D\"ornte}}{1929}]{dor3}
{\textcolor{blue}{\sc D\"ornte, W.}} (1929).
\newblock Unterschungen \"uber einen verallgemeinerten {G}ruppenbegriff.
\newblock {\em Math. Z.\/}~{\bf 29}, 1--19.

\bibitem[\protect\citeauthoryear{Dray and Manogue}{\textcolor{blue}{\sc Dray
  and Manogue}}{2015}]{dra/man}
{\textcolor{blue}{\sc Dray, T. and C.~A. Manogue}} (2015).
\newblock {\em The geometry of the octonions}.
\newblock Hackensack, NJ: World Scientific.

\bibitem[\protect\citeauthoryear{Dubrovski and Volkov}{\textcolor{blue}{\sc
  Dubrovski and Volkov}}{2007}]{dub/vol}
{\textcolor{blue}{\sc Dubrovski, A. and G.~Volkov}} (2007).
\newblock Ternary numbers and algebras: {R}eflexive numbers and {B}erger
  graphs.
\newblock {\em Adv. Appl. Clifford Algebras\/}~{\bf 17}, 159--181.

\bibitem[\protect\citeauthoryear{Duplij}{\textcolor{blue}{\sc
  Duplij}}{2018}]{dup2018}
{\textcolor{blue}{\sc Duplij, S.}} (2018).
\newblock {\em Exotic Algebraic and Geometric Structures in Theoretical
  Physics}.
\newblock New York: Nova Publishers.

\bibitem[\protect\citeauthoryear{Duplij}{\textcolor{blue}{\sc
  Duplij}}{2022}]{duplij2022}
{\textcolor{blue}{\sc Duplij, S.}} (2022).
\newblock {\em Polyadic Algebraic Structures}.
\newblock Bristol: IOP Publishing.

\bibitem[\protect\citeauthoryear{Faulkner}{\textcolor{blue}{\sc
  Faulkner}}{1995}]{fau95}
{\textcolor{blue}{\sc Faulkner, J.~R.}} (1995).
\newblock Higher order invertibility in {J}ordan pairs.
\newblock {\em Comm. Algebra\/}~{\bf 23} (9), 3429--3446.

\bibitem[\protect\citeauthoryear{Filippov}{\textcolor{blue}{\sc
  Filippov}}{1985}]{fil0}
{\textcolor{blue}{\sc Filippov, V.~T.}} (1985).
\newblock $n$-{L}ie algebras.
\newblock {\em Sib. Math. J.\/}~{\bf 26}, 879--891.

\bibitem[\protect\citeauthoryear{Flaut}{\textcolor{blue}{\sc
  Flaut}}{2019}]{fla2019}
{\textcolor{blue}{\sc Flaut, C.}} (2019).
\newblock Some remarks about the levels and sublevels of algebras obtained by
  the {Cayley}-{Dickson} process.
\newblock {\em Algebr. Represent. Th.\/}~{\bf 22} (3), 651--663.

\bibitem[\protect\citeauthoryear{Gal'mak}{\textcolor{blue}{\sc
  Gal'mak}}{2003}]{galmak1}
{\textcolor{blue}{\sc Gal'mak, A.~M.}} (2003).
\newblock {\em $n$-{A}ry Groups, Part 1}.
\newblock Gomel: Gomel University.

\bibitem[\protect\citeauthoryear{Gau, Tsai, and Wang}{\textcolor{blue}{\sc Gau
  et~al.}}{2013}]{gau/tsai/wang}
{\textcolor{blue}{\sc Gau, H.-L., M.-C. Tsai, and H.-C. Wang}} (2013).
\newblock Weighted shift matrices: Unitary equivalence, reducibility and
  numerical ranges.
\newblock {\em Linear Algebra and its Applications\/}~{\bf 438} (1), 498--513.

\bibitem[\protect\citeauthoryear{Gleichgewicht and
  G{\l}azek}{\textcolor{blue}{\sc Gleichgewicht and G{\l}azek}}{1967}]{gle/gla}
{\textcolor{blue}{\sc Gleichgewicht, B. and K.~G{\l}azek}} (1967).
\newblock Remarks on $n$-groups as abstract algebras.
\newblock {\em Colloq. Math.\/}~{\bf 17}, 209--219.

\bibitem[\protect\citeauthoryear{Gubareni}{\textcolor{blue}{\sc
  Gubareni}}{2021}]{gubareni}
{\textcolor{blue}{\sc Gubareni, N.}} (2021).
\newblock {\em Introduction To Modern Algebra And Its Applications}.
\newblock Boca Raton: CRC Press.

\bibitem[\protect\citeauthoryear{Hackbusch}{\textcolor{blue}{\sc
  Hackbusch}}{2019}]{hackbusch}
{\textcolor{blue}{\sc Hackbusch, W.}} (2019).
\newblock {\em Tensor Spaces and Numerical Tensor Calculus}.
\newblock Springer.

\bibitem[\protect\citeauthoryear{Hawkes}{\textcolor{blue}{\sc
  Hawkes}}{1902}]{haw1902}
{\textcolor{blue}{\sc Hawkes, H.~E.}} (1902).
\newblock On hypercomplex number systems.
\newblock {\em Trans. Amer. Math. Soc.\/}~{\bf 3} (3), 312--330.

\bibitem[\protect\citeauthoryear{Joyner}{\textcolor{blue}{\sc
  Joyner}}{2008}]{joyner}
{\textcolor{blue}{\sc Joyner, D.}} (2008).
\newblock {\em Adventures in Group Theory}.
\newblock Baltimore: Johns Hopkins University Press.

\bibitem[\protect\citeauthoryear{Kantor and Solodovnikov}{\textcolor{blue}{\sc
  Kantor and Solodovnikov}}{1989}]{kan/sol}
{\textcolor{blue}{\sc Kantor, I.~L. and A.~S. Solodovnikov}} (1989).
\newblock {\em Hypercomplex Numbers}.
\newblock New York: Springer-Verlag.

\bibitem[\protect\citeauthoryear{Kerner}{\textcolor{blue}{\sc
  Kerner}}{2000}]{ker2000}
{\textcolor{blue}{\sc Kerner, R.}} (2000).
\newblock Ternary algebraic structures and their applications in physics.
\newblock In: {\textcolor{blue}{\sc V.~Dobrev, A.~Inomata, G.~Pogosyan,
  L.~Mardoyan, and A.~Sisakyan}} (Eds.), {\em 23rd International Colloquium on
  Group Theoretical Methods in Physics}, Dubna: JINR Publishing.
\newblock arXiv preprint: math-ph/0011023.

\bibitem[\protect\citeauthoryear{Kerner}{\textcolor{blue}{\sc
  Kerner}}{2012}]{ker2012}
{\textcolor{blue}{\sc Kerner, R.}} (2012).
\newblock A ${Z}_{3}$ generalization of {P}auli's principle, quark algebra and
  the {L}orentz invariance.
\newblock In: {\textcolor{blue}{\sc J.~{Alves Rodrigues}, Waldyr, R.~{Kerner},
  G.~O. {Pires}, and C.~{Pinheiro}}} (Eds.), {\em The Sixth International
  School on Field Theory and Gravitation}, Vol. 1483 of {\em AIP Conference
  Series}, pp.\  144--168.

\bibitem[\protect\citeauthoryear{Kim, Lee, Kim, and Ahn}{\textcolor{blue}{\sc
  Kim et~al.}}{2004}]{kim/lee2004}
{\textcolor{blue}{\sc Kim, K.~S., S.~H. Lee, Y.~H. Kim, and J.~Y. Ahn}} (2004).
\newblock Design of binary {LDPC} code using cyclic shift matrices.
\newblock {\em Electron. Lett\/}~{\bf 40} (5), 325--326.

\bibitem[\protect\citeauthoryear{Kornilowicz}{\textcolor{blue}{\sc
  Kornilowicz}}{2012}]{kor2012}
{\textcolor{blue}{\sc Kornilowicz, A.}} (2012).
\newblock Cayley-{D}ickson construction.
\newblock {\em Formaliz. Math.\/}~{\bf 20} (4), 281--290.

\bibitem[\protect\citeauthoryear{Lipatov, Rausch$\:$de$\;${T}raubenberg, and
  Volkov}{\textcolor{blue}{\sc Lipatov et~al.}}{2008}]{lip/rau/vol}
{\textcolor{blue}{\sc Lipatov, L.~N., M.~Rausch$\:$de$\;${T}raubenberg, and
  G.~G. Volkov}} (2008).
\newblock On the ternary complex analysis and its applications.
\newblock {\em J. Math. Phys.\/}~{\bf 49}, 013502.

\bibitem[\protect\citeauthoryear{Loos}{\textcolor{blue}{\sc
  Loos}}{1974}]{loo74}
{\textcolor{blue}{\sc Loos, O.}} (1974).
\newblock A structure theory of {J}ordan pairs.
\newblock {\em Bull. Amer. Math. Soc.\/}~{\bf 80}, 67--71.

\bibitem[\protect\citeauthoryear{Loos}{\textcolor{blue}{\sc Loos}}{1975}]{loos}
{\textcolor{blue}{\sc Loos, O.}} (1975).
\newblock {\em Jordan Pairs}.
\newblock Lecture Notes in Mathematics, Vol. 460. Berlin-New York:
  Springer-Verlag.

\bibitem[\protect\citeauthoryear{Lovett}{\textcolor{blue}{\sc
  Lovett}}{2016}]{lovett}
{\textcolor{blue}{\sc Lovett, S.}} (2016).
\newblock {\em Abstract Algebra: Structures And Applications}.
\newblock Boca Raton, FL: CRC Press.

\bibitem[\protect\citeauthoryear{Magnus and Neudecker}{\textcolor{blue}{\sc
  Magnus and Neudecker}}{2019}]{mag/neu}
{\textcolor{blue}{\sc Magnus, J.~R. and H.~Neudecker}} (2019).
\newblock {\em Matrix Differential Calculus with Applications in Statistics and
  Econometrics}.
\newblock New York: John Wiley.

\bibitem[\protect\citeauthoryear{McCoy}{\textcolor{blue}{\sc
  McCoy}}{1972}]{mccoy}
{\textcolor{blue}{\sc McCoy, N.~H.}} (1972).
\newblock {\em Fundamentals of Abstract Algebra}.
\newblock Boston: Allyn and Bacon.

\bibitem[\protect\citeauthoryear{Michor and Vinogradov}{\textcolor{blue}{\sc
  Michor and Vinogradov}}{1996}]{mic/vin}
{\textcolor{blue}{\sc Michor, P.~W. and A.~M. Vinogradov}} (1996).
\newblock $n$-{A}ry {L}ie and associative algebras.
\newblock {\em Rend. Sem. Mat. Univ. Pol. Torino\/}~{\bf 54} (4), 373--392.

\bibitem[\protect\citeauthoryear{Nambu}{\textcolor{blue}{\sc
  Nambu}}{1973}]{nam0}
{\textcolor{blue}{\sc Nambu, Y.}} (1973).
\newblock Generalized {H}amiltonian dynamics.
\newblock {\em Phys. Rev.\/}~{\bf 7}, 2405--2412.

\bibitem[\protect\citeauthoryear{Neurkich}{\textcolor{blue}{\sc
  Neurkich}}{1999}]{neurkich}
{\textcolor{blue}{\sc Neurkich, J.}} (1999).
\newblock {\em Algebraic Number Theory}.
\newblock Berlin-New York: Springer.

\bibitem[\protect\citeauthoryear{Nikitin}{\textcolor{blue}{\sc
  Nikitin}}{1984}]{nik84a}
{\textcolor{blue}{\sc Nikitin, A.~N.}} (1984).
\newblock Semisimple {Artinian} (2,n)-rings.
\newblock {\em Moscow Univ. Math. Bull.\/}~{\bf 39} (6), 1--6.

\bibitem[\protect\citeauthoryear{Onoi and Ursu}{\textcolor{blue}{\sc Onoi and
  Ursu}}{2019}]{ono/urs}
{\textcolor{blue}{\sc Onoi, V.~I. and L.~A. Ursu}} (2019).
\newblock On generalization of the notion of {M}oufang loop to {$n$}-ary case.
\newblock {\em Bul. Acad. \c{S}tiin\c{t}e Repub. Moldova.
  Matematika\/}~(1(89)), 52--70.

\bibitem[\protect\citeauthoryear{Post}{\textcolor{blue}{\sc Post}}{1940}]{pos}
{\textcolor{blue}{\sc Post, E.~L.}} (1940).
\newblock Polyadic groups.
\newblock {\em Trans. Amer. Math. Soc.\/}~{\bf 48}, 208--350.

\bibitem[\protect\citeauthoryear{Pumpl{\"u}n}{\textcolor{blue}{\sc
  Pumpl{\"u}n}}{2014}]{pum2014}
{\textcolor{blue}{\sc Pumpl{\"u}n, S.}} (2014).
\newblock A note on {Moufang} loops arising from octonion algebras over rings.
\newblock {\em Beitr. Alg. Geom.\/}~{\bf 55} (1), 33--42.

\bibitem[\protect\citeauthoryear{Pumplün}{\textcolor{blue}{\sc
  Pumplün}}{2011}]{pum2011}
{\textcolor{blue}{\sc Pumplün, S.}} (2011).
\newblock Forms of higher degree permitting composition.
\newblock {\em Beitr. Algebra Geom.\/}~{\bf 52} (2), 265--284.

\bibitem[\protect\citeauthoryear{Rotman}{\textcolor{blue}{\sc
  Rotman}}{2010}]{rotman}
{\textcolor{blue}{\sc Rotman, J.~J.}} (2010).
\newblock {\em Advanced Modern Algebra\/} (2nd ed.).
\newblock Providence: AMS.

\bibitem[\protect\citeauthoryear{Saltman}{\textcolor{blue}{\sc
  Saltman}}{1992}]{sal92}
{\textcolor{blue}{\sc Saltman, D.~J.}} (1992).
\newblock Finite-dimensional division algebras.
\newblock In: {\textcolor{blue}{\sc D.~Haile and J.~Osterburg}} (Eds.), {\em
  Azumaya Algebras, Actions, And Modules}, Providence: Amer. Math. Soc., pp.\
  203--214.

\bibitem[\protect\citeauthoryear{Samuel}{\textcolor{blue}{\sc
  Samuel}}{1972}]{samuel}
{\textcolor{blue}{\sc Samuel, P.}} (1972).
\newblock {\em Algebraic theory of numbers}.
\newblock Paris: Hermann.

\bibitem[\protect\citeauthoryear{Schafer}{\textcolor{blue}{\sc
  Schafer}}{1954}]{scha54}
{\textcolor{blue}{\sc Schafer, R.~D.}} (1954).
\newblock On the algebras formed by the {C}ayley-{D}ickson process.
\newblock {\em Amer. J. Math.\/}~{\bf 76}, 435--446.

\bibitem[\protect\citeauthoryear{Schafer}{\textcolor{blue}{\sc
  Schafer}}{1966}]{schafer}
{\textcolor{blue}{\sc Schafer, R.~D.}} (1966).
\newblock {\em An Introduction to Nonassociative Algebras}.
\newblock Academic Press.

\bibitem[\protect\citeauthoryear{Spindler}{\textcolor{blue}{\sc
  Spindler}}{2018}]{spindler}
{\textcolor{blue}{\sc Spindler, K.}} (2018).
\newblock {\em Abstract Algebra with Applications: Vector Spaces and Groups},
  Vol.~1.
\newblock Boca Raton, FL: CRC Press.

\bibitem[\protect\citeauthoryear{Taber}{\textcolor{blue}{\sc
  Taber}}{1904}]{tab1904}
{\textcolor{blue}{\sc Taber, H.}} (1904).
\newblock On hypercomplex number systems.
\newblock {\em Trans. Amer. Math. Soc.\/}~{\bf 5}, 509--548.

\bibitem[\protect\citeauthoryear{Takhtajan}{\textcolor{blue}{\sc
  Takhtajan}}{1994}]{tak3}
{\textcolor{blue}{\sc Takhtajan, L.}} (1994).
\newblock On foundation of the generalized {N}ambu mechanics.
\newblock {\em Commun. Math. Phys.\/}~{\bf 160}, 295--315.

\bibitem[\protect\citeauthoryear{Wedderburn}{\textcolor{blue}{\sc
  Wedderburn}}{1908}]{wed1908}
{\textcolor{blue}{\sc Wedderburn, J. H.~M.}} (1908).
\newblock On hypercomplex numbers.
\newblock {\em Proc. London Math. Soc. Second Series\/}~{\bf 6} (1), 77--118.

\bibitem[\protect\citeauthoryear{Yaglom}{\textcolor{blue}{\sc
  Yaglom}}{1968}]{yaglom}
{\textcolor{blue}{\sc Yaglom, I.~M.}} (1968).
\newblock {\em Complex Numbers In Geometry}.
\newblock New York-London: Academic Press.

\end{thebibliography}
\end{document}